\documentclass[11pt, a4paper]{amsart}
\usepackage[utf8]{inputenc}

\usepackage{amsmath, amsfonts, amssymb, amsthm}

\usepackage{graphicx}
\usepackage[inline]{enumitem}
\usepackage{hyperref}
\usepackage{stmaryrd} 

\usepackage{comment}
\usepackage{multicol}

\theoremstyle{definition}
\newtheorem{definition}{Definition}[subsection]
\newtheorem{remark}[definition]{Remark}
\theoremstyle{plain}
\newtheorem{proposition}[definition]{Proposition}
\newtheorem{propdef}[definition]{Proposition-Definition}
\newtheorem{lemma}[definition]{Lemma}
\newtheorem{theorem}[definition]{Theorem}
\newtheorem{corollary}[definition]{Corollary}
\newtheorem{thm}{Result}

\usepackage{tikz, float, tikz-cd}
	\usetikzlibrary{arrows}

\numberwithin{equation}{section}
\numberwithin{figure}{section}

\newcommand{\N}{\mathbb{N}}

\newcommand{\Q}{\mathbb{Q}}

\newcommand{\C}{\mathbb{C}}
\newcommand{\K}{\mathbf{k}} 
\newcommand{\G}{\mathcal{G}} 
\newcommand{\V}{\mathcal{V}} 
\newcommand{\W}{\mathcal{W}} 
\newcommand{\M}{\mathcal{M}} 
\newcommand{\DR}{\mathrm{DR}} 
\newcommand{\B}{\mathrm{B}} 
\newcommand{\KX}{\mathbf{k}\langle\langle X \rangle \rangle}
\newcommand{\KY}{\mathbf{k}\langle\langle Y \rangle \rangle}

\newcommand{\DMR}{\mathsf{DMR}} 
\newcommand{\Stab}{\mathsf{Stab}} 
\newcommand{\iso}{\mathrm{iso}}
\newcommand{\aut}{\mathrm{aut}}
\newcommand{\Aut}{\mathrm{Aut}}

\newcommand{\alg}{\mathrm{alg}}
\newcommand{\Mod}{\mathrm{mod}}
\newcommand{\Emb}{\mathrm{Emb}}
\newcommand{\comp}{\mathrm{comp}}
\newcommand{\Ad}{\mathrm{Ad}}
\newcommand{\q}{\mathbf{q}} 

\author{YADDADEN Khalef}
\address{Institut de Recherche Mathématique Avancée, UMR 7501, Université de Strasbourg, \, 7 rue René Descartes, 67000 Strasbourg, France}
\email{kyaddaden@math.unistra.fr}

\title[The double shuffle torsor in terms of Betti and de Rham coproducts]{The cyclotomic double shuffle torsor in terms of Betti and de Rham coproducts}

\begin{document}
    \maketitle
    \hypersetup{pdfborder={0 0 0}} 
    \begin{abstract}
        In order to describe the double shuffle and regularization relations between multiple polylogarithm values at $N$\textsuperscript{th} roots of unity, Racinet attached to each finite cyclic group $G$ of order $N$ and each group embedding $\iota : G \to \mathbb{C}^{\times}$, a $\mathbb{Q}$-scheme $\mathsf{DMR}^{\iota}$ which associates to each commutative $\Q$-algebra $\K$, a set $\DMR^{\iota}(\K)$ that can be decomposed as a disjoint union of sets $\DMR^{\iota}_{\lambda}(\K)$ with $\lambda \in \K$. He also exhibited a $\Q$-group scheme $\mathsf{DMR}_0^G$ and showed, for any commutative $\Q$-algebra $\K$ and any $\lambda \in \K^{\times}$, that $\mathsf{DMR}^{\iota}_{\lambda}(\K)$ is a torsor for the action of $\mathsf{DMR}_0^G(\K)$.
        Then, Enriquez and Furusho showed for $N=1$ that a subscheme $\mathsf{DMR}^{\iota}_{\times}$ of $\mathsf{DMR}^{\iota}$ is a torsor of isomorphisms relating ``de Rham'' and ``Betti'' objects. In previous work, we reformulated Racinet's construction in terms of crossed products and identified Racinet's coproduct with a coproduct $\widehat{\Delta}^{\mathcal{M}, \mathrm{DR}}_G$ defined on a module $\widehat{\mathcal{M}}_G^{\mathrm{DR}}$ over an algebra $\widehat{\mathcal{W}}_G^{\mathrm{DR}}$, which is equipped with its own coproduct $\widehat{\Delta}^{\mathcal{W}, \mathrm{DR}}_G$.
        In this paper, we define the main ingredients for a generalization of Enriquez and Furusho's result to any $N \geq 1$: we exhibit a module $\widehat{\mathcal{M}}_N^{\mathrm{B}}$ over an algebra $\widehat{\mathcal{W}}_N^{\mathrm{B}}$ and we prove the existence of two compatible coproducts $\widehat{\Delta}^{\mathcal{W}, \mathrm{B}}_N$ and $\widehat{\Delta}^{\mathcal{M}, \mathrm{B}}_N$ on $\widehat{\mathcal{W}}_N^{\mathrm{B}}$ and $\widehat{\mathcal{M}}_N^{\mathrm{B}}$ respectively such that $\mathsf{DMR}^{\iota}_{\times}$ is contained in the torsor of isomorphisms relating $\widehat{\Delta}^{\mathcal{W}, \mathrm{B}}_N$ (resp. $\widehat{\Delta}^{\mathcal{M}, \mathrm{B}}_N$) to $\widehat{\Delta}^{\mathcal{W}, \mathrm{DR}}_G$ (resp. $\widehat{\Delta}^{\mathcal{M}, \mathrm{DR}}_G$).
    \end{abstract}
	
    {\footnotesize \tableofcontents}
	
    \section*{Introduction}
A \emph{multiple L-value} (MLV in short) is a complex number defined by the following series
\begin{equation*}
    L_{(k_1, \dots, k_r)}(z_1, \dots, z_r) := \sum_{0 < m_1 < \dots < m_r} \frac{z_1^{k_1} \cdots z_r^{k_r}}{m_1^{k_1} \cdots m_r^{k_r}}
\end{equation*}
where $r, k_1, \dots, k_r \in \N^{\ast}$ and $z_1, \dots, z_r$ in $\mu_N$ the group of $N$\textsuperscript{th} roots of unity in $\C$, where $N$ is an integer $\geq 1$. This series converges if and only if $(k_r, z_r) \neq (1, 1)$.
These values satisfy a set of algebraic relations; our main interest here are the \emph{double shuffle and regularisation} ones.

Understanding the double shuffle and regularisation relations has been greatly improved thanks to Racinet's work \cite{Rac}. He generalises the group $\mu_N$ to a finite cyclic group $G$ and attaches to each pair $(G, \iota)$ of a finite cyclic group $G$ and a group injection $\iota : G \to \C^{\times}$, a $\Q$-scheme $\DMR^{\iota}$ which associates to each commutative $\Q$-algebra $\K$, a set $\DMR^{\iota}(\K)$\index{$\DMR^{\iota}(\K)$} that can be decomposed as a disjoint union of sets $\DMR^{\iota}_{\lambda}(\K)$ for $\lambda \in \K$ (see \cite[Definition 3.2.1]{Rac}).
The double shuffle and regularisation relations on MLVs are then encoded in the statement that a suitable generating series of these values belongs to the set $\DMR^{\iota_{can}}_{i 2\pi}(\C)$ where $\iota_{can} : G=\mu_N \to \C^{\star}$ is the canonical embedding.
Racinet also proved that for any pair $(G, \iota)$, the set $\DMR^{\iota}_0(\K)$ equipped with the ``twisted Magnus'' product (see \eqref{circledast}) is a group that is independent of the choice of $\iota$. It is therefore denoted $\DMR_0^G(\K)$.

The main result of Racinet in \cite[Theorem I]{Rac} is that, for each pair $(\lambda, \iota)$ where $\lambda \in \K^{\times}$ and $\iota : G \hookrightarrow \C^{\times}$, the set $\mathsf{DMR}^{\iota}_{\lambda}(\K)$ is equipped with a torsor structure for the action of the group $(\DMR_0^G(\K), \circledast)$. For any $\iota : G \hookrightarrow \C^{\times}$, this yields a torsor structure on the set $\displaystyle \mathsf{DMR}^{\iota}_{\times}(\K) := \bigsqcup_{\lambda \in \K^{\times}} \mathsf{DMR}^{\iota}_{\lambda}(\K)$ for the action of a semidirect product group $\K^{\times} \ltimes \DMR_0^G(\K)$ (see Proposition \ref{DMR_torsor}). This gives rise to a torsor structure on $\displaystyle \DMR_{\times}(\K) := \bigsqcup_{\iota} \mathsf{DMR}^{\iota}_{\times}(\K)$ (where $\iota$ runs over all group embeddings from $G$ to $\C^{\times}$) for the action of the semidirect product group $(\Aut(G) \times \K^{\times}) \ltimes \DMR_0^G(\K)$ (see Proposition \ref{DMR_x_torsor}).

On the other hand, we introduced in \cite{Yad} a \emph{crossed product} formalism and showed that Racinet's objects can be expressed within it. This constitutes the ``de Rham'' side of the double shuffle theory. In this framework, the crossed product algebra is identified to a topological $\K$-algebra $\widehat{\V}_G^{\DR}$ defined by a presentation with generators and relations (see Proposition \ref{KXGtoVG}). Next, Racinet's objects are given in the form of a subalgebra $\widehat{\W}_G^{\DR}$ of $\widehat{\V}_G^{\DR}$ and a quotient module $\widehat{\M}_G^{\DR}$ of the left regular $\widehat{\V}_G^{\DR}$-module. The algebra $\widehat{\W}_G^{\DR}$ is equipped with a Hopf algebra coproduct $\widehat{\Delta}^{\W, \DR}_G$ and the module $\widehat{\M}_G^{\DR}$ is equipped with a compatible coalgebra coproduct $\widehat{\Delta}^{\M, \DR}_G$.

Following the stabilizer interpretation of $\DMR_0^G(\K)$ given by Enriquez and Furusho in \cite{EF0}, we defined two stabilizers $\Stab(\widehat{\Delta}^{\M, \DR}_{G})(\K)$ and $\Stab(\widehat{\Delta}^{\W, \DR}_{G})(\K)$ for the action of grouplike elements equipped with the twisted Magnus product and therefore obtained the following chain of inclusions
\[
    \DMR_0^G(\K) \subset \Stab(\widehat{\Delta}^{\M, \DR}_{G})(\K) \subset \Stab(\widehat{\Delta}^{\W, \DR}_{G})(\K),
\]
which is a generalisation of the $G=\{1\}$ result of {\cite[Theorem 3.1]{EF2}}. This enables us to obtain the following semidirect product group chain of inclusions 
\begin{thm}[Corollary \ref{DMR_sub_StabM_sub_StabW}]
    \begin{eqnarray*}
        (\Aut(G) \times \K^{\times}) \ltimes \DMR_0^G(\K) \subset & (\Aut(G) \times \K^{\times}) \ltimes \Stab(\widehat{\Delta}^{\M, \DR}_{G})(\K) \\
        & \cap \\
        & (\Aut(G) \times \K^{\times}) \ltimes \Stab(\widehat{\Delta}^{\W, \DR}_{G})(\K)
    \end{eqnarray*}
\end{thm}

For $G=\{1\}$ Enriquez and Furusho introduced a ``Betti'' formalism of the double shuffle theory in \cite{EF1}. It is based on the filtered algebra $\V^{\B}$, which denotes the group algebra over $\K$ of the free group of rank $2$ denoted $F_2$ with generators $X_0$ and $X_1$ and equipped with the filtration induced by the augmentation ideal. The completion $\widehat{\V}^{\B}$ is a topological $\K$-algebra. Next, we have a Hopf algebra $(\widehat{\W}^{\B}, \widehat{\Delta}^{\W, \B})$ which consists of a subalgebra $\widehat{\W}^{\B}$ of $\widehat{\V}^{\B}$ linearly generated by $1 \in \widehat{\V}^{\B}$ and the left ideal generated by $X_1 - 1$. It is presented as an algebra with generators $X_1$, $X_1^{-1}$, $Y_n^+ = - (X_0 - 1)^{n-1} X_0 (X_1 - 1)$ and $Y_n^- = - (X_0^{-1} - 1)^{n-1} X_0^{-1} (X_1^{-1} - 1)$ for $n \in \N^{\ast}$, with relation $X_1 X_1^{-1} = X_{-1} X_1 = 1$. In addition, we have a Hopf algebra coproduct $\widehat{\Delta}^{\W, \B} : \widehat{\W}^{\B} \to (\widehat{\W}^{\B})^{\hat{\otimes} 2}$ given by
\[
    \mbox{\small$\widehat{\Delta}^{\W, \B}(X_1^{\pm 1}) = X_1^{\pm 1} \otimes X_1^{\pm 1}$} \text{ and for } n \in \N^{\ast}, \mbox{\small$\displaystyle\widehat{\Delta}^{\W, \B}(Y_n^{\pm}) = Y_n^{\pm} \otimes 1 + 1 \otimes Y_n^{\pm} \hspace{-0.1cm} + \hspace{-0.2cm}\sum_{\substack{k,l \in \N^{\ast} \\ k+l=n}} \hspace{-0.2cm} Y_k^{\pm} \otimes Y_l^{\pm}.$}
\]
Finally, we have a coalgebra $(\widehat{\M}^{\B}, \widehat{\Delta}^{\M, \B})$ which consists of a quotient module $\widehat{\M}^{\B}=\widehat{\V}^{\B}/\widehat{\V}^{\B}(X_0 - 1)$ isomorphic to $\widehat{\W}^{\B}$, as a $\K$-module (see \cite[(2.1.1)]{EF1}) together with a coalgebra coproduct $\widehat{\Delta}^{\M, \B}$ compatible with the coproduct $\widehat{\Delta}^{\W, \B}$.

In §\ref{Betti_side}, we construct analogues of the Hopf algebra $(\widehat{\W}^{\B}, \widehat{\Delta}^{\W, \B})$ and of the module-coalgebra $(\widehat{\M}^{\B}, \widehat{\Delta}^{\M, \B})$, for a finite cyclic group $G$ of order $N$. It is based on the filtered algebra $\V^{\B}_N$, which denotes the group algebra $\K F_2$ equipped with the filtration induced by the ideal $\ker(\K F_2 \to \K\mu_N)$; where $\K F_2 \to \K\mu_N$ is the algebra morphism induced by the group morphism $F_2 \to \mu_N$ given by $X_0 \mapsto e^{\frac{i 2\pi}{N}}$ and $X_1 \mapsto 1$. Its completion is the inverse limit of the projective system induced by the filtration and is denoted $\widehat{\V}^{\B}_N$. It is a topological algebra isomorphic to $\widehat{\V}^{\DR}_G$ (see Proposition-Definition \ref{isoViota}).
Next, we define a filtered algebra $\W^{\B}_N$ which is linearly generated by $1 \in \V^{\B}_N$ and the left ideal generated by $X_1 - 1$ and whose filtration is induced by that of $\V^{\B}_N$. Its completion $\widehat{\W}^{\B}_N$ is isomorphic to $\widehat{\W}^{\DR}_G$ (see Proposition-Definition \ref{isoW}). We also define a filtered module $\M^{\B}_N$ which consists of the quotient module $\K F_2 / \K F_2 (X_0 - 1)$ and whose filtration is induced by that of $\V^{\B}_N$. Its completion $\widehat{\M}^{\B}_N$ is isomorphic to $\widehat{\M}^{\DR}_G$ (see Proposition-Definition \ref{isoMiota}). We then have compatible Hopf algebra and coalgebra structures on $\widehat{\W}^{\B}_N$ and $\widehat{\M}^{\B}_N$ respectively thanks to the following result:
\begin{thm}[Theorem \ref{Delta_B_N} and Corollary \ref{WM_categories}]
    There exists a topological $\K$-algebra morphism $\widehat{\Delta}^{\W, \B}_N : \widehat{\W}^{\B}_N \to (\widehat{\W}^{\B}_N)^{\hat{\otimes} 2}$ and a topological $\K$-module morphism $\widehat{\Delta}^{\M, \B}_N : \widehat{\M}^{\B}_N \to (\widehat{\M}^{\B}_N)^{\hat{\otimes} 2}$ that endows $\widehat{\W}^{\B}_N$ and $\widehat{\M}^{\B}_N$ respectively with compatible Hopf algebra and coalgebra structures.
\end{thm}

Finally, in §\ref{torsorBDR}, we deduce the following result:
\begin{thm}[Theorem \ref{inclusion_torsor_stabs}]
    $\mathsf{DMR}_{\times}$ is contained in the torsor of isomorphisms relating $\widehat{\Delta}^{\mathcal{W}, \mathrm{B}}_N$ (resp. $\widehat{\Delta}^{\mathcal{M}, \mathrm{B}}_N$) to $\widehat{\Delta}^{\mathcal{W}, \mathrm{DR}}_G$ (resp. $\widehat{\Delta}^{\mathcal{M}, \mathrm{DR}}_G$).
\end{thm}
    \subsubsection*{\textbf{Acknowledgements}} The author is grateful to Benjamin Enriquez for the helpful discussions, ideas and careful reading.
    \subsubsection*{Notation} Throughout this paper, \index{$G$}$G$ is a (multiplicative) finite cyclic group of order $N$ and \index{$\K$}$\K$ is a commutative $\Q$-algebra. For a $\K$-algebra $A$, an element $x \in A$ and a left $A$-module $M$ we consider:
    \begin{itemize}[leftmargin=*]
        \item $\ell_{x} : M \to M$ to be the $\K$-module endomorphism defined by $m \mapsto xm$ and if $x$ is invertible, then $\ell_x$ is an automorphism.
        \item $r_{x} : A \to A$ to be the $\K$-module endomorphism defined by $a \mapsto ax$ and if $x$ is invertible, then $r_x$ is an automorphism.
        \item $\Ad_x : A \to A$ to be the $\K$-algebra automorphism defined by $a \mapsto x a x^{-1}$ with $x \in A^{\times}$.
    \end{itemize}
    \setcounter{section}{-1}
\section{Some categories of algebra-modules}
First, let us recall various categories introduced in \cite{EF4} that will be used throughout this paper:
\begin{itemize}[leftmargin=*]
    \item \index{$\K\text{-}\Mod$}$\K\text{-}\Mod$ is the category of $\K$-modules.
    \item \index{$\K\text{-}\alg$}$\K\text{-}\alg$ is the category of $\K$-algebras.
    \item \index{$\K\text{-}\alg\text{-}\Mod$} $\K\text{-}\alg\text{-}\Mod$ is the category of pairs $(A, M)$ where $A$ is a $\K$-algebra and $M$ is a left $A$-module.
    \item \index{$\K\text{-}\mathrm{coalg}$}$\K\text{-}\mathrm{coalg}$ is the category of coassociative cocommutative coalgebras over $\K$.
    \item \index{$\K\text{-}\mathrm{Hopf}$}$\K\text{-}\mathrm{Hopf}$ is the category of Hopf algebras over $\K$.
    \item \index{$\K\text{-}\mathrm{HAMC}$} $\K\text{-}\mathrm{HAMC}$ is the category of Hopf-Algebra-Module-Coalgebras where objects are pairs $\left((A, \Delta^A), (M, \Delta^M)\right)$ where $(A, \Delta^A)$ is a Hopf algebra and $(M, \Delta^M)$ is a coalgebra such that
    \begin{itemize}[label=$\blacktriangleright$, leftmargin=*]
        \item The pair $(A, M)$ is an algebra-module.
        \item For $(a, m) \in A \times M$, we have $\Delta^M(am) = \Delta^A(a)\Delta^M(m)$.
    \end{itemize}
    \item $\K{\text -}\Mod_{\mathrm{top}}$ is the category of topological $\K$-modules with objects pairs $(M, (\mathcal{F}^iM)_{i \in \N})$, where $M$ is a $\K$-module and $(\mathcal{F}^iM)_{i \in \N}$ is a decreasing filtration of $M$ such that the map $\displaystyle M \to \lim_{\longleftarrow} M/F^iM$ is a $\K$-module isomorphism, i.e. $M$ is complete and separated for the topology defined by the filtration $(\mathcal{F}^iM)_{i \in \N}$. It is equipped with a tensor product $\hat{\otimes}$, with respect to which it is a symmetric tensor category.
    \item $\K{\text -}\alg_{\mathrm{top}}$ is the category of topological $\K$-algebras. i. e. algebras in the category  $\K{\text -}\Mod_{\mathrm{top}}$ in the sense of \cite{McL}.
    \item $\K\text{-}\alg\text{-}\Mod_{\mathrm{top}}$ is the category of topological $\K$-algebra-modules. i. e. $\K$-algebra-modules in the category $\K{\text -}\Mod_{\mathrm{top}}$ in the sens of \cite{McL}.
    \item $\K{\text -}\mathrm{coalg}_{\mathrm{top}}$ is the category of topological $\K$-coalgebras. i. e. coalgebras in the category $\K{\text -}\mathrm{mod}_{\mathrm{top}}$ in the sens of \cite{McL}.
    \item $\K{\text -}\mathrm{Hopf}_{\mathrm{top}}$ is the category of topological $\K$-Hopf algebras. i. e. Hopf algebras in the category $\K{\text -}\mathrm{mod}_{\mathrm{top}}$ in the sens of \cite{McL}.
    \item $\K{\text -}\mathrm{HAMC}_{\mathrm{top}}$ is the category of topological Hopf-Algebra-Module-Coalgebras. i. e. Hopf-Algebra-Module-Coalgebras in the category $\K{\text -}\mathrm{mod}_{\mathrm{top}}$ in the sens of \cite{McL}.
\end{itemize}

Finally, let $\mathcal{C}$ be a symmetric tensor category and $O$ an object of $\mathcal{C}$. We define $\mathrm{Cop}_{\mathcal{C}}(O)$ to be the set of morphisms $D : O \to O^{\hat{\otimes} 2}$.
One checks that the group $\Aut_{\mathcal{C}}(O)$ acts on $\mathrm{Cop}_{\mathcal{C}}(O)$ by
\begin{equation}
    \label{from_Aut_to_Cop}
    \alpha \cdot D := \alpha^{\otimes 2} \circ D \circ \alpha^{-1},
\end{equation}
with $\alpha \in \Aut_{\mathcal{C}}(O)$ and $D \in \mathrm{Cop}_{\mathcal{C}}(O)$.
    \section{The double shuffle torsors}
In this section, we recall the various double shuffle torsors arising from Racinet's work in \cite{Rac}. In §\ref{preliminaries}, we recall the basic framework of Racinet's formalism. Namely, two Hopf algebras $(\KX, \widehat{\Delta})$ and $(\KY, \widehat{\Delta}_{\star}^{\alg})$, a coalgebra $(\KX/\KX x_0, \widehat{\Delta}_{\star}^{\Mod})$ and a group $(\G(\KX), \circledast)$. Additionally, we also recall the basic material of the crossed product formalism introduced in \cite{Yad} which consists of an algebra $\widehat{\V}^{\DR}_G$ and its relation with a Hopf algebra $(\widehat{\W}^{\DR}_G, \widehat{\Delta}^{\W, \DR}_G)$ isomorphic to the Hopf algebra $(\KY, \widehat{\Delta}_{\star}^{\alg})$ and a coalgebra $(\widehat{\M}^{\DR}_G, \widehat{\Delta}^{\M, \DR}_G)$ isomorphic to the coalgebra $(\KX/\KX x_0, \widehat{\Delta}_{\star}^{\Mod})$. In §\ref{DMRiotalambda}, we introduce the double shuffle set $\DMR_{\lambda}^{\iota}(\K)$ for $\lambda \in \K$ and $\iota : G \to \C^{\times}$ an injective group morphism, which is a torsor over the double shuffle group $\DMR_0^G(\K)$, a subgroup of $\G(\KX), \circledast)$ (\cite{Rac}). In §\ref{DMRxiota}, we define a set $\displaystyle \DMR_{\times}^{\iota}(\K) = \bigsqcup_{\lambda \in \K^{\times}} \DMR_{\lambda}^{\iota}(\K)$ and show that it is a torsor for a group given by a semidirect product arising from an action of $\K^{\times}$. Finally, in §\ref{DMRx}, we define a set $\displaystyle \DMR_{\times}(\K) = \bigsqcup_{\iota} \DMR_{\times}^{\iota}(\K)$ where $\iota$ runs over all injections $G \to \C^{\times}$ and show that it is a torsor for a group given by semidirect product arising from an action of $\Aut(G) \times \K^{\times}$. 

\subsection{Preliminaries} \label{preliminaries}

\subsubsection{Basic objects of Racinet's formalism}
Let \index{$\KX$}$\KX$ be the free noncommutative associative series algebra with unit over the alphabet \index{$X$}$X := \{x_0\} \sqcup \{x_g | g \in G\}$\index{$x_0$}\index{$x_g$}. It is complete graded with $\deg(x_0) = \deg(x_g) = 1$ for $g \in G$.
This algebra provides an object in $\K{\text-}\mathrm{Hopf}_{\mathrm{top}}$ when equipped with the coproduct \index{$\widehat{\Delta}$}$\widehat{\Delta} : \KX \to {\KX}^{\hat{\otimes} 2}$, which is the unique topological $\K$-algebra morphism given by $\widehat{\Delta}(x_g) = x_g \otimes 1 + 1 \otimes x_g$, for any $g \in G \sqcup \{0\}$ (\cite[§2.2.3]{Rac}).
Let then \index{$\G(\KX)$}$\G(\KX)$ be the set of grouplike elements of $\KX$ for the coproduct $\widehat{\Delta}$, i.e. the set
\begin{equation}
    \label{GKX}
    \G(\KX) = \{ \Psi \in \KX^{\times} \, | \, \widehat{\Delta}(\Psi) = \Psi \otimes \Psi \},
\end{equation}
where $\KX^{\times}$ denotes the set of invertible elements of $\KX$. Since $(\KX, \widehat{\Delta})$ is a Hopf algebra, $\G(\KX)$ is a group for the product of $\KX$.

The group $G$ acts on $\KX$ by topological $\K$-algebra automorphisms by $g \mapsto t_g$, where for any $g \in G$, the topological $\K$-algebra automorphism \index{$t_g$}$t_g$ is given by $t_g(x_0) = x_0, \, t_g(x_h) = x_{gh}$ for $h \in G$ (\cite[§3.1.1]{Rac}).
For any $g \in G$, we have
\begin{equation}
    \widehat{\Delta} \circ t_g = t_g^{\otimes 2} \circ \widehat{\Delta},
\end{equation}
this can be verified by checking this identity on generators since both sides are given as a composition of $\K$-algebra morphisms.
As a consequence, for any $g \in G$, the $\K$-algebra automorphism $t_g : \KX \to \KX$ restricts to a group automorphism of $\G(\KX)$.

Let \index{$\q$}$\q$ be the $\K$-module automorphism of $\KX$ given on the topological $\K$-module basis $\left(x_0^{n_1}x_{g_1}x_0^{n_2}x_{g_2} \cdots x_0^{n_r}x_{g_r}x_0^{n_{r+1}}\right)_{\substack{r, n_1, \dots, n_{r+1} \in \N \\ g_1, \dots, g_r \in G}}$ of $\KX$ (\cite[§2.2.7]{Rac}) by
\begin{equation*}
    \q(x_0^{n_1}x_{g_1}x_0^{n_2}x_{g_2} \cdots x_0^{n_r}x_{g_r} x_0^{n_{r+1}}) = x_0^{n_1}x_{g_1}x_0^{n_2}x_{g_2g_1^{-1}} \cdots x_0^{n_r}x_{g_rg_{r-1}^{-1}}x_0^{n_{r+1}}.
\end{equation*}

For $(n, g) \in \N_{> 0} \times G$, set \index{$y_{n, g}$}$y_{n, g} := x_0^{n-1}x_g$. Let \index{$Y$}$Y := \{ y_{n, g} | (n, g) \in \N_{> 0} \times G \}$. We define \index{$\KY$}$\KY$ to be the topologically free $\K$-algebra over $Y$, where for every $(n, g) \in \N_{> 0} \times G$, the element $y_{n, g}$ is of degree $n$. One shows that $\KY$ is equal to the $\K$-subalgebra $\displaystyle \K \oplus \bigoplus_{g \in G} \KX x_g$ of $\KX$ (\cite[§2.2.5]{Rac} and \cite[§2.2]{EF0}).

Let \index{$\widehat{\Delta}^{\alg}_{\star}$}$\widehat{\Delta}^{\alg}_{\star} : \KY \to \left(\KY \right)^{\hat{\otimes} 2}$ be the unique topological $\K$-algebra morphism such that for any $(n, g) \in \N_{> 0} \times G$
\begin{equation}
    \widehat{\Delta}^{\alg}_{\star}(y_{n,g}) = y_{n,g} \otimes 1 + 1 \otimes y_{n,g} + \sum_{\substack{k=1 \\ h \in G}}^{n-1} y_{k,h} \otimes y_{n-k,gh^{-1}}.
\end{equation}
The element $\widehat{\Delta}^{\alg}_{\star} \in \mathrm{Cop}_{\K{\text -}\alg_{\mathrm{top}}}(\KY)$ is called the \emph{harmonic coproduct} (\cite[§2.3.1]{Rac}) and the pair $(\KY, \widehat{\Delta}^{\alg}_{\star})$ is an object of $\K{\text-}\mathrm{Hopf}_{\mathrm{top}}$.

Let us consider the topological $\K$-module quotient \index{$\KX / \KX x_0$}$\KX / \KX x_0$. The pair $(\KY, \KX/\KX x_0)$ is an object of the category $\K{\text -}\alg{\text -}\Mod_{\mathrm{top}}$.
The restriction to $\KY$ of the projection morphism \index{$\pi_Y$}$\pi_Y : \KX \to \KX / \KX x_0$ is an isomorphism, therefore $\KX / \KX x_0$ is free of rank $1$ over $\KY$.
It follows that there is a topological $\K$-module morphism \index{$\widehat{\Delta}_{\star}^{\Mod}$}
$\widehat{\Delta}_{\star}^{\Mod} \in \mathrm{Cop}_{\K{\text -}\Mod_{\mathrm{top}}}(\KX/\KX x_0)$ uniquely defined by the condition that the diagram
\begin{equation}
    \label{harmonic_coproduct_M}
    \begin{tikzcd}
        \KY \ar["\widehat{\Delta}_{\star}^{\alg}"]{rr} \ar["\pi_Y"']{d} && \KY^{\hat{\otimes} 2} \ar["\pi_Y^{\otimes 2}"]{d} \\
        \KX / \KX x_0 \ar["\widehat{\Delta}_{\star}^{\Mod}"]{rr} && \left(\KX / \KX x_0\right)^{\hat{\otimes} 2}
    \end{tikzcd}
\end{equation}
commutes.
The pair $(\widehat{\Delta}^{\alg}_{\star}, \widehat{\Delta}^{\Mod}_{\star})$ is an element of $\mathrm{Cop}_{\K{\text -}\alg{\text -}\Mod_{\mathrm{top}}}(\KY, \KX/\KX x_0)$. The pair $(\KX / \KX x_0, \widehat{\Delta}^{\Mod}_{\star})$ is an object of $\K{\text -}\mathrm{coalg}_{\mathrm{top}}$ and the pair $\big((\KY, \widehat{\Delta}_{\star}^{\alg})$, $(\KX / \KX x_0, \widehat{\Delta}_{\star}^{\Mod})\big)$ is an object of $\K{\text -}\mathrm{HAMC}_{\mathrm{top}}$.

\subsubsection{Basic objects of the crossed product formalism}
Let \index{$\widehat{\V}^{\DR}_G$}$\widehat{\V}^{\DR}_G$ be the complete graded algebra generated by \index{$e_0$}\index{$e_1$}\index{$g$}\(\{e_0, e_1\} \sqcup G\) where $e_0$ and $e_1$ are of degree $1$ and elements $g \in G$ are of degree $0$ satisfying the relations:
\setlength\multicolsep{5pt}
\begin{multicols}{3}
\begin{enumerate}[leftmargin=*, label=(\roman*)]
    \item $g \times h = g h$;
    \item $1 = 1_G$;
    \item $g \times e_0 = e_0 \times g$;
\end{enumerate}
\end{multicols}
\setlength\multicolsep{5pt}
\noindent for any $g, h \in G$; where “$\times$” is the algebra multiplication\footnote{which we will no longer denote if there is no risk of ambiguity.}(\cite[§2.1.1]{Yad}).

Recall that the map $g \mapsto t_g$ defines an action of $G$ on $\KX$ by $\K$-algebra automorphisms. One then considers the crossed product algebra \index{$\KX \rtimes G$}$\KX \rtimes G$ for this action, which is the $\K$-module $\KX \otimes \K G$ equipped with the product given for any $a, b \in \KX$ and any $g, h \in G$ by (\cite[Chapter 3, Page 180, Exercise 11]{Bou07})
\begin{equation}
    (a \otimes g) \ast (b \otimes h) = a \, t_g(b) \otimes gh.
\end{equation}

\begin{proposition}[{\cite[Proposition 2.1.3]{Yad}}]
    \label{KXGtoVG}
    There is a unique $\K$-algebra isomorphism \index{$\beta$}$\beta : \KX \rtimes G \to \widehat{\V}^{\DR}_G$ such that $x_0 \otimes 1 \mapsto e_0$ and for $g \in G$, $x_g \otimes 1 \mapsto -g e_1 g^{-1}$ and $1 \otimes g  \mapsto g$.
\end{proposition}

Let \index{$\widehat{\W}^{\DR}_G$}$\widehat{\W}^{\DR}_G$ be the complete graded $\K$-subalgebra of $\widehat{\V}^{\DR}_G$ given by
\begin{equation}
    \widehat{\W}^{\DR}_G := \K \oplus \widehat{\V}^{\DR}_G e_1.
\end{equation}
It is freely generated by the family \index{$Z$}\index{$z_{n, g}$}
\[
    Z=\{z_{n,g}:= -e_0^{n-1}ge_1 \, | \, (n,g) \in \N_{>0} \times G\},
\]
with $\deg(z_{n,g}) = n$ (\cite[Proposition 2.1.5.(b)]{Yad}). As a consequence, there is a unique $\K$-algebra isomorphism \index{$\varpi$}$\varpi : \KY \to \widehat{\W}^{\DR}_G$ given for $(n, g) \in \N_{>0} \times G$ by $y_{n, g} \mapsto z_{n, g}$.

One then has a unique topological $\K$-algebra morphism \index{$\widehat{\Delta}^{\W, \DR}_{G}$} $\widehat{\Delta}^{\W, \DR}_{G} : \widehat{\W}^{\DR}_G \to (\widehat{\W}^{\DR}_G)^{\hat{\otimes} 2}$ such that for any $(n, g) \in \N_{> 0} \times G$
\begin{equation}
    \label{DeltaW}
    \widehat{\Delta}^{\W, \DR}_{G}(z_{n, g}) = z_{n, g} \otimes 1 + 1 \otimes z_{n, g} + \sum_{\substack{k = 1 \\ h \in G}}^{n-1} z_{k, h} \otimes z_{n-k, gh^{-1}}. 
\end{equation}
The coproduct $\widehat{\Delta}^{\W, \DR}_{G}$ is an element of $\mathrm{Cop}_{\K{\text -}\alg_{\mathrm{top}}}(\widehat{\W}^{\DR}_G)$. The pair $(\widehat{\W}^{\DR}_G, \widehat{\Delta}^{\W, \DR}_G)$ is an object in the category $\K{\text -}\mathrm{Hopf}_{\mathrm{top}}$ and the $\K$-algebra isomorphism $\varpi : \KY \to \widehat{\W}^{\DR}_G$ is an isomorphism between the Hopf algebras $(\KY, \widehat{\Delta}_{\star}^{\alg})$ and $(\widehat{\W}^{\DR}_G, \widehat{\Delta}^{\W, \DR}_G)$.

Let \index{$\widehat{\M}^{\DR}_G$}$\widehat{\M}^{\DR}_G$ be the complete graded $\K$-module given by
\[
    \widehat{\M}^{\DR}_G := \widehat{\V}^{\DR}_G \Big/ \Big(\widehat{\V}^{\DR}_G e_0 + \sum_{g \in G} \widehat{\V}^{\DR}_G (g-1)\Big).
\]
Let \index{$1_{\DR}$}$1_{\DR}$ be the class of $1 \in \widehat{\V}^{\DR}_G$ in $\widehat{\M}^{\DR}_G$. The map \index{$- \cdot 1_{\DR}$}$- \cdot 1_{\DR} : \widehat{\V}^{\DR}_G \to \widehat{\M}^{\DR}_G$ is a surjective $\K$-module morphism with kernel $\displaystyle \widehat{\V}^{\DR}_G e_0 + \sum_{g \in G} \widehat{\V}^{\DR}_G (g-1)$. In addition, the pair $(\widehat{\V}^{\DR}_G, \widehat{\M}^{\DR}_G)$ is an object in the category $\K{\text -}\alg{\text -}\Mod_{\mathrm{top}}$.
Moreover, one deduces from \cite[Proposition 2.1.6]{Yad} that there is a unique $\K$-module isomorphism \index{$\kappa$}$\kappa : \KX / \KX x_0 \to \widehat{\M}^{\DR}_G$ determined by the commutativity of the diagram
\begin{equation}
    \label{diag_iso_MG}
    \begin{tikzcd}
        \KX \ar["\beta \circ (- \otimes 1)"]{rr} \ar["\pi_Y \circ \, \q"']{d} & & \widehat{\V}^{\DR}_G \ar["- \cdot 1_{\DR}"]{d} \\
        \KX /\KX x_0 \ar["\kappa"]{rr} & & \widehat{\M}^{\DR}_G
    \end{tikzcd}
\end{equation}

On the other hand, the following diagram
\begin{equation}
    \label{diag_projections}
    \begin{tikzcd}
        \KY \ar["\varpi"]{rr} \ar["\pi_Y"']{d} && \widehat{\W}^{\DR}_G \ar["- \cdot 1_{\DR}"]{d} \\
        \KX / \KX x_0 \ar["\kappa"]{rr} & & \widehat{\M}^{\DR}_G
    \end{tikzcd}
\end{equation}
commutes (\cite[Corollary 2.1.8]{Yad}). As a consequence, the map $- \cdot 1_{\DR} : \widehat{\W}^{\DR}_G \to \widehat{\M}^{\DR}_G$ is a $\K$-module isomorphism since all other arrows of Diagram \eqref{diag_projections} are isomorphisms.
In addition, we obtain the following result
\begin{lemma}
    \label{WM_in_algmod}
    The pair $(\widehat{\W}^{\DR}_G, \widehat{\M}^{\DR}_G)$ is an object of the category $\K{\text -}\alg{\text -}\Mod_{\mathrm{top}}$. Moreover, $\widehat{\M}^{\DR}_G$ is free $\widehat{\W}^{\DR}_G$-module of rank $1$.
\end{lemma}
\begin{proof}
    The first statement follows from the fact that $(\widehat{\W}^{\DR}_G, \widehat{\M}^{\DR}_G)$ is the pull-back of the $\K$-algebra-module $(\widehat{\V}^{\DR}_G, \widehat{\M}^{\DR}_G)$ by the $\K$-algebra morphism $\widehat{\W}^{\DR}_G \hookrightarrow \widehat{\V}^{\DR}_G$. \newline
    The second statement comes from the fact that $\KX / \KX x_0$ is a free $\KY$-module of rank $1$ thanks to the commutativity of Diagram \eqref{diag_projections}.
\end{proof}

\noindent This enables us to construct a topological $\K$-module morphism $\widehat{\Delta}^{\M, \DR}_G \in \mathrm{Cop}_{\K{\text -}\Mod_{\mathrm{top}}}(\widehat{\M}^{\DR}_G)$ uniquely defined such that the following diagram
\begin{equation}
    \label{diag_DeltaW_DeltaM}
    \begin{tikzcd}
        \widehat{\W}^{\DR}_G \ar["\widehat{\Delta}^{\W, \DR}_G"]{rrr} \ar["- \cdot 1_{\DR}"']{d} & & & (\widehat{\W}^{\DR}_G)^{\hat{\otimes} 2} \ar["- \cdot 1_{\DR}^{\otimes 2}"]{d} \\
        \widehat{\M}^{\DR}_G \ar["\widehat{\Delta}^{\M, \DR}_G"]{rrr} & & & (\widehat{\M}^{\DR}_G)^{\hat{\otimes} 2}
    \end{tikzcd}
\end{equation}
commutes, thanks to Lemma \ref{WM_in_algmod} and the free rank $1$ property of the $\widehat{\W}^{\DR}_G$-module $\widehat{\M}^{\DR}_G$.
The pair $(\widehat{\M}^{\DR}_G, \widehat{\Delta}^{\M, \DR}_G)$ is an object in the category $\K{\text -}\mathrm{coalg}_{\mathrm{top}}$ and the $\K$-module isomorphism $\kappa : \KX / \KX x_0 \to \widehat{\M}^{\DR}_G$ is an isomorphism of coalgebras $(\KX / \KX x_0, \widehat{\Delta}^{\Mod}_{\star})$ and $(\widehat{\M}^{\DR}_G, \widehat{\Delta}^{\M, \DR}_G)$.
\begin{lemma}
    \label{DeltaW_DelaM_DR_Cop}
    The pair $\left(\widehat{\Delta}_{G}^{\W, \DR}, \widehat{\Delta}_{G}^{\M, \DR}\right)$ is an element of $\mathrm{Cop}_{\K{\text -}\alg{\text -}\Mod_{\mathrm{top}}}\left(\widehat{\W}^{\DR}_G, \widehat{\M}^{\DR}_G\right)$.    
\end{lemma}
\begin{proof}
    Let $w \in \widehat{\W}^{\DR}_G$ and $m \in \widehat{\M}^{\DR}_G$. Thanks to Lemma \ref{WM_in_algmod} there is a unique $w^{\prime} \in \widehat{\W}^{\DR}_G$ such that $m = w^{\prime} \cdot 1_{\DR}$. We have
    \begin{align*}
        \widehat{\Delta}_{G}^{\M, \DR}(w \cdot m) & = \widehat{\Delta}_{G}^{\M, \DR}(w w^{\prime} \cdot 1_{\DR}) = \widehat{\Delta}_{G}^{\W, \DR}(w w^{\prime}) \cdot 1_{\DR}^{\otimes 2} = \widehat{\Delta}_{G}^{\W, \DR}(w) \widehat{\Delta}_{G}^{\W, \DR}(w^{\prime}) \cdot 1_{\DR}^{\otimes 2} \\
        & = \widehat{\Delta}_{G}^{\W, \DR}(w) \widehat{\Delta}_{G}^{\M, \DR}(w^{\prime} \cdot 1_{\DR}) = \widehat{\Delta}_{G}^{\W, \DR}(w) \widehat{\Delta}_{G}^{\M, \DR}(m),
    \end{align*}
    where the second and fourth equalities come from the commutativity of Diagram \eqref{diag_DeltaW_DeltaM}.
\end{proof}

\noindent As a consequence, the pair $\left((\widehat{\W}^{\DR}_G, \widehat{\Delta}_{G}^{\W, \DR}), (\widehat{\M}^{\DR}_G, \widehat{\Delta}_{G}^{\M, \DR})\right)$ is an object of $\K{\text -}\mathrm{HAMC}_{\mathrm{top}}$. In addition, the pair
\[
    (\varpi, \kappa) : \left((\KY, \widehat{\Delta}_{\star}^{\alg}), (\KX / \KX x_0, \widehat{\Delta}_{\star}^{\Mod})\right) \to \left((\widehat{\W}^{\DR}_G, \widehat{\Delta}_{G}^{\W, \DR}), (\widehat{\M}^{\DR}_G, \widehat{\Delta}_{G}^{\M, \DR})\right)
\]
is a morphism of objects of $\K{\text -}\mathrm{HAMC}_{\mathrm{top}}$.

\subsection{The torsor \texorpdfstring{$\DMR_{\lambda}^{\iota}(\K)$}{DMRiotalambda}} \label{DMRiotalambda}
Let us denote $\mbox{\small$\KX \to \K^{\{\text{words in } x_0, (x_g)_{g \in G}\}}, v \mapsto \big((v | w)\big)_{w}$}$ the map such that $v = \sum_{w} (v | w) w$ (the empty word is equal to $1$).

Let \index{$\Gamma$}$\Gamma : \KX \to \K[[x]]^{\times}, \Psi \mapsto \Gamma_{\Psi}$ the function\footnote{This function is related to the classical gamma function as established in \cite{Fu11}, page 344 thanks to \cite{Dri90}.} given by (\cite[(3.2.1.2)]{Rac})
\begin{equation}
    \Gamma_{\Psi}(x) := \exp\left( \sum_{n \geq 2} \frac{(-1)^{n-1}}{n} (\Psi | x_0^{n-1} x_1) x^n \right).
\end{equation}

\begin{definition}[{\cite[Definition 3.2.1]{Rac}}]
    \label{DMR_lambda}
    Let $\lambda \in \K$ and \index{$\iota$}$\iota : G \to \C^{\times}$ be a group embedding. We define \index{$\DMR_{\lambda}^{\iota}(\K)$}$\DMR_{\lambda}^{\iota}(\K)$ to be the set of $\Psi \in \G(\KX)$ such that
    \begin{enumerate}[label=(\roman*), leftmargin=*]
        \setlength\multicolsep{5pt}
        \setlength{\columnsep}{0pt}
        \begin{multicols}{2}
        \item \label{DMR_lambda_condition_i} $(\Psi | x_0) = (\Psi | x_1) = 0$;
        \item \label{DMR_lambda_condition_ii} $\widehat{\Delta}_{\star}^{\Mod}(\Psi_{\star}) = \Psi_{\star} \otimes \Psi_{\star}$;
        \item \label{DMR_lambda_condition_iii} If $|G| \in \{1, 2\}, (\Psi | x_0 x_1) = -\frac{\lambda^2}{24}$;
        \item \label{DMR_lambda_condition_iv} If $|G| \geq 3, \, \left(\Psi | x_{g_{\iota}} - x_{g_{\iota}^{-1}}\right) = \frac{|G| - 2}{2} \lambda$;
        \end{multicols}
        \item \label{DMR_lambda_condition_v} For $k \in \{1, \dots, |G|/2\}$, $\left(\Psi | x_{g_{\iota}^k} - x_{g_{\iota}^{-k}}\right) = \frac{|G| - 2k}{|G| - 2} \left(\Psi | x_{g_{\iota}} - x_{g_{\iota}^{-1}}\right)$,
    \end{enumerate}
    \setlength\multicolsep{0pt}
    where \index{$g_{\iota}$}$g_{\iota} := \iota^{-1}(e^{\frac{i2\pi}{|G|}})$ and $\Psi_{\star} := \pi_Y \circ \q \left(\Gamma^{-1}_{\Psi}(x_1) \Psi\right) \in \KX/\KX x_0$.
\end{definition}

\begin{remark} \ 
    \begin{enumerate}[label=(\roman*), leftmargin=*]
        \item Thanks to \cite[§3.2.3]{Rac}, $\DMR_{\lambda}^{\iota}(\K)$ is a non-empty set.
        \item If $|G| \in \{1, 2\}$, the embedding $\iota$ is unique.
    \end{enumerate}
\end{remark}

\begin{propdef}[{\cite[Remark 3.2.2]{Rac}}]
    For $\lambda=0$, the condition\footnote{This also holds for condition \ref{DMR_lambda_condition_v}.} \ref{DMR_lambda_condition_iv} of Definition \ref{DMR_lambda} does not depend of the choice of $\iota$. The set $\DMR_{0}^{\iota}(\K)$ is then denoted $\DMR_{0}^{G}(\K)$ instead.
\end{propdef}

\begin{proposition}
    \label{DMR_lambda_condition_ii_new}
    Condition \ref{DMR_lambda_condition_ii} of Definition \ref{DMR_lambda} is equivalent to
    \begin{equation}
        \label{DeltaM_PsiDR}
        \widehat{\Delta}^{\M, \DR}_G(\Psi^{\star}) = \Psi^{\star} \otimes \Psi^{\star},
    \end{equation}
    where
    $\Psi^{\star} := \left(\Gamma^{-1}_{\Psi}(-e_1) \beta(\Psi \otimes 1)\right) \cdot 1_{\DR} \in \widehat{\M}^{\DR}_G$.
\end{proposition}
\begin{proof}
    Thanks to Diagram \eqref{diag_iso_MG}, it follows that $\kappa(\Psi_{\star}) = \Psi^{\star}$. Equality \eqref{DeltaM_PsiDR} then follows from the fact that $\kappa : (\KX / \KX x_0, \widehat{\Delta}^{\Mod}_{\star}) \to (\widehat{\M}^{\DR}_G, \widehat{\Delta}^{\M, \DR}_G)$ is a coalgebra isomorphism. 
\end{proof}

Recall the set $ \G(\KX)$ of grouplike elements of $(\KX, \widehat{\Delta})$ given in \eqref{GKX}. In addition to its usual group structure, it is also a group for the ``twisted Magnus'' product denoted $\circledast$ and given for any $\Psi, \Phi \in \G(\KX)$ by
\begin{equation}
    \Psi \circledast \Phi := \Psi \aut_{\Psi}(\Phi),
    \label{circledast}
\end{equation}
where \index{$\aut_{\Psi}$}$\aut_{\Psi}$ is the topological $\K$-algebra automorphism of $\KX$ given by (\cite{EF0}, §4.1.3 based on \cite{Rac}, §3.1.2)
\begin{equation}
    x_0 \mapsto x_0 \qquad \text{and for } g \in G, x_g \mapsto \Ad_{t_g(\Psi^{-1})}(x_g).
\end{equation}

\begin{proposition}[{\cite[Theorem I]{Rac}}]
    \label{Racinet_main_result}
    Let $\lambda \in \K$ and $\iota : G \to \C^{\times}$ be a group embedding.
    \begin{enumerate}[label=(\roman*), leftmargin=*]
        \item \label{DMR0G} The pair $(\DMR_0^G(\K), \circledast)$ is a subgroup of $(\G(\KX), \circledast)$.
        \item \label{Racinet_torsor} The group $(\DMR_{0}^{G}(\K), \circledast)$ acts freely and transitively on $\DMR_{\lambda}^{\iota}(\K)$ by left multiplication $\circledast$.
    \end{enumerate}
\end{proposition}

\subsection{The torsor \texorpdfstring{$\DMR_{\times}^{\iota}(\K)$}{DMRxiota}} \label{DMRxiota}

\subsubsection{Action of the group \texorpdfstring{$\K^{\times}$}{kx} on \texorpdfstring{$\KX$}{KX}}
The group $\K^{\times}$ acts on $\KX$ by $\K$-algebra automorphisms by \index{$\lambda \bullet -$}
\begin{equation}
    \K^{\times} \longrightarrow \Aut_{\K{\text -}\alg}(\KX); \quad \lambda \longmapsto \lambda \bullet - : x_g \mapsto \lambda x_g, \text{ for } g \in G \sqcup \{0\}.
\end{equation}
One checks that, for any $\lambda \in \K^{\times}$, the automorphism $\lambda \bullet -$ is a Hopf algebra automorphism of $(\KX, \widehat{\Delta})$.
In addition, for any $\lambda \in \K^{\times}$ and any $g \in G$, we have
\begin{equation}
    \label{lambda_circ_t}
    (\lambda \bullet -) \circ t_g = t_g \circ (\lambda \bullet -),
\end{equation}
this can be verified by checking this identity on generators since both sides are given as a composition of $\K$-algebra morphisms.

\begin{proposition}
    \label{bullet_group_aut}
    For any $\lambda \in \K^{\times}$, the map $\lambda \bullet - : \KX \to \KX$ restricts to a group automorphism of $(\G(\KX), \circledast)$.
\end{proposition}
\noindent In order the prove this, we will need the following Lemma:
\begin{lemma}
    \label{lambda_circ_aut}
    For any $(\lambda, \Psi) \in \K^{\times} \times \G(\KX)$, we have
    \begin{equation}
        (\lambda \bullet -) \circ \aut_{\Psi} =  \aut_{\lambda \bullet \Psi} \circ (\lambda \bullet -).
    \end{equation}
\end{lemma}
\begin{proof}
    Let $(\lambda, \Psi) \in \K^{\times} \times \G(\KX)$. Since all the morphisms are algebra automorphisms of $\KX$, it is enough to check this identity on generators. We have
    \[
        (\lambda \bullet -) \circ \aut_{\Psi} (x_0) = \lambda \bullet x_0 = \lambda x_0 = \lambda \aut_{\lambda \bullet \Psi}(x_0) = \aut_{\lambda \bullet \Psi} (\lambda x_0) = \aut_{\lambda \bullet \Psi} (\lambda \bullet x_0)
    \]
    and for $g \in G$,
    \begin{align*}
        (\lambda \bullet -) \circ \aut_{\Psi} (x_g) & = (\lambda \bullet -) \circ \Ad_{t_g(\Psi^{-1})}(x_g) = \Ad_{\lambda \bullet t_g(\Psi^{-1})}(\lambda \bullet x_g) \\
        & = \Ad_{t_g(\lambda \bullet \Psi^{-1})}(\lambda \bullet x_g) = \aut_{\lambda \bullet \Psi} (\lambda \bullet x_g),
    \end{align*}
    where the third equality comes from Identity \eqref{lambda_circ_t}.
\end{proof}

\begin{proof}[Proof of Proposition \ref{bullet_group_aut}]
    Let $\lambda \in \K^{\times}$. Since $\lambda \bullet -$ is a Hopf algebra automorphism of $(\KX, \widehat{\Delta})$, it restricts to a map $\G(\KX) \to \G(\KX)$. Let $\Psi, \Phi \in \G(\KX)$. We have
    \begin{align*}
        \lambda \bullet (\Psi \circledast \Phi) & = \lambda \bullet \left(\Psi \aut_{\Psi}(\Phi)\right) = \left(\lambda \bullet \Psi\right) \left(\lambda \bullet \aut_{\Psi}(\Phi)\right) \\
        & = \left(\lambda \bullet \Psi\right) \aut_{\lambda \bullet \Psi}(\lambda \bullet \Phi) = (\lambda \bullet \Psi) \circledast (\lambda \bullet \Phi),
    \end{align*}
    where the third equality comes from Lemma \ref{lambda_circ_aut}. This proves that $\lambda \bullet -$ is a group endomorphism of $(\G(\KX), \circledast)$. Finally, one has that $(\lambda \bullet -)^{-1} = \lambda^{-1} \bullet -$ and the above computations shows that $(\lambda \bullet -)^{-1}$ is an endomorphism of $(\G(\KX), \circledast)$ thus proving the statement.
\end{proof}

\noindent Proposition \ref{bullet_group_aut} enables us to define the following: 
\begin{definition}
    We denote \index{$\K^{\times} \ltimes \G(\KX)$}$\K^{\times} \ltimes \G(\KX)$ the semi-direct product of $\K^{\times}$ and $\G(\KX)$ with respect to the action given in Proposition \ref{bullet_group_aut}. It consists of the set $\K^{\times} \times \G(\KX)$ endowed with a group law which will also be denoted $\circledast$ and we have for $(\lambda, \Psi), (\nu, \Phi) \in \K^{\times} \times \G(\KX)$, \index{$(\lambda, \Psi) \circledast (\nu, \Phi)$}
    \begin{equation}
        (\lambda, \Psi) \circledast (\nu, \Phi) := (\lambda \nu, \Psi \circledast (\lambda \bullet \Phi)).
    \end{equation}
\end{definition}

\subsubsection{Action of the group \texorpdfstring{$\K^{\times}$}{kx} on crossed product algebras and module}
The group $\K^{\times}$ acts on $\widehat{\V}^{\DR}_G$ by $\K$-algebra automorphisms by \index{$\lambda \bullet_{\V} -$}
\begin{equation}
    \label{bulletV_def}
    \mbox{\small$\K^{\times} \longrightarrow \Aut_{\K{\text -}\alg}(\widehat{\V}^{\DR}_G); \, \lambda \longmapsto \lambda \bullet_{\V} - : e_i \mapsto \lambda e_i; \, g \mapsto g$}, \text{ for } i \in \{0, 1\} \text{ and } g \in G.
\end{equation}

\begin{lemma}
    Let $\lambda \in \K^{\times}$.
    \begin{enumerate}[label=(\roman*), leftmargin=*]
        \item The following diagram
        \begin{equation}
            \label{Diag_bullets}
            \begin{tikzcd}
                \KX \ar["\lambda \bullet -"]{rr} \ar["\beta \circ (- \otimes 1)"']{d} && \KX \ar["\beta \circ (- \otimes 1)"]{d} \\
                \widehat{\V}^{\DR}_G \ar["\lambda \bullet_{\V} -"]{rr} && \widehat{\V}^{\DR}_G
            \end{tikzcd}
        \end{equation}
        commutes.
        \item For $\Psi \in \KX$, we have
        \begin{equation}
            \label{bullet_Gamma}
            \Gamma_{\lambda \bullet \Psi}(-e_1) = \lambda \bullet_{\V} \Gamma_{\Psi}(-e_1).   
        \end{equation}
    \end{enumerate}
\end{lemma}
\begin{proof} \ 
    \begin{enumerate}[label=(\roman*), leftmargin=*]
        \item Since all arrows are $\K$-algebra morphisms, one easily checks the commutativity of generators.
        \item It follows from the fact that, for $n \in \N_{>0}$, we have $(\lambda \bullet \Psi | x_0^{n-1} x_1) = \lambda^n (\Psi | x_0^{n-1} x_1)$.
    \end{enumerate}
\end{proof}

\begin{lemma}
    \label{from_lambdaV}
    Let $\lambda \in \K^{\times}$.
    \begin{enumerate}[label=(\roman*), leftmargin=*]
        \item \label{lambdaW} The $\K$-algebra automorphism $\lambda \bullet_{\V} -$ of $\widehat{\V}^{\DR}_G$ restricts to a Hopf algebra automorphism $\lambda \bullet_{\W} -$ of $(\widehat{\W}^{\DR}_G, \widehat{\Delta}^{\W, \DR}_G)$.
        \item \label{lambdaM} The $\K$-algebra automorphism $\lambda \bullet_{\V} -$ of $\widehat{\V}^{\DR}_G$ induces a coalgebra automorphism $\lambda \bullet_{\M} -$ of $(\widehat{\M}^{\DR}_G, \widehat{\Delta}^{\M, \DR}_G)$.
    \end{enumerate}
\end{lemma}
\begin{proof} \ 
    \begin{enumerate}[label=(\roman*), leftmargin=*]
        \item For $(n, g) \in \N_{>0} \times G$ we have
        \[
            \lambda \bullet_{\V} z_{n, g} = \lambda \bullet_{\V} (- e_0^{n-1} g e_1) = - \lambda^n e_0^{n-1} g e_1 = \lambda^n z_{n, g}.
        \]
        Since the algebra $\widehat{\W}^{\DR}_G$ is freely generated by $(z_{n, g})_{(n, g) \in \N_{>0} \times G}$, it follows that $\lambda \bullet_{\V}(\widehat{\W}^{\DR}_G) \subset \widehat{\W}^{\DR}_G$. Similarly, $(\lambda \bullet_{\V} -)^{-1}(\widehat{\W}^{\DR}_G) \subset \widehat{\W}^{\DR}_G$.
        Hence, $\lambda \bullet_{\V}(\widehat{\W}^{\DR}_G) = \widehat{\W}^{\DR}_G$.
        This implies that $\lambda \bullet_{\V} -$ restricts to a $\K$-algebra automorphism $\lambda \bullet_{\W} -$ of $\widehat{\W}^{\DR}_G$ and that the following diagram
        \begin{equation}
            \label{diag_lambdaW_lambdaV}
            \begin{tikzcd}
                \widehat{\W}^{\DR}_G \ar["\lambda \bullet_{\W} -"]{rrr} \ar[hook]{d} &&& \widehat{\W}^{\DR}_G \ar[hook]{d} \\
                \widehat{\V}^{\DR}_G \ar["\lambda \bullet_{\V} -"]{rrr} &&& \widehat{\V}^{\DR}_G
            \end{tikzcd}
        \end{equation}
        commutes. Let us show that the following diagram
        \begin{equation}
            \begin{tikzcd}
                \widehat{\W}^{\DR}_G \ar["\lambda \bullet_{\W} -"]{rr} \ar["\widehat{\Delta}^{\W, \DR}_G"']{d} && \widehat{\W}^{\DR}_G \ar["\widehat{\Delta}^{\W, \DR}_G"]{d} \\
                (\widehat{\W}^{\DR}_G)^{\hat{\otimes} 2} \ar["(\lambda \bullet_{\W} -)^{\otimes 2}"]{rr} && (\widehat{\W}^{\DR}_G)^{\hat{\otimes} 2}
            \end{tikzcd}
        \end{equation}
        commutes. Indeed, for $(n, g) \in \N_{>0} \times G$ we have
        \begin{align*}
            \mbox{\small$\widehat{\Delta}^{\W, \DR}_G \circ (\lambda \bullet_{\W} -) (z_{n, g})$} & = \widehat{\Delta}^{\W, \DR}_G(\lambda^n z_{n, g}) = \lambda^n \widehat{\Delta}^{\W, \DR}_G(z_{n, g}) \\
            & = \lambda^n z_{n, g} \otimes 1 + 1 \otimes \lambda^n z_{n, g} + \lambda^n \sum_{\substack{k=1 \\ h \in G}}^{n-1} z_{k, \phi(h)} \otimes z_{n-k, gh^{-1})} \\
            & = \mbox{\small$\displaystyle\left(\lambda \bullet_{\W} -\right)^{\otimes 2}\Big(z_{n, g} \otimes 1 + 1 \otimes z_{n, g} + \sum_{\substack{k = 1 \\ h \in G}}^{n-1} z_{k, h} \otimes z_{n-k, g h^{-1}}\Big)$} \\
            & = \left(\lambda \bullet_{\W} -\right)^{\otimes 2} \circ \widehat{\Delta}^{\W, \DR}_G (z_{n, g}). 
        \end{align*}
        \item One checks that $\lambda \bullet_{\V} -$ preserves the submodule $\mbox{\small$\displaystyle \widehat{\V}^{\DR}_G e_0 + \sum_{g \in G} \widehat{\V}^{\DR}_G (g - 1)$}$. It follows that there is a unique $\K$-module automorphism $\lambda \bullet_{\M} -$ of $\widehat{\M}^{\DR}_G$ such that the following diagram
        \begin{equation}
            \label{diag_lambdaV_lambdaM}
            \begin{tikzcd}
                \widehat{\V}^{\DR}_G \ar["\lambda \bullet_{\V} -"]{rrr} \ar["- \cdot 1_{\DR}"']{d} &&& \widehat{\V}^{\DR}_G \ar["- \cdot 1_{\DR}"]{d} \\
                \widehat{\M}^{\DR}_G \ar["\lambda \bullet_{\M} -"]{rrr} &&& \widehat{\M}^{\DR}_G
            \end{tikzcd}
        \end{equation}
        commutes. Since $\lambda \bullet_{\V} -$ restricts to the automorphism $\lambda \bullet_{\W} -$ of $\widehat{\W}^{\DR}_G$, we obtain the following commutative diagram
        \begin{equation}
            \label{diag_lambdaW_lambdaM}
            \begin{tikzcd}
                \widehat{\W}^{\DR}_G \ar["\lambda \bullet_{\W} -"]{rrr} \ar["- \cdot 1_{\DR}"']{d} &&& \widehat{\W}^{\DR}_G \ar["- \cdot 1_{\DR}"]{d} \\
                \widehat{\M}^{\DR}_G \ar["\lambda \bullet_{\M} -"]{rrr} &&& \widehat{\M}^{\DR}_G
            \end{tikzcd}
        \end{equation}
        We then have the following cube
        \[\begin{tikzcd}[row sep={40,between origins}, column sep={40,between origins}]
            &  & \widehat{\M}^{\DR}_G \arrow[ddd, "\widehat{\Delta}^{\M, \DR}_G"] \arrow[rrrr, "\lambda \bullet_{\M} -"] & & & & \widehat{\M}^{\DR}_G \arrow[ddd, "\widehat{\Delta}^{\M, \DR}_G"] \\
            \widehat{\W}^{\DR}_G \arrow[rrrr, "\hspace{1.5cm} \lambda \bullet_{\W} -"] \arrow[ddd, "\widehat{\Delta}^{\W, \DR}_G"'] \arrow[rru, "- \cdot 1_{\DR}"] & & & & \widehat{\W}^{\DR}_G \arrow[ddd, "\widehat{\Delta}^{\W, \DR}_G"] \arrow[rru, "- \cdot 1_{\DR}"'] & & \\
            & & & & & & \\
            & & (\widehat{\M}^{\DR}_G)^{\hat{\otimes} 2} \arrow[rrrr, "\hspace{-1.8cm}(\lambda \bullet_{\M} -)^{\otimes 2}"'] & & & & (\widehat{\M}^{\DR}_G)^{\hat{\otimes} 2} \\
            (\widehat{\W}^{\DR}_G)^{\hat{\otimes} 2} \arrow[rrrr, "(\lambda \bullet_{\W} -)^{\otimes 2}"] \arrow[rru, "- \cdot 1_{\DR}^{\otimes 2}"] & & & & (\widehat{\W}^{\DR}_G)^{\hat{\otimes} 2} \arrow[rru, "- \cdot 1_{\DR}^{\otimes 2}"'] & &
        \end{tikzcd}\]
        The left and the right faces are exactly the same square, which is commutative since it corresponds to Diagram \ref{diag_DeltaW_DeltaM}. The upper side commutes thanks to Diagram \eqref{diag_lambdaW_lambdaM} and the lower side is the tensor square of the upper side so is commutative. Finally, \ref{lambdaW} gives us the commutativity of the front side. This collection of commutativities together with the surjectivity of $- \cdot 1_{\DR}$ implies that the back side of the cube commutes, which proves that $\lambda \bullet_{\M} -$ is a coalgebra automorphism of $(\widehat{\M}^{\DR}_G, \widehat{\Delta}^{\M, \DR}_G)$.
    \end{enumerate}
\end{proof}

\begin{proposition}
    \label{bullet_pairs}
    Let $\lambda \in \K^{\times}$.
    \begin{enumerate}[label=(\roman*), leftmargin=*]
        \item \label{lambdaV_lambdaM_alg_mod} The pair $(\lambda \bullet_{\V} -, \lambda \bullet_{\M} -)$ is an automorphism of $(\widehat{\V}^{\DR}_G, \widehat{\M}^{\DR}_G)$ in the category $\K{\text -}\alg{\text -}\Mod_{\mathrm{top}}$.
        \item \label{lambdaW_lambdaM_HAMC} The pair $(\lambda \bullet_{\W} -, \lambda \bullet_{\M} -)$ is an automorphism of $((\widehat{\W}^{\DR}_G, \widehat{\Delta}^{\W, \DR}_G), (\widehat{\M}^{\DR}_G, \widehat{\Delta}^{\M, \DR}_G))$ in the category $\K{\text -}\mathrm{HAMC}_{\mathrm{top}}$.
    \end{enumerate}
\end{proposition}
\begin{proof} \ 
    \begin{enumerate}[label=(\roman*), leftmargin=*]
        \item Let $(v, m) \in \widehat{\V}^{\DR}_G \times \widehat{\M}^{\DR}_G$. Since $- \cdot 1_{\DR} : \widehat{\V}^{\DR}_G \to \widehat{\M}^{\DR}_G$ is surjective, there exist $v' \in \widehat{\V}^{\DR}_G$ such that $m = v' \cdot 1_{\DR}$. We have
        \begin{align*}
            \lambda \bullet_{\M} (vm) & = \lambda \bullet_{\M} (v v' \cdot 1_{\DR}) = (\lambda \bullet_{\V} vv') \cdot 1_{\DR} = (\lambda \bullet_{\V}v) \, (\lambda \bullet_{\V}v') \cdot 1_{\DR} \\
            & = (\lambda \bullet_{\V}v) \, (\lambda \bullet_{\M} m),
        \end{align*}
        where the second and fourth equalities come from the commutativity of Diagram \eqref{diag_lambdaV_lambdaM}.
        \item It follows from \ref{lambdaV_lambdaM_alg_mod} and from Lemma \ref{from_lambdaV}.
    \end{enumerate}
\end{proof}

\subsubsection{The torsor \texorpdfstring{$\DMR_{\times}^{\iota}(\K)$}{DMRxiota}}
\begin{proposition}
    \label{bullet_DMR}
    For any $\lambda \in \K^{\times}$, the map $\lambda \bullet - : \KX \to \KX$ restricts to a group automorphism of $(\DMR_{0}^{G}(\K), \circledast)$.
\end{proposition}

\begin{proof}
    It follows from Proposition \ref{bullet_group_aut} that $(\lambda \bullet - )_{| \, \G(\KX)}$ is a group automorphism of $(\G(\KX), \circledast)$.
    It remains to prove that this permutation of $\G(\KX)$ induces a permutation of its subset $\DMR_{0}^{G}(\K)$. Let $\lambda \in \K^{\times}$ and $\Psi \in \DMR_0^G(\K)$. Since $\lambda \bullet x_0 = \lambda x_0$ and $\lambda \bullet x_g = \lambda x_g$ for $g \in G$, Conditions \ref{DMR_lambda_condition_i}, \ref{DMR_lambda_condition_iii}, \ref{DMR_lambda_condition_iv} and \ref{DMR_lambda_condition_v} of Definition \ref{DMR_lambda} are immediately satisfied by $\lambda \bullet \Psi$. In order to prove that Condition \ref{DMR_lambda_condition_ii} is satisfied, let us use Proposition \ref{DMR_lambda_condition_ii_new}. We have
    \begin{align}
        (\lambda \bullet \Psi)^{\star} & = \left(\Gamma^{-1}_{\lambda \bullet \Psi}(-e_1) \beta(\lambda \bullet \Psi \otimes 1)\right) \cdot 1_{\DR} = \left((\lambda \bullet_{\V} \Gamma^{-1}_{\Psi}(-e_1)) \, (\lambda \bullet_{\V} \beta(\Psi \otimes 1))\right) \cdot 1_{\DR} \notag \\
        & = \left(\lambda \bullet_{\V} (\Gamma^{-1}_{\Psi}(-e_1) \beta(\Psi \otimes 1))\right) \cdot 1_{\DR} = \lambda \bullet_{\M} \left(\Gamma^{-1}_{\Psi}(-e_1) \beta(\Psi \otimes 1) \cdot 1_{\DR}\right) \notag \\
        & = \lambda \bullet_{\M} \Psi^{\star}, \label{lambda_bullet_DR}
    \end{align}
    where the second equality comes from the commutativity of Diagram \eqref{Diag_bullets} and Identity \eqref{bullet_Gamma}, the third one from the fact that $\lambda \bullet_{\V} -$ is an algebra morphism and the fourth one from Lemma \ref{from_lambdaV}. \ref{lambdaM}.
    Therefore, we obtain that
    \begin{align*}
        \widehat{\Delta}^{\M, \DR}_{G}((\lambda \bullet \Psi)^{\star}) & = \widehat{\Delta}^{\M, \DR}_{G}(\lambda \bullet_{\M} \Psi^{\star}) = (\lambda \bullet_{\M} -)^{\otimes 2} \circ \widehat{\Delta}^{\M, \DR}_{G}(\Psi^{\star}) \\
        & = (\lambda \bullet_{\M} -)^{\otimes 2} ((\Psi^{\star})^{\otimes 2}) = (\lambda \bullet_{\M} \Psi^{\star})^{\otimes 2} \\
        & = (\lambda \bullet \Psi)^{\star} \otimes (\lambda \bullet \Psi)^{\star},   
    \end{align*}
    where the first and last equalities come from \eqref{lambda_bullet_DR}, the second one from Lemma \ref{from_lambdaV}. \ref{lambdaM} and the third one from Proposition \ref{DMR_lambda_condition_ii_new} and the fact that $\Psi \in \DMR_0^G(\K)$. This proves that $\lambda \bullet -$ restricts to a self-map of $\DMR_0^G(\K)$.
    Following the same steps, one shows that $(\lambda \bullet -)^{-1} = \lambda^{-1} \bullet -$ restricts to a self-map of $\DMR_0^G(\K)$ thus proving the statement.
\end{proof}

\noindent Proposition \ref{bullet_DMR} enables us to state the following definition:
\begin{definition}
    We denote \index{$\K^{\times} \ltimes \DMR_{0}^{G}(\K)$}$\K^{\times} \ltimes \DMR_{0}^{G}(\K)$ the semi-direct product of $\K^{\times}$ and $\DMR_{0}^{G}(\K)$ with respect to the action given in Proposition \ref{bullet_DMR}. It is a subgroup of $\K^{\times} \ltimes \G(\KX)$.
\end{definition}

\begin{definition}
    Let $\iota : G \to \C^{\times}$ be a group embedding. We define \index{$\DMR^{\iota}_{\times}(\K)$}
    \begin{equation}
        \DMR^{\iota}_{\times}(\K) := \{ (\lambda, \Psi) \in \K^{\times} \times \G(\KX) \, | \, \Psi \in \DMR_{\lambda}^{\iota}(\K) \}.
    \end{equation}
\end{definition}

\begin{proposition}
    \label{DMR_torsor}
    Let $\iota : G \to \C^{\times}$ be a group embedding. The group $\K^{\times} \ltimes \DMR_{0}^{G}(\K)$ acts freely and transitively on $\DMR^{\iota}_{\times}(\K)$ by left multiplication $\circledast$. 
\end{proposition}
\noindent In order to prove this, we will need the following Lemma:
\begin{lemma}
    \label{DMR_transfert}
    Let $\iota : G \to \C^{\times}$ be a group embedding and $\lambda, \nu \in \K^{\times}$. If $\Phi \in \DMR_{\nu}^{\iota}(\K)$ then $\lambda \nu^{-1} \bullet \Phi \in \DMR_{\lambda}^{\iota}(\K)$.
\end{lemma}
\begin{proof}
    Let $\Phi \in \DMR_{\nu}^{\iota}(\K)$. Since $\lambda \nu^{-1} \bullet x_0 = \lambda \nu^{-1} x_0$ and $\lambda \nu^{-1} \bullet x_g = \lambda \nu^{-1} x_g$ for $g \in G$, we have
    \begin{itemize}[leftmargin=*]
        \item $(\lambda \nu^{-1} \bullet \Phi | x_0) = \lambda \nu^{-1} (\Phi | x_0) = 0$ and $(\lambda \nu^{-1} \bullet \Phi | x_1) = \lambda \nu^{-1} (\Phi | x_1) = 0$.
        \item $(\lambda \nu^{-1} \bullet \Phi | x_0x_1) = \lambda^2 {\nu^{-1}}^2 (\Phi | x_0x_1) = - \lambda^2 {\nu^{-1}}^2 \frac{\nu^2}{24} = - \frac{\lambda^2}{24}$.
        \item $(\lambda \nu^{-1} \bullet \Phi | x_{g_{\iota}} - x_{g_{\iota}^{-1}}) = \lambda \nu^{-1} (\Phi | x_{g_{\iota}} - x_{g_{\iota}^{-1}}) = \lambda \nu^{-1} \frac{|G| - 2}{2} \nu = \frac{|G| - 2}{2} \lambda$.
    \end{itemize}
    This proves respectively conditions \ref{DMR_lambda_condition_i}, \ref{DMR_lambda_condition_iii} and \ref{DMR_lambda_condition_iv} of Definition \ref{DMR_lambda}. Condition \ref{DMR_lambda_condition_v} follows from condition \ref{DMR_lambda_condition_iv}.
    Finally, one shows that Condition \ref{DMR_lambda_condition_ii} is satisfied by using the same arguments as in the proof of Proposition \ref{bullet_DMR}.
\end{proof}

\begin{proof}[Proof of Proposition \ref{DMR_torsor}]
    Since the action of the group $\K^{\times} \ltimes \G(\KX)$ on the set $\K^{\times} \times \G(\KX)$ by left multiplication $\circledast$ is free, so is its restriction to the action of $\K^{\times} \ltimes \DMR_{0}^{G}(\K)$ on $\DMR^{\iota}_{\times}(\K)$. Let us show that this action is transitive. Let $(\lambda, \Psi)$ and $(\nu, \Phi) \in \DMR^{\iota}_{\times}(\K)$. Set $\mu = \lambda \nu^{-1}$. It follows from Lemma \ref{DMR_transfert} that $\mu \bullet \Phi \in \DMR_{\lambda}^{\iota}(\K)$. Thanks to Proposition \ref{Racinet_main_result}.\ref{Racinet_torsor}, the action of the group $(\DMR_{0}^{G}(\K), \circledast)$ on $\DMR_{\lambda}^{\iota}(\K)$ is transitive, therefore, there exists $\Lambda \in \DMR_{0}^{G}(\K)$ such that $\Lambda \circledast (\mu \bullet \Phi) = \Psi$. Thus $(\mu, \Lambda) \in \K^{\times} \ltimes \DMR_{0}^{G}(\K)$ is such that
    \[
        (\mu, \Lambda) \circledast (\nu, \Phi) = (\lambda, \Psi),
    \]
    which proves the transitivity.
\end{proof}

\subsection{The torsor \texorpdfstring{$\DMR_{\times}(\K)$}{DMRx}} \label{DMRx}

\subsubsection{Action of the group \texorpdfstring{$\Aut(G)$}{Aut(G)} on \texorpdfstring{$\KX$}{KX}}
The group $\Aut(G)$ acts on $\KX$ by $\K$-algebra automorphisms, the element $\phi \in \Aut(G)$ acting by the automorphism \index{$\eta_{\phi}$}$\eta_{\phi}$ given by
\begin{equation}
    x_0 \mapsto x_0, \quad x_g \mapsto x_{\phi(g)} \text{ for } g \in G.
\end{equation}
One checks that, for any $\phi \in \Aut(G)$, the automorphism $\eta_{\phi}$ is a Hopf algebra automorphism of $(\KX, \widehat{\Delta})$.

In addition, for any $\phi \in \Aut(G)$ and any $g \in G$, we have
\begin{equation}
    \label{eta_phi_circ_t}
    \eta_{\phi} \circ t_g = t_g \circ \eta_{\phi},
\end{equation}
which can be verified by checking this identity on generators since both sides are given as a composition of $\K$-algebra morphisms. Let us show that

\begin{proposition}
    \label{eta_phi_GKX}
    Let $\phi \in \Aut(G)$. The map $\eta_{\phi}$ restricts to a group automorphism of $(\G(\KX), \circledast)$.
\end{proposition}

\noindent In order to prove this, we will need the following Lemma:
\begin{lemma}
    \label{eta_phi_aut_Psi}
    Let $\phi \in \Aut(G)$, $g \in G$ and $\Psi \in \G(\KX)$. We have
    \begin{equation}
        \eta_{\phi} \circ \aut_{\Psi} = \aut_{\eta_{\phi}(\Psi)} \circ \eta_{\phi}.
    \end{equation}
\end{lemma}
\begin{proof}
    Since all morphisms are algebra automorphisms of $\KX$, it is enough to check this identity on generators. We have
    \[
        \eta_{\phi} \circ \aut_{\Psi}(x_0) = \eta_{\phi}(x_0) = x_0 = \aut_{\eta_{\phi}(\Psi)}(x_0) = \aut_{\eta_{\phi}(\Psi)} \circ \eta_{\phi} (x_0)
    \]
    and for $g \in G$,
    \begin{align*}
        \eta_{\phi} \circ \aut_{\Psi}(x_g) & = \eta_{\phi}(\Ad_{t_g(\Psi^{-1})}(x_g)) = \Ad_{t_{\phi(g)}(\eta_{\phi}(\Psi)^{-1})}(\eta_{\phi}(x_g)) \\
        & = \Ad_{t_{\phi(g)}(\eta_{\phi}(\Psi)^{-1})}(x_{\phi(g)}) = \aut_{\eta_{\phi}(\Psi)}(x_{\phi(g)}) = \aut_{\eta_{\phi}(\Psi)} \circ \eta_{\phi}(x_g),
    \end{align*}
    where the second equality comes from identity \eqref{eta_phi_circ_t}.
\end{proof}

\begin{proof}[Proof of Proposition \ref{eta_phi_GKX}]
    Let $\phi \in \Aut(G)$. Since $\eta_{\phi}$ is a Hopf algebra automorphism of $(\KX, \widehat{\Delta})$, it restricts to a map $\G(\KX) \to \G(\KX)$. Let $\Psi, \Phi \in \G(\KX)$. We have
    \begin{align*}
        \eta_{\phi}(\Psi \circledast \Phi) & = \eta_{\phi}(\Psi \aut_{\Psi}(\Phi)) = \eta_{\phi}(\Psi) \eta_{\phi}(\aut_{\Psi}(\Phi)) \\
        & = \eta_{\phi}(\Psi) \aut_{\eta_{\phi}(\Psi)}(\eta_{\phi}(\Phi)) = \eta_{\phi}(\Psi) \circledast \eta_{\phi}(\Phi),
    \end{align*}
    where the third equality comes from Lemma \ref{eta_phi_aut_Psi}. This proves that $\eta_{\phi}$ restricts to a group endomorphism of $(\G(\KX), \circledast)$. Finally, one has that $\eta_{\phi}^{-1} = \eta_{\phi^{-1}}$ and the above computations shows that $\eta_{\phi}^{-1}$ is an endomorphism of $(\G(\KX), \circledast)$, thus proving the statement.   
\end{proof}

\begin{lemma}
    \label{eta_bullet}
    For $(\phi, \lambda) \in \Aut(G) \times \K^{\times}$, we have
    \[
        \eta_{\phi} \circ (\lambda \bullet -) = (\lambda \bullet -) \circ \eta_{\phi}. 
    \]
\end{lemma}
\begin{proof}
    Let $(\phi, \lambda) \in \Aut(G) \times \K^{\times}$. Since all the morphisms are algebra automorphisms of $\KX$, it is enough to check this identity on generators. We have
    \[
        \eta_{\phi} \circ (\lambda \bullet x_0) = \eta_{\phi}(\lambda x_0) = \lambda x_0 = \lambda \bullet x_0 = \lambda \bullet \eta_{\phi}(x_0) = (\lambda \bullet -) \circ \eta_{\phi} (x_0)
    \]
    and for $g \in G$,
    \[
        \eta_{\phi} \circ (\lambda \bullet x_g) = \eta_{\phi} (\lambda x_g) = \lambda x_{\phi(g)} = \lambda \bullet x_{\phi(g)} = \lambda \bullet \eta_{\phi}(x_g) = (\lambda \bullet -) \circ \eta_{\phi} (x_g).
    \]
\end{proof}

\noindent Propositions \ref{bullet_group_aut} and \ref{eta_phi_GKX} and Lemma \ref{eta_bullet} enable us to define the following:
\begin{definition}
    We denote \index{$(\Aut(G) \times \K^{\times}) \ltimes \G(\KX)$}$(\Aut(G) \times \K^{\times}) \ltimes \G(\KX)$ the semi-direct product of $\Aut(G) \times \K^{\times}$ and $\G(\KX)$ with respect to the action given in Propositions \ref{bullet_group_aut} and \ref{eta_phi_GKX}. It consists of the set $\Aut(G) \times \K^{\times} \times \G(\KX)$ endowed with a group law which will also be denoted $\circledast$ and we have for $(\phi, \lambda, \Psi), (\phi^{\prime}, \nu, \Phi) \in \Aut(G) \times \K^{\times} \times \G(\KX)$, \index{$(\phi, \lambda, \Psi) \circledast (\phi^{\prime}, \nu, \Phi)$}
    \begin{equation}
        (\phi, \lambda, \Psi) \circledast (\phi^{\prime}, \nu, \Phi) := (\phi \circ \phi^{\prime}, \lambda \nu, \Psi \circledast \eta_{\phi}(\lambda \bullet \Phi)).
    \end{equation}
\end{definition}

\subsubsection{Action of the group \texorpdfstring{$\Aut(G)$}{Aut(G)} on \texorpdfstring{$\Emb(G)$}{Emb(G)}}
Let us denote \index{$\Emb(G)$}
\begin{equation}
    \Emb(G) := \{ \iota : G \to \C^{\times} \, | \, \iota \text{ is a group embedding}\}.
\end{equation}
\begin{lemma}
    \label{Emb_torsor}
    The group \index{$\Aut(G)$}$\Aut(G)$ acts freely and transitively on $\Emb(G)$ by
    \begin{equation}
        (\phi, \iota) \longmapsto \iota \circ \phi^{-1},
    \end{equation}
    for $(\phi, \iota) \in \Aut(G) \times \Emb(G)$.
\end{lemma}
\begin{proof}
    One knows that for any $\iota \in \Emb(G)$, $\iota(G) = \mu_N$. That gives rise to a group isomorphism $\widetilde{\iota} : G \to \mu_N(\C)$.
    Therefore, for any $\iota, \iota^{\prime} \in \Emb(G)$ there is a unique group automorphism $\phi = \widetilde{\iota^{\prime}}^{-1} \circ \widetilde{\iota}$ of $G$ such that $\iota \circ \phi^{-1} = \iota^{\prime}$.
\end{proof}

\begin{corollary}
    \label{EmbGxkxxGkX_torsor}
    The group $(\Aut(G) \times \K^{\times}) \ltimes \G(\KX)$ acts freely and transitively on $\Emb(G) \times \K^{\times} \times \G(\KX)$ by
    \begin{equation}
        (\phi, \lambda, \Psi) \cdot (\iota, \nu, \Phi) = (\iota \circ \phi^{-1}, \lambda \nu, \Psi \circledast \eta_{\phi}(\lambda \bullet \Phi)),
    \end{equation}
    for $(\phi, \lambda, \Psi) \in (\Aut(G) \times \K^{\times}) \ltimes \G(\KX)$ and $(\iota, \nu, \Phi) \in \Emb(G) \times \K^{\times} \times \G(\KX)$.
\end{corollary}
\begin{proof}
    Let $(\iota, \nu, \Phi), (\iota^{\prime}, \nu^{\prime}, \Phi^{\prime}) \in \Emb(G) \times \K^{\times} \times \G(\KX)$. Thanks to Lemma \ref{Emb_torsor}, there is a unique $\phi \in \Aut(G)$ such that $\iota^{\prime} = \iota \circ \phi^{-1}$. Set
    \[
        \lambda = \nu^{-1} \nu^{\prime} \text{ and } \Psi = \Phi^{\prime} \circledast \eta_{\phi}(\lambda \bullet \Phi)^{\circledast (-1)}.
    \]
    In conclusion, there is a unique $(\phi, \lambda, \Psi) \in (\Aut(G) \times \K^{\times}) \ltimes \G(\KX)$ such that
    \[
        (\phi, \lambda, \Psi) \cdot (\iota, \nu, \Phi) = (\iota^{\prime}, \nu^{\prime}, \Phi^{\prime}),
    \]
    which proves the statement.
\end{proof}

\subsubsection{Action of the group \texorpdfstring{$\Aut(G)$}{Aut(G)} on crossed product algebras and module}
The group $\Aut(G)$ acts on $\widehat{\V}^{\DR}_G$ by $\K$-algebra automorphisms the element $\phi \in \Aut(G)$ acting by the automorphism \index{$\eta_{\phi}^{\V}$}$\eta_{\phi}^{\V}$ given by
\begin{equation}
    \label{etaV_def}
    e_0 \mapsto e_0, \quad e_1 \mapsto e_1 \,\, \text{ and } g \mapsto \phi(g) \,\, \text{ for } g \in G.
\end{equation}

\begin{lemma}
    Let $\phi \in \Aut(G)$. The following diagram
    \begin{equation}
        \label{Diag_etaV_eta}
        \begin{tikzcd}
            \KX \ar["\eta_{\phi}"]{rr} \ar["\beta \circ (- \otimes 1)"']{d} && \KX \ar["\beta \circ (- \otimes 1)"]{d} \\
            \widehat{\V}^{\DR}_G \ar["\eta^{\V}_{\phi}"]{rr} && \widehat{\V}^{\DR}_G
        \end{tikzcd}
    \end{equation}
    commutes.
\end{lemma}
\begin{proof}
    Since all arrows are $\K$-algebra morphisms, one easily checks the commutativity of generators.
\end{proof}

\begin{lemma}
    \label{from_etaV}
    Let $\phi \in \Aut(G)$.
    \begin{enumerate}[label=(\roman*), leftmargin=*]
        \item \label{etaW} The $\K$-algebra automorphism $\eta_{\phi}^{\V}$ of $\widehat{\V}^{\DR}_G$ restricts to a Hopf algebra automorphism \index{$\eta_{\phi}^{\W}$}$\eta_{\phi}^{\W}$ of $(\widehat{\W}^{\DR}_G, \widehat{\Delta}^{\W, \DR}_G)$.
        \item \label{etaM} The $\K$-algebra automorphism $\eta_{\phi}^{\V}$ of $\widehat{\V}^{\DR}_G$ induces a coalgebra automorphism \index{$\eta_{\phi}^{\M}$}$\eta_{\phi}^{\M}$ of $(\widehat{\M}^{\DR}_G, \widehat{\Delta}^{\M, \DR}_G)$.
    \end{enumerate}
\end{lemma}
\begin{proof} \ 
    \begin{enumerate}[label=(\roman*), leftmargin=17pt]
        \item For $(n, g) \in \N_{>0} \times G$ we have
        \[
            \eta^{\V}_{\phi}(z_{n, g}) = \eta^{\V}_{\phi}(- e_0^{n-1} g e_1) = - e_0^{n-1} \phi(g) e_1 = z_{n, \phi(g)}.
        \]
        Since the algebra $\widehat{\W}^{\DR}_G$ is freely generated by the family $(z_{n, g})_{(n, g) \in \N_{>0} \times G}$, it follows that $\eta^{\V}_{\phi}(\widehat{\W}^{\DR}_G) \subset \widehat{\W}^{\DR}_G$. Similarly, $(\eta^{\V}_{\phi})^{-1}(\widehat{\W}^{\DR}_G) \subset \widehat{\W}^{\DR}_G$. Hence, $\eta^{\V}_{\phi}(\widehat{\W}^{\DR}_G) = \widehat{\W}^{\DR}_G$.
        This implies that $\eta^{\V}_{\phi}$ restricts to a $\K$-algebra automorphism of $\widehat{\W}^{\DR}_G$ which we denote $\eta^{\W}_{\phi}$ and that we have the following commutative diagram
        \begin{equation}
            \label{diag_etaW_etaV}
            \begin{tikzcd}
                \widehat{\W}^{\DR}_G \ar["\eta^{\W}_{\phi}"]{rrr} \ar[hook]{d} &&& \widehat{\W}^{\DR}_G \ar[hook]{d} \\
                \widehat{\V}^{\DR}_G \ar["\eta^{\V}_{\phi}"]{rrr} &&& \widehat{\V}^{\DR}_G
            \end{tikzcd}
        \end{equation}
        Let us show that the following diagram
        \begin{equation}
            \begin{tikzcd}
                \widehat{\W}^{\DR}_G \ar["\eta^{\W}_{\phi}"]{rr} \ar["\widehat{\Delta}^{\W, \DR}_G"']{d} && \widehat{\W}^{\DR}_G \ar["\widehat{\Delta}^{\W, \DR}_G"]{d} \\
                (\widehat{\W}^{\DR}_G)^{\otimes 2} \ar["(\eta^{\W}_{\phi})^{\otimes 2}"]{rr} && (\widehat{\W}^{\DR}_G)^{\otimes 2}
            \end{tikzcd}
        \end{equation}
        commutes. Indeed, for $(n, g) \in \N_{>0} \times G$ we have
        \begin{align*}
            \widehat{\Delta}^{\W, \DR}_G \circ \eta^{\W}_{\phi}(z_{n, g}) & = \widehat{\Delta}^{\W, \DR}_G(z_{n, \phi(g)}) \\
            & = z_{n, \phi(g)} \otimes 1 + 1 \otimes z_{n, \phi(g)} + \sum_{\substack{k=1 \\ h \in G}}^{n-1} z_{k, h} \otimes z_{n-k, \phi(g)h^{-1}} \\
            & = z_{n, \phi(g)} \otimes 1 + 1 \otimes z_{n, \phi(g)} + \sum_{\substack{k=1 \\ h \in G}}^{n-1} z_{k, \phi(h)} \otimes z_{n-k, \phi(g)\phi(h^{-1})} \\
            & = (\eta^{\W}_{\phi})^{\otimes 2}\Big(z_{n, g} \otimes 1 + 1 \otimes z_{n, g} + \sum_{\substack{k = 1 \\ h \in G}}^{n-1} z_{k, h} \otimes z_{n-k, g h^{-1}}\Big) \\
            & = (\eta^{\W}_{\phi})^{\otimes 2} \circ \widehat{\Delta}^{\W, \DR}_G (z_{n, g}). 
        \end{align*}
        \item Let $\phi \in \Aut(G)$. One checks that $\eta^{\V}_{\phi}$ preserves the submodule $\mbox{\scriptsize$\displaystyle \widehat{\V}^{\DR}_G e_0 + \sum_{g \in G} \widehat{\V}^{\DR}_G (g - 1)$}$. It follows that there is a unique $\K$-module automorphism $\eta^{\M}_{\phi}$ of $\widehat{\M}^{\DR}_G$ such that the following diagram
        \begin{equation}
            \label{diag_etaV_etaM}
            \begin{tikzcd}
                \widehat{\V}^{\DR}_G \ar["\eta^{\V}_{\phi}"]{rrr} \ar["- \cdot 1_{\DR}"']{d} &&& \widehat{\V}^{\DR}_G \ar["- \cdot 1_{\DR}"]{d} \\
                \widehat{\M}^{\DR}_G \ar["\eta^{\M}_{\phi}"]{rrr} &&& \widehat{\M}^{\DR}_G
            \end{tikzcd}
        \end{equation}
        commutes. Combined with \ref{etaW}, it gives the following commutative diagram \begin{equation}
            \label{diag_etaW_etaM}
            \begin{tikzcd}
                \widehat{\W}^{\DR}_G \ar["\eta^{\W}_{\phi}"]{rrr} \ar["- \cdot 1_{\DR}"']{d} &&& \widehat{\W}^{\DR}_G \ar["- \cdot 1_{\DR}"]{d} \\
                \widehat{\M}^{\DR}_G \ar["\eta^{\M}_{\phi}"]{rrr} &&& \widehat{\M}^{\DR}_G
            \end{tikzcd}
        \end{equation}
        We then have the following cube
        \[\begin{tikzcd}[row sep={40,between origins}, column sep={40,between origins}]
            &  & \widehat{\M}^{\DR}_G \arrow[ddd, "\widehat{\Delta}^{\M, \DR}_G"] \arrow[rrrr, "\eta_{\phi}^{\M}"] & & & & \widehat{\M}^{\DR}_G \arrow[ddd, "\widehat{\Delta}^{\M, \DR}_G"] \\
            \widehat{\W}^{\DR}_G \arrow[rrrr, "\hspace{1cm} \eta_{\phi}^{\W}"] \arrow[ddd, "\widehat{\Delta}^{\W, \DR}_G"'] \arrow[rru, "- \cdot 1_{\DR}"] & & & & \widehat{\W}^{\DR}_G \arrow[ddd, "\widehat{\Delta}^{\W, \DR}_G"] \arrow[rru, "- \cdot 1_{\DR}"'] & & \\
            & & & & & & \\
            & & (\widehat{\M}^{\DR}_G)^{\hat{\otimes} 2} \arrow[rrrr, "\hspace{-1.2cm}(\eta_{\phi}^{\M})^{\otimes 2}"'] & & & & (\widehat{\M}^{\DR}_G)^{\hat{\otimes} 2} \\
            (\widehat{\W}^{\DR}_G)^{\hat{\otimes} 2} \arrow[rrrr, "(\eta_{\phi}^{\W})^{\otimes 2}"] \arrow[rru, "- \cdot 1_{\DR}^{\otimes 2}"] & & & & (\widehat{\W}^{\DR}_G)^{\hat{\otimes} 2} \arrow[rru, "- \cdot 1_{\DR}^{\otimes 2}"'] & &
        \end{tikzcd}\]
        The left and the right faces are exactly the same square, which is commutative since it corresponds to Diagram \ref{diag_DeltaW_DeltaM}. The upper side commutes thanks to Diagram \eqref{diag_etaW_etaM} and the lower side is the tensor square of the upper side so is commutative. Finally, \ref{etaW} gives us the commutativity of the front side. This collection of commutativities together with the surjectivity of $- \cdot 1_{\DR}$ implies that the back side of the cube commutes, which proves that $\eta_{\phi}^{\M}$ is a coalgebra automorphism of $(\widehat{\M}^{\DR}_G, \widehat{\Delta}^{\M, \DR}_G)$.
    \end{enumerate}
\end{proof}

\begin{proposition}
    \label{eta_V_etaWM}
    Let $\phi \in \Aut(G)$.
    \begin{enumerate}[label=(\roman*), leftmargin=17pt]
        \item \label{etaV_etaM_alg_mod} The pair $(\eta_{\phi}^{\V}, \eta_{\phi}^{\M})$ is an automorphism of $(\widehat{\V}^{\DR}_G, \widehat{\M}^{\DR}_G)$ in the category $\K{\text -}\alg{\text -}\Mod_{\mathrm{top}}$.
        \item \label{etaW_etaM_HAMC} The pair $(\eta_{\phi}^{\W}, \eta_{\phi}^{\M})$ is an automorphism of $((\widehat{\W}^{\DR}_G, \widehat{\Delta}^{\W, \DR}_G), (\widehat{\M}^{\DR}_G, \widehat{\Delta}^{\M, \DR}_G))$ in the category $\K{\text -}\mathrm{HAMC}_{\mathrm{top}}$.
    \end{enumerate}
\end{proposition}
\begin{proof} \ 
    \begin{enumerate}[label=(\roman*), leftmargin=17pt]
        \item Let $(v, m) \in \widehat{\V}^{\DR}_G \times \widehat{\M}^{\DR}_G$. Since $- \cdot 1_{\DR} : \widehat{\V}^{\DR}_G \to \widehat{\M}^{\DR}_G$ is surjective, there exist $v' \in \widehat{\V}^{\DR}_G$ such that $m = v' \cdot 1_{\DR}$. We have
        \begin{align*}
            \eta_{\phi}^{\M}(vm) & = \eta_{\phi}^{\M}(v v' \cdot 1_{\DR}) = \eta_{\phi}^{\V}(vv') \cdot 1_{\DR} \\
            & = \eta_{\phi}^{\V}(v) \, \eta_{\phi}^{\V}(v') \cdot 1_{\DR} = \eta_{\phi}^{\V}(v) \, \eta_{\phi}^{\M}(m),
        \end{align*}
        where the second and fourth equalities come from Lemma \ref{from_etaV}. \ref{etaM}.
        \item It follows from \ref{etaV_etaM_alg_mod} and from Lemma \ref{from_etaV}.
    \end{enumerate}
\end{proof}

\subsubsection{The torsor \texorpdfstring{$\DMR_{\times}(\K)$}{DMRx}}
Lemma \ref{Emb_torsor} sets up the following result:
\begin{proposition}
    \label{eta_phi_DMR}
    Let $\lambda \in \K$. For $\iota, \iota^{\prime} \in \Emb(G)$, the element $\phi \in \Aut(G)$ such that $\iota^{\prime} = \iota \circ \phi$ is such that $\eta_{\phi}$ is a bijection between $\DMR_{\lambda}^{\iota}(\K)$ and $\DMR_{\lambda}^{\iota^{\prime}}(\K)$.
\end{proposition}
\begin{proof}
    Since $\eta_{\phi}$ is a Hopf algebra automorphism of $(\KX, \widehat{\Delta})$, it restricts to group automorphism of $\G(\KX)$. It remains to show that $\eta_{\phi} : \DMR_{\lambda}^{\iota}(\K) \to \DMR_{\lambda}^{\iota \circ \phi^{-1}}(\K)$ is a bijection.
    Let $\Psi \in \DMR_{\lambda}^{\iota}(\K)$.  
    Since $\phi(x_0)=x_0$ and $\phi(x_1)=x_1$, Conditions \ref{DMR_lambda_condition_i} and \ref{DMR_lambda_condition_iii} of Definition \ref{DMR_lambda} are immediately satisfied by $\eta_{\phi}(\Psi)$. Moreover, since $g_{\iota \circ \phi^{-1}} = \phi(g_{\iota})$, we have
    \begin{align*}
        \left(\eta_{\phi}(\Psi) | x_{g_{\iota \circ \phi^{-1}}} - x_{g_{\iota \circ \phi^{-1}}^{-1}}\right) & = \left(\eta_{\phi}(\Psi) | x_{\phi(g_{\iota})} - x_{\phi(g_{\iota}^{-1})}\right) = \left(\eta_{\phi}(\Psi) | \eta_{\phi}(x_{g_{\iota}} - x_{g_{\iota}^{-1}})\right) \\
        & = \left(\Psi | x_{g_{\iota}} - x_{g_{\iota}^{-1}}\right) = \frac{|G| - 2}{2} \lambda.
    \end{align*}
    Then Identity \ref{DMR_lambda_condition_iv} of Definition \ref{DMR_lambda} follows. One checks Identity \ref{DMR_lambda_condition_v} in a similar way.
    Let us prove that Condition \ref{DMR_lambda_condition_ii} is satisfied by $\eta_{\phi}(\Psi)$. We have
    \begin{align*}
        (\eta_{\phi}(\Psi))^{\star} & = \left(\Gamma_{\eta_{\phi}(\Psi)}^{-1}(-e_1) \beta(\eta_{\phi}(\Psi) \otimes 1)\right) \cdot 1_{\DR} = \left(\Gamma_{\Psi}^{-1}(-e_1) \beta(\eta_{\phi}(\Psi) \otimes 1)\right) \cdot 1_{\DR} \\
        & = \left(\Gamma_{\Psi}^{-1}(-e_1) \eta_{\phi}^{\V}(\beta(\Psi \otimes 1))\right) \cdot 1_{\DR} = \left(\eta_{\phi}^{\V}\left(\Gamma_{\Psi}^{-1}(-e_1) \beta(\Psi \otimes 1)\right)\right) \cdot 1_{\DR} \\
        & = \eta_{\phi}^{\M}\left(\Gamma_{\Psi}^{-1}(-e_1) \beta(\Psi \otimes 1) \cdot 1_{\DR}\right) = \eta_{\phi}^{\M}(\Psi^{\star}),
    \end{align*}
    where the second equality comes from the fact that $\eta_{\phi}(x_1)=x_1$, the third one from the commutativity of Diagram \eqref{Diag_etaV_eta}, the fourth one from the fact that $\eta_{\phi}^{\V}(e_1)=e_1$ and the fifth one from the commutativity of Diagram \eqref{diag_etaV_etaM}. Therefore, thanks to Lemma \ref{from_etaV}.\ref{etaM} and the fact that $\Psi \in \DMR^{\iota}_{\lambda}(\K)$, we obtain that
    \begin{align*}
        \widehat{\Delta}^{\M, \DR}_{G}\left((\eta_{\phi}(\Psi))^{\star}\right) & = \widehat{\Delta}^{\M, \DR}_{G}\left(\eta_{\phi}^{\M}(\Psi^{\star})\right) = (\eta_{\phi}^{\M})^{\otimes 2}\left(\widehat{\Delta}^{\M, \DR}_{G}(\Psi^{\star})\right) \\
        & = (\eta_{\phi}^{\M})^{\otimes 2}(\Psi^{\star} \otimes \Psi^{\star}) = \eta_{\phi}^{\M}(\Psi^{\star})^{\otimes 2} = (\eta_{\phi}(\Psi))^{\star} \otimes (\eta_{\phi}(\Psi))^{\star},
    \end{align*}
    which implies, by Proposition \ref{DMR_lambda_condition_ii_new}, that condition \ref{DMR_lambda_condition_ii} of Definition \ref{DMR_lambda} is verified by $(\eta_{\phi}(\Psi))^{\star}$. This proves that $\eta_{\phi}$ restricts to a map $\DMR_{\lambda}^{\iota}(\K) \to \DMR_{\lambda}^{\iota \circ \phi^{-1}}(\K)$. Finally, following the same steps, one shows that $\eta_{\phi}^{-1} = \eta_{\phi^{-1}}$ restricts to a map $\DMR_{\lambda}^{\iota \circ \phi^{-1}}(\K) \to \DMR_{\lambda}^{\iota}(\K)$ thus proving the statement.
\end{proof}

\begin{corollary}
    \label{eta_phi_DMR0}
    For any $\phi \in \Aut(G)$, the map $\eta_{\phi}$ restricts to a group automorphism of $(\DMR_{0}^{G}(\K), \circledast)$.
\end{corollary}
\begin{proof}
    From Proposition \ref{eta_phi_DMR} it follows that for $\lambda = 0$ and any $\phi \in \Aut(G)$ the map $\eta_{\phi}$ restricts to a bijection from $\DMR_{0}^{G}(\K)$ to itself. In addition, since $(\DMR_{0}^{G}(\K), \circledast)$ is a subgroup of $(\G(\KX), \circledast)$, Proposition \ref{eta_phi_GKX} states that this map is a group morphism.
\end{proof}

\noindent Proposition \ref{bullet_DMR}, Corollary \ref{eta_phi_DMR0} and Lemma \ref{eta_bullet} enable us to state the following definition:
\begin{definition}
    We denote \index{$(\Aut(G) \times \K^{\times}) \ltimes \DMR_{0}^{G}(\K)$}$(\Aut(G) \times \K^{\times}) \ltimes \DMR_{0}^{G}(\K)$ the semi-direct product of $\Aut(G) \times \K^{\times}$ and $\DMR_{0}^{G}(\K)$ with respect to the group action of $\Aut(G) \times \K^{\times}$ induced by Corollary \ref{eta_phi_DMR0} and Proposition \ref{bullet_DMR}. It is a subgroup of $(\Aut(G) \times \K^{\times}) \ltimes \G(\KX)$.
\end{definition}

\begin{definition}
    We define \index{$\DMR_{\times}(\K)$}
    \begin{equation}
        \DMR_{\times}(\K) := \{ (\iota, \lambda, \Psi) \in \Emb(G) \times \K^{\times} \times \G(\KX) \, | \, (\lambda, \Psi) \in \DMR_{\times}^{\iota}(\K) \}.
    \end{equation}
\end{definition}

\begin{proposition}
    \label{DMR_x_torsor}
    The group $(\Aut(G) \times \K^{\times}) \ltimes \DMR_{0}^{G}(\K)$ acts freely and transitively on $\DMR_{\times}(\K)$ by
    \begin{equation}
        (\phi, \lambda, \Psi) \cdot (\iota, \nu, \Phi) = (\iota \circ \phi^{-1}, \lambda \nu, \Psi \circledast \eta_{\phi}(\lambda \bullet \Phi)),
    \end{equation}
    for $(\phi, \lambda, \Psi) \in \Aut(G) \times \K^{\times} \times \DMR_{0}^{G}(\K)$ and $(\iota, \nu, \Phi) \in \DMR_{\times}(\K)$.
\end{proposition}
\begin{proof}
    Let $(\iota, \nu, \Phi), (\iota^{\prime}, \nu^{\prime}, \Phi^{\prime}) \in \DMR_{\times}(\K)$. Thanks to Lemma \ref{Emb_torsor}, there is a unique $\phi \in \Aut(G)$ such that $\iota^{\prime} = \iota \circ \phi^{-1}$. Set $\lambda = \nu^{-1} \nu^{\prime}$. Since $\Phi \in \DMR^{\iota}_{\nu}(\K)$, thanks to Lemma \ref{DMR_transfert} and Proposition \ref{eta_phi_DMR}, it follows that $\eta_{\phi}(\lambda \bullet \Phi) \in \DMR^{\iota^{\prime}}_{\nu^{\prime}}(\K)$. Thanks to Proposition \ref{Racinet_main_result}.\ref{Racinet_torsor}, the set $\DMR_{\nu^{\prime}}^{\iota^{\prime}}(\K)$ is a torsor for the action of the group $(\DMR_{0}^{G}(\K), \circledast)$. Therefore, there is a unique $\Psi \in \DMR_{0}^{G}(\K)$ such that $\Psi \circledast \eta_{\phi}(\lambda \bullet \Phi) = \Phi^{\prime}$. In conclusion, there is a unique $(\phi, \lambda, \Psi) \in (\Aut(G) \times \K^{\times}) \ltimes \DMR_{0}^{G}(\K)$ such that
    \[
        (\phi, \lambda, \Psi) \cdot (\iota, \nu, \Phi) = (\iota^{\prime}, \nu^{\prime}, \Phi^{\prime}),
    \]
    which proves the statement.
\end{proof}

\begin{corollary}
    The pair $\Big((\Aut(G) \times \K^{\times}) \ltimes \DMR_0^G(\K), \DMR_{\times}(\K)\Big)$ is a subtorsor of $\Big((\Aut(G) \times \K^{\times}) \ltimes \G(\KX), \Emb(G) \times \K^{\times} \times \G(\KX)\Big)$.
\end{corollary}
\begin{proof}
    It follows from Propositions \ref{EmbGxkxxGkX_torsor} and \ref{DMR_x_torsor}. 
\end{proof}
    \section{The double shuffle group as a stabilizer of a ``de Rham'' coproduct}

In this section, we recall the action of the group $(\G(\KX), \circledast)$ on the algebra-module $\big(\widehat{\W}^{\DR}_G, \widehat{\M}^{\DR}_G\big)$ given in \cite{Yad}. This action enables us in §\ref{actions_on_crossed_product}, to construct an action of the group $(\Aut(G) \times \K^{\times}) \ltimes \G(\KX)$ on the algebra-module $\big(\widehat{\W}^{\DR}_G, \widehat{\M}^{\DR}_G\big)$. This leads us in §\ref{Stabs} to define the stabilizer groups of the coproducts $\widehat{\Delta}^{\W, \DR}_G$ and $\widehat{\Delta}^{\M, \DR}_G$. These stabilizers are related to stabilizers arising from the action of $(\G(\KX), \circledast)$, which contain $\DMR_0^G(\K)$ thanks to \cite{EF0}. Thanks to the main result of \cite{Yad}, we conclude in Corollary \ref{DMR_sub_StabM_sub_StabW} a chain of inclusions involving the former stabilizers and $(\Aut(G) \times \K^{\times}) \ltimes \DMR_0^G(\K)$.

\subsection{Group actions on the algebra-module \texorpdfstring{$\big(\widehat{\W}^{\DR}_G, \widehat{\M}^{\DR}_G\big)$}{(WDR, MDR)}} \label{actions_on_crossed_product}

\subsubsection{Actions of the group \texorpdfstring{$(\G(\KX), \circledast)$}{GKX*}}
For $\Psi \in \G(\KX)$, there is a unique topological $\K$-algebra automorphism $^{\Gamma}\aut^{\V, (1)}_{\Psi}$ of $\widehat{\V}^{\DR}_G$ such that (\cite[Definition 2.3.1]{Yad}
\begin{align}
    e_0 & \mapsto \Gamma_{\Psi}^{-1}(-e_1) \beta(\Psi \otimes 1) \, e_0 \, \beta(\Psi^{-1} \otimes 1) \Gamma_{\Psi}(-e_1) \notag \\
    e_1 & \mapsto \Gamma_{\Psi}^{-1}(-e_1) \, e_1 \, \Gamma_{\Psi}(-e_1) \\
    g & \mapsto \Gamma_{\Psi}^{-1}(-e_1) \beta(\Psi \otimes 1) \, g \, \beta(\Psi^{-1} \otimes 1) \Gamma_{\Psi}(-e_1). \notag
\end{align}
This automorphism restricts to a topological $\K$-algebra automorphism $^{\Gamma}\aut^{\W, (1)}_{\Psi}$ of $\widehat{\W}^{\DR}_G$ (\cite[Proposition-Definition 2.3.2]{Yad}). It is such that the following diagram
\begin{equation}
    \label{diag_GammaautW_GammaautV}
    \begin{tikzcd}
        \widehat{\W}^{\DR}_G \ar["^{\Gamma}\aut^{\W, (1)}_{\Psi}"]{rrr} \ar[hook]{d} &&& \widehat{\W}^{\DR}_G \ar[hook]{d} \\
        \widehat{\V}^{\DR}_G \ar["^{\Gamma}\aut^{\V, (1)}_{\Psi}"]{rrr} &&& \widehat{\V}^{\DR}_G
    \end{tikzcd}
\end{equation}
commutes.
Next, one defines the topological $\K$-module automorphism $^{\Gamma}\aut^{\V, (10)}_{\Psi}$ of $\widehat{\V}^{\DR}_G$ by
\begin{equation}
    ^{\Gamma}\aut^{\V, (10)}_{\Psi} := \,^{\Gamma}\aut^{\V, (1)}_{\Psi} \, \circ \, r_{\Gamma_{\Psi}^{-1}(-e_1)\beta(\Psi \otimes 1)}. 
\end{equation}
This automorphism induces a topological $\K$-module automorphism $^{\Gamma}\aut^{\M, (10)}_{\Psi}$ of $\widehat{\M}^{\DR}_G$ such that the following diagram (\cite[Definition 2.3.4]{Yad})
\begin{equation}
    \label{diag_GammaautV_GammaautM}
    \begin{tikzcd}
        \widehat{\V}^{\DR}_G \ar["^{\Gamma}\aut^{\V, (10)}_{\Psi}"]{rrrr} \ar["- \cdot 1_{\DR}"']{d} &&&& \widehat{\V}^{\DR}_G \ar["- \cdot 1_{\DR}"]{d} \\
        \widehat{\M}^{\DR}_G \ar["^{\Gamma}\aut^{\M, (10)}_{\Psi}"]{rrrr} &&&& \widehat{\M}^{\DR}_G
    \end{tikzcd}
\end{equation}
commutes.

\begin{lemma}[{\cite[Lemma 2.3.5]{Yad}}]
    \label{compat_Gamma_MWV}
    For any $\Psi \in \G(\KX)$, the following pairs are automorphisms in the category $\K\text{-}\alg\text{-}\Mod_{\mathrm{top}}$: 
    \begin{enumerate}[label=(\roman*), leftmargin=*]
        \item \label{rel_Gamma_aut_V10V1} $\left(^{\Gamma}\aut^{\V, (1)}_{\Psi}, \,^{\Gamma}\aut^{\V, (10)}_{\Psi}\right)$ is an automorphism of $(\widehat{\V}^{\DR}_G, \widehat{\V}^{\DR}_G)$.
        \item \label{rel_Gamma_aut_M10V1} $\left(^{\Gamma}\aut^{\V, (1)}_{\Psi}, \,^{\Gamma}\aut^{\M, (10)}_{\Psi}\right)$ is an automorphism of $(\widehat{\V}^{\DR}_G, \widehat{\M}^{\DR}_G)$.
        \item \label{rel_Gamma_aut_M10W1} $\left(^{\Gamma}\aut^{\W, (1)}_{\Psi}, \,^{\Gamma}\aut^{\M, (10)}_{\Psi}\right)$ is an automorphism of $(\widehat{\W}^{\DR}_G, \widehat{\M}^{\DR}_G)$.
    \end{enumerate}
\end{lemma}

The group $(\G(\KX), \circledast)$ acts on $\widehat{\V}^{\DR}_G$ by (\cite[Proposition 2.3.3]{Yad})
\begin{equation}
    \label{GKX_act_VG}
    (\G(\KX), \circledast) \longrightarrow \Aut_{\K{\text -}\alg}^{\mathrm{top}}(\widehat{\V}^{\DR}_G); \quad \Psi \longmapsto \,^{\Gamma}\aut^{\V, (1)}_{\Psi}.
\end{equation}
Thanks to this and the commutativity of Diagram \eqref{diag_GammaautW_GammaautV}, the group $(\G(\KX), \circledast)$ acts on $\widehat{\W}^{\DR}_G$ by (\cite[Proposition 2.3.3]{Yad})
\begin{equation}
    \label{GKX_act_WG}
    (\G(\KX), \circledast) \longrightarrow \Aut_{\K{\text -}\alg}^{\mathrm{top}}(\widehat{\W}^{\DR}_G); \quad \Psi \longmapsto \,^{\Gamma}\aut^{\W, (1)}_{\Psi}.
\end{equation}
On the other hand, the action \eqref{GKX_act_VG} induces an action of $(\G(\KX), \circledast)$ on $\widehat{\V}^{\DR}_G$ by
\begin{equation}
    \label{GKX_act_VG10}
    (\G(\KX), \circledast) \longrightarrow \Aut_{\K{\text -}\Mod}^{\mathrm{top}}(\widehat{\V}^{\DR}_G); \quad \Psi \longmapsto \,^{\Gamma}\aut^{\V, (10)}_{\Psi}.
\end{equation}
Thanks to the commutativity of Diagram \eqref{diag_GammaautV_GammaautM}, the action action \eqref{GKX_act_VG10} induces an action of the group $(\G(\KX), \circledast)$ on $\widehat{\M}^{\DR}_G$ by (\cite[Proposition 2.3.6]{Yad})
\begin{equation}
    \label{GKX_act_MG}
    (\G(\KX), \circledast) \longrightarrow \Aut_{\K{\text-}\Mod}^{\mathrm{top}}(\widehat{\M}^{\DR}_G); \quad \Psi \longmapsto \,^{\Gamma}\aut^{\M, (10)}_{\Psi}.
\end{equation}

\begin{proposition} \ 
    \begin{enumerate}[label=(\roman*), leftmargin=*]
        \item The group $(\G(\KX), \circledast)$ acts on $(\widehat{\V}^{\DR}_G, \widehat{\V}^{\DR}_G)$ by
        \[
            (\G(\KX), \circledast) \longrightarrow \Aut_{\K{\text-}\alg{\text-}\Mod}^{\mathrm{top}}(\widehat{\V}^{\DR}_G, \widehat{\V}^{\DR}_G); \quad \Psi \longmapsto (^{\Gamma}\aut^{\V, (1)}_{\Psi}, \,^{\Gamma}\aut^{\V, (10)}_{\Psi}).
        \]
        \item The group $(\G(\KX), \circledast)$ acts on $(\widehat{\V}^{\DR}_G, \widehat{\M}^{\DR}_G)$ by
        \[
            (\G(\KX), \circledast) \longrightarrow \Aut_{\K{\text-}\alg{\text-}\Mod}^{\mathrm{top}}(\widehat{\V}^{\DR}_G, \widehat{\M}^{\DR}_G); \quad \Psi \longmapsto (^{\Gamma}\aut^{\V, (1)}_{\Psi}, \,^{\Gamma}\aut^{\M, (10)}_{\Psi}).
        \]
        \item \label{GKX_act_WM} The group $(\G(\KX), \circledast)$ acts on $(\widehat{\W}^{\DR}_G,\widehat{\M}^{\DR}_G)$ by
        \[
            (\G(\KX), \circledast) \longrightarrow \Aut_{\K{\text-}\alg{\text-}\Mod}^{\mathrm{top}}(\widehat{\W}^{\DR}_G, \widehat{\M}^{\DR}_G); \quad \Psi \longmapsto (^{\Gamma}\aut^{\W, (1)}_{\Psi}, \,^{\Gamma}\aut^{\M, (10)}_{\Psi}).
        \]
    \end{enumerate}
    \label{GKX_acts}
\end{proposition}
\begin{proof}
    This follows from Lemma \ref{compat_Gamma_MWV} and the fact that \eqref{GKX_act_VG} -- \eqref{GKX_act_MG} define actions.
\end{proof}
\subsubsection{Actions of the group \texorpdfstring{$\K^{\times} \ltimes \G(\KX)$}{kxGKX}}
\begin{definition}
    For $(\lambda, \Psi) \in \K^{\times} \times \G(\KX)$, we define the topological $\K$-algebra-module automorphism $\left(\,^{\Gamma}\aut^{\V, (1)}_{(\lambda, \Psi)}, \,^{\Gamma}\aut^{\V, (10)}_{(\lambda, \Psi)}\right)$ of $\left(\widehat{\V}^{\DR}_G, \widehat{\V}^{\DR}_G\right)$ given by
    \begin{equation*}
        \left(\,^{\Gamma}\aut^{\V, (1)}_{(\lambda, \Psi)}, \,^{\Gamma}\aut^{\V, (10)}_{(\lambda, \Psi)}\right) := \left(\,^{\Gamma}\aut^{\V, (1)}_{\Psi}, \,^{\Gamma}\aut^{\V, (10)}_{\Psi}\right) \circ \big((\lambda \bullet_{\V} -), (\lambda \bullet_{\V} -)\big),
    \end{equation*}
    with $(\lambda \bullet_{\V} -) \in \Aut_{\K{\text -}\alg_{\mathrm{top}}}(\widehat{\V}^{\DR}_G)$ given in \eqref{bulletV_def}.
\end{definition}

\begin{propdef}
    \label{Gamma_aut_alg_mod_2}
    For $(\lambda, \Psi) \in \K^{\times} \times \G(\KX)$, we define the topological $\K$-algebra-module automorphism $\left(\,^{\Gamma}\aut^{\W, (1)}_{(\lambda, \Psi)}, \,^{\Gamma}\aut^{\M, (10)}_{(\lambda, \Psi)}\right)$ of $\left(\widehat{\W}^{\DR}_G, \widehat{\M}^{\DR}_G\right)$ given by
    \begin{equation*}
        \left(\,^{\Gamma}\aut^{\W, (1)}_{(\lambda, \Psi)}, \,^{\Gamma}\aut^{\M, (10)}_{(\lambda, \Psi)}\right) := \left(\,^{\Gamma}\aut^{\W, (1)}_{\Psi}, \,^{\Gamma}\aut^{\M, (10)}_{\Psi}\right) \circ \big((\lambda \bullet_{\W} -), (\lambda \bullet_{\M} -)\big).
    \end{equation*}
    It is such that the following diagrams
    \begin{equation}
        \label{diag_GammaautW2_GammaautV2}
        \begin{tikzcd}
            \widehat{\W}^{\DR}_G \ar["^{\Gamma}\aut^{\W, (1)}_{(\lambda, \Psi)}"]{rrrr} \ar[hook]{d} &&&& \widehat{\W}^{\DR}_G \ar[hook]{d} \\
            \widehat{\V}^{\DR}_G \ar["^{\Gamma}\aut^{\V, (1)}_{(\lambda, \Psi)}"]{rrrr} &&&& \widehat{\V}^{\DR}_G
        \end{tikzcd}
    \end{equation}
    and
    \begin{equation}
        \label{diag_GammaautV2_GammaautM2}
        \begin{tikzcd}
            \widehat{\V}^{\DR}_G \ar["^{\Gamma}\aut^{\V, (10)}_{(\lambda, \Psi)}"]{rrrr} \ar["- \cdot 1_{\DR}"']{d} &&&& \widehat{\V}^{\DR}_G \ar["- \cdot 1_{\DR}"]{d} \\
            \widehat{\M}^{\DR}_G \ar["^{\Gamma}\aut^{\M, (10)}_{(\lambda, \Psi)}"]{rrrr} &&&& \widehat{\M}^{\DR}_G
        \end{tikzcd}
    \end{equation}
    commute.
\end{propdef}
\begin{proof}
    From Propositions \ref{bullet_pairs}.\ref{lambdaW_lambdaM_HAMC} and \ref{GKX_acts}.\ref{GKX_act_WM} we have that the pairs $(\lambda \bullet_{\W} -, \lambda \bullet_{\W} -)$ and $\left(\,^{\Gamma}\aut^{\W, (1)}_{\Psi}, \,^{\Gamma}\aut^{\M, (10)}_{\Psi}\right)$ are morphisms in $\K{\text -}\alg{\text -}\Mod_{\mathrm{top}}$; the composition is then a morphism in $\K{\text -}\alg{\text -}\Mod_{\mathrm{top}}$. Next, the commutativity of the diagrams follows from the commutativity of Diagrams \eqref{diag_lambdaW_lambdaV} and \eqref{diag_GammaautW_GammaautV} and Diagrams \eqref{diag_lambdaV_lambdaM} and \eqref{diag_GammaautV_GammaautM}.       
\end{proof}

\begin{lemma}
    \label{Gammaaut_lambdabullet}
    For $(\lambda, \Psi) \in \K^{\times} \times \G(\KX)$, we have
    \begin{equation}
        ^{\Gamma}\aut^{\V, (1)}_{\lambda \bullet \Psi} \circ (\lambda \bullet_{\V} -) = (\lambda \bullet_{\V} -) \circ \,^{\Gamma}\aut^{\V, (1)}_{\Psi}.   
    \end{equation}
\end{lemma}
\begin{proof}
    Since both sides are given as a composition of $\K$-algebra morphisms of $\widehat{\V}^{\DR}_G$, it is enough to verify this identity on generators. We have
    \begin{align*}
        ^{\Gamma}\aut^{\V, (1)}_{\lambda \bullet \Psi}(\lambda \bullet_{\V} e_0) & = \,^{\Gamma}\aut^{\V, (1)}_{\lambda \bullet \Psi}(\lambda e_0) = \lambda \, ^{\Gamma}\aut^{\V, (1)}_{\lambda \bullet \Psi}(e_0) \\
        & = \lambda \, \Gamma_{\lambda \bullet \Psi}^{-1}(-e_1) \beta(\lambda \bullet \Psi \otimes 1) \, e_0 \, \beta(\lambda \bullet \Psi^{-1} \otimes 1) \Gamma_{\lambda \bullet \Psi}(-e_1) \\
        & = \mbox{\small$\left(\lambda \bullet_{\V} \Gamma_{\Psi}^{-1}(-e_1)\right) \left(\lambda \bullet_{\V} \beta(\Psi \otimes 1)\right) \, \lambda \, e_0 \, \left(\lambda \bullet_{\V} \beta(\Psi^{-1} \otimes 1)\right) \left(\lambda \bullet \Gamma_{\Psi}(-e_1)\right)$} \\
        & = \lambda \bullet_{\V} \left( \Gamma_{\Psi}^{-1}(-e_1) \beta(\Psi \otimes 1) \, e_0 \, \beta(\Psi^{-1} \otimes 1) \Gamma_{\Psi}(-e_1)\right) \\
        & = \lambda \bullet_{\V} \,^{\Gamma}\aut^{\V, (1)}_{\Psi}(e_0),
    \end{align*}
    where the fourth equality comes from the commutativity of Diagram \eqref{Diag_bullets} and Identity \eqref{bullet_Gamma} and the fifth one from the fact that $\lambda \bullet_{\V} -$ is an algebra morphism. Next,
    \begin{align*}
        ^{\Gamma}\aut^{\V, (1)}_{\lambda \bullet \Psi}(\lambda \bullet_{\V} e_1) & = \,^{\Gamma}\aut^{\V, (1)}_{\lambda \bullet \Psi}(\lambda e_1) = \lambda \, ^{\Gamma}\aut^{\V, (1)}_{\lambda \bullet \Psi}(e_1) = \lambda \, \Gamma_{\lambda \bullet \Psi}^{-1}(-e_1) \, e_1 \, \Gamma_{\lambda \bullet \Psi}(-e_1) \\
        & = \left(\lambda \bullet_{\V} \Gamma_{\Psi}^{-1}(-e_1)\right) \, \lambda \, e_1 \, \left(\lambda \bullet \Gamma_{\Psi}(-e_1)\right) \\
        & = \lambda \bullet_{\V} \left( \Gamma_{\Psi}^{-1}(-e_1) \, e_1 \, \Gamma_{\Psi}(-e_1)\right) \\
        & = \lambda \bullet_{\V} \,^{\Gamma}\aut^{\V, (1)}_{\Psi}(e_1),
    \end{align*}
    where the fourth equality comes from Identity \eqref{bullet_Gamma} and the fifth one from the fact that $\lambda \bullet_{\V} -$ is an algebra morphism. Finally, for $g \in G$,
    \begin{align*}
        ^{\Gamma}\aut^{\V, (1)}_{\lambda \bullet \Psi}(\lambda \bullet_{\V} g) & = \,^{\Gamma}\aut^{\V, (1)}_{\lambda \bullet \Psi}(g) = \,^{\Gamma}\aut^{\V, (1)}_{\lambda \bullet \Psi}(g) \\
        & = \,\Gamma_{\lambda \bullet \Psi}^{-1}(-e_1) \beta(\lambda \bullet \Psi \otimes 1) \, g \, \beta(\lambda \bullet \Psi^{-1} \otimes 1) \Gamma_{\lambda \bullet \Psi}(-e_1) \\
        & = \mbox{\small$\left(\lambda \bullet_{\V} \Gamma_{\Psi}^{-1}(-e_1)\right) \left(\lambda \bullet_{\V} \beta(\Psi \otimes 1)\right) \, g \, \left(\lambda \bullet_{\V} \beta(\Psi^{-1} \otimes 1)\right) \left(\lambda \bullet \Gamma_{\Psi}(-e_1)\right)$} \\
        & = \lambda \bullet_{\V} \left( \Gamma_{\Psi}^{-1}(-e_1) \beta(\Psi \otimes 1) \, g \, \beta(\Psi^{-1} \otimes 1) \Gamma_{\Psi}(-e_1)\right) \\
        & = \lambda \bullet_{\V} \,^{\Gamma}\aut^{\V, (1)}_{\Psi}(g),
    \end{align*}
    where the fourth equality comes from the commutativity of Diagram \eqref{Diag_bullets} and Identity \eqref{bullet_Gamma} and the fifth one from the fact that $\lambda \bullet_{\V} -$ is an algebra morphism.
\end{proof}

\begin{corollary}
    \label{GammaautWM_lambdabullet}
    For $(\lambda, \Psi) \in \K^{\times} \times \G(\KX)$, we have
    \begin{enumerate}[label=(\roman*), leftmargin=*]
        \item \label{GammaautW_lambdabullet} $^{\Gamma}\aut^{\W, (1)}_{\lambda \bullet \Psi} \circ (\lambda \bullet_{\W} -) = (\lambda \bullet_{\W} -) \circ \,^{\Gamma}\aut^{\W, (1)}_{\Psi}$.
        \item \label{GammaautM_lambdabullet} $^{\Gamma}\aut^{\M, (10)}_{\lambda \bullet \Psi} \circ (\lambda \bullet_{\M} -) = (\lambda \bullet_{\M} -) \circ \,^{\Gamma}\aut^{\M, (10)}_{\Psi}$.
    \end{enumerate}
\end{corollary}
\begin{proof} \ 
    \begin{enumerate}[label=(\roman*), leftmargin=*]
        \item This follows from Lemma \ref{Gammaaut_lambdabullet} thanks to Lemma \ref{from_lambdaV}.\ref{lambdaW} and to the commutativity of Diagram \eqref{diag_GammaautW_GammaautV}.
        \item This follows from Lemma \ref{Gammaaut_lambdabullet} thanks to Lemma \ref{from_lambdaV}.\ref{lambdaM} and to the commutativity of Diagram \eqref{diag_GammaautV_GammaautM}.
    \end{enumerate}
\end{proof}

\begin{corollary}
    \label{act_gamma_aut_WM}
    The group $\K^{\times} \ltimes \G(\KX)$ acts on $(\widehat{\W}^{\DR}_G, \widehat{\M}^{\DR}_G)$ by
    \begin{equation*}
        \K^{\times} \ltimes \G(\KX) \longrightarrow \Aut_{\K{\text -}\alg{\text -}\Mod}^{\mathrm{top}}\left(\widehat{\W}^{\DR}_G, \widehat{\M}^{\DR}_G\right); (\lambda, \Psi) \longmapsto \left(\,^{\Gamma}\aut^{\W, (1)}_{(\lambda, \Psi)},  \,^{\Gamma}\aut^{\M, (10)}_{(\lambda, \Psi)}\right).
    \end{equation*}
\end{corollary}
\begin{proof}
    Let $(\lambda, \Psi), (\nu, \Phi) \in \K^{\times} \times \G(\KX)$. We have
    \begin{align*}
        ^{\Gamma}\aut^{\W, (1)}_{(\lambda, \Psi) \circledast (\nu, \Phi)} & = \,^{\Gamma}\aut^{\W, (1)}_{(\lambda \nu, \Psi \circledast \lambda \bullet \Phi)} = \,^{\Gamma}\aut^{\W, (1)}_{\Psi \circledast \lambda \bullet \Phi} \, \circ \, (\lambda \nu \bullet_{\V} -) \\
        & = \,^{\Gamma}\aut^{\W, (1)}_{\Psi} \, \circ \,^{\Gamma}\aut^{\W, (1)}_{\lambda \bullet \Phi} \, \circ \, (\lambda \bullet_{\W} -) \, \circ \, (\nu \bullet_{\W} -) \\
        & = \,^{\Gamma}\aut^{\W, (1)}_{\Psi} \, \circ \, (\lambda \bullet_{\W} -) \, \circ \,^{\Gamma}\aut^{\W, (1)}_{\Phi} \, \circ \, (\nu \bullet_{\W} -) \\
        & = \,^{\Gamma}\aut^{\W, (1)}_{(\lambda, \Psi)} \, \circ \,^{\Gamma}\aut^{\W, (1)}_{(\nu, \Phi)}, 
    \end{align*}
    where the third equality comes from the fact that $\Psi \mapsto \,^{\Gamma}\aut^{\V, (1)}_{\Psi}$ and $\lambda \mapsto (\lambda \bullet_{\V} -)$ are group actions and the fourth one from Corollary \ref{GammaautWM_lambdabullet}.\ref{GammaautW_lambdabullet}. Next, we have
    \begin{align*}
        ^{\Gamma}\aut^{\M, (10)}_{(\lambda, \Psi) \circledast (\nu, \Phi)} & = \,^{\Gamma}\aut^{\M, (10)}_{(\lambda \nu, \Psi \circledast \lambda \bullet \Phi)} = \,^{\Gamma}\aut^{\M, (10)}_{\Psi \circledast \lambda \bullet \Phi} \, \circ \, (\lambda \nu \bullet_{\M} -) \\
        & = \,^{\Gamma}\aut^{\M, (10)}_{\Psi} \, \circ \,^{\Gamma}\aut^{\M, (10)}_{\lambda \bullet \Phi} \, \circ \, (\lambda \bullet_{\M} -) \, \circ \, (\nu \bullet_{\M} -) \\
        & = \,^{\Gamma}\aut^{\M, (1)}_{\Psi} \, \circ \, (\lambda \bullet_{\M} -) \, \circ \,^{\Gamma}\aut^{\M, (10)}_{\Phi} \, \circ \, (\nu \bullet_{\M} -) \\
        & = \,^{\Gamma}\aut^{\M, (10)}_{(\lambda, \Psi)} \, \circ \,^{\Gamma}\aut^{\M, (10)}_{(\nu, \Phi)}, 
    \end{align*}
    where the third equality comes from the fact that $\Psi \mapsto \,^{\Gamma}\aut^{\V, (1)}_{\Psi}$ and $\lambda \mapsto (\lambda \bullet_{\V} -)$ are group actions and the fourth one from Corollary \ref{GammaautWM_lambdabullet}.\ref{GammaautM_lambdabullet}.
\end{proof}

\subsubsection{Actions of the group \texorpdfstring{$(\Aut(G) \times \K^{\times}) \ltimes \G(\KX)$}{AutGkxGKX}}
\begin{definition}
    For $(\phi, \lambda, \Psi) \in \Aut(G) \times \K^{\times} \times \G(\KX)$, we define the topological $\K$-algebra-module automorphism $\left(\,^{\Gamma}\aut^{\V, (1)}_{(\phi, \lambda, \Psi)}, \,^{\Gamma}\aut^{\V, (10)}_{(\phi, \lambda, \Psi)}\right)$ of $\left(\widehat{\V}^{\DR}_G, \widehat{\V}^{\DR}_G\right)$ given by
    \[
        \left(\,^{\Gamma}\aut^{\V, (1)}_{(\phi, \lambda, \Psi)}, \,^{\Gamma}\aut^{\V, (10)}_{(\phi, \lambda, \Psi)}\right) := \left(\,^{\Gamma}\aut^{\V, (1)}_{(\lambda, \Psi)}, \,^{\Gamma}\aut^{\V, (10)}_{(\lambda, \Psi)}\right) \circ \left(\eta_{\phi}^{\V}, \eta_{\phi}^{\V}\right),
    \]
    with $\eta^{\V}_{\phi} \in \Aut_{\K{\text -}\alg_{\mathrm{top}}}(\widehat{\V}^{\DR}_G)$ given in \eqref{etaV_def}.
\end{definition}

\begin{propdef}
    For $(\phi, \lambda, \Psi) \in \Aut(G) \times \K^{\times} \times \G(\KX)$, we define the topological $\K$-algebra-module automorphism $\left(\,^{\Gamma}\aut^{\W, (1)}_{(\phi, \lambda, \Psi)}, \,^{\Gamma}\aut^{\M, (10)}_{(\phi, \lambda, \Psi)}\right)$ of $\left(\widehat{\W}^{\DR}_G, \widehat{\M}^{\DR}_G\right)$ given by
    \[
        \left(\,^{\Gamma}\aut^{\W, (1)}_{(\phi, \lambda, \Psi)}, \,^{\Gamma}\aut^{\M, (10)}_{(\phi, \lambda, \Psi)}\right) := \left(\,^{\Gamma}\aut^{\W, (1)}_{(\lambda, \Psi)}, \,^{\Gamma}\aut^{\M, (10)}_{(\lambda, \Psi)}\right) \circ \left(\eta_{\phi}^{\W}, \eta_{\phi}^{\M}\right).
    \]
    It is such that the following diagrams
    \begin{equation}
        \label{diag_GammaautW3_GammaautV3}
        \begin{tikzcd}
            \widehat{\W}^{\DR}_G \ar["^{\Gamma}\aut^{\W, (1)}_{(\phi, \lambda, \Psi)}"]{rrrr} \ar[hook]{d} &&&& \widehat{\W}^{\DR}_G \ar[hook]{d} \\
            \widehat{\V}^{\DR}_G \ar["^{\Gamma}\aut^{\V, (1)}_{(\phi, \lambda, \Psi)}"]{rrrr} &&&& \widehat{\V}^{\DR}_G
        \end{tikzcd}
    \end{equation}
    and
    \begin{equation}
        \label{diag_GammaautV3_GammaautM3}
        \begin{tikzcd}
            \widehat{\V}^{\DR}_G \ar["^{\Gamma}\aut^{\V, (10)}_{(\phi, \lambda, \Psi)}"]{rrrr} \ar["- \cdot 1_{\DR}"']{d} &&&& \widehat{\V}^{\DR}_G \ar["- \cdot 1_{\DR}"]{d} \\
            \widehat{\M}^{\DR}_G \ar["^{\Gamma}\aut^{\M, (10)}_{(\phi, \lambda, \Psi)}"]{rrrr} &&&& \widehat{\M}^{\DR}_G
        \end{tikzcd}
    \end{equation}
    commute.
\end{propdef}
\begin{proof}
    From Proposition \ref{eta_V_etaWM}.\ref{etaW_etaM_HAMC} and Proposition-Definition \ref{Gamma_aut_alg_mod_2}, we have that the pairs $(\eta^{\W}_{\phi}, \eta^{\M}_{\phi})$ and $\left(\,^{\Gamma}\aut^{\W, (1)}_{(\lambda, \Psi)}, \,^{\Gamma}\aut^{\M, (10)}_{(\lambda, \Psi)}\right)$ are morphisms in $\K{\text -}\alg{\text -}\Mod_{\mathrm{top}}$; the composition is then a morphism in $\K{\text -}\alg{\text -}\Mod_{\mathrm{top}}$. Next, the commutativity of the diagrams follows from the commutativity of Diagrams \eqref{diag_etaW_etaV} and \eqref{diag_GammaautW2_GammaautV2} and Diagrams \eqref{diag_etaV_etaM} and \eqref{diag_GammaautV2_GammaautM2}.
\end{proof}

\begin{lemma}
    \label{Gammaaut_eta_etaV_Gammaaut}
    For $(\phi, \lambda, \Psi) \in \Aut(G) \times \K^{\times} \times \G(\KX)$, we have
    \[
        ^{\Gamma}\aut^{\V, (1)}_{(\lambda, \eta_{\phi}(\Psi))} = \eta^{\V}_{\phi} \circ \,^{\Gamma}\aut^{\V, (1)}_{(\lambda, \Psi)} \circ (\eta^{\V}_{\phi})^{-1}.
    \]
\end{lemma}
\begin{proof}
    Since both sides are given as composition of $\K$-algebra morphisms of $\widehat{\V}^{\DR}_G$, it is enough to verify this identity on generators. We have
    \begin{align*}
        & ^{\Gamma}\aut^{\V, (1)}_{(\lambda, \eta_{\phi}(\Psi))}(e_0) = \Ad_{\Gamma_{\eta_{\phi}(\Psi)}^{-1}(-e_1) \beta(\eta_{\phi}(\Psi) \otimes 1)}(\lambda e_0) = \Ad_{\Gamma_{\Psi}^{-1}(-e_1) \eta^{\V}_{\phi}(\beta(\Psi \otimes 1))}(\lambda e_0) \\
        & = \eta^{\V}_{\phi}\left(\Ad_{\Gamma_{\Psi}^{-1}(-e_1) \beta(\Psi \otimes 1)}\big(\lambda (\eta^{\V}_{\phi})^{-1}(e_0)\big)\right) = \eta^{\V}_{\phi}\left(\Ad_{\Gamma_{\Psi}^{-1}(-e_1) \beta(\Psi \otimes 1)}\big(\lambda e_0\big)\right) \\
        & = \eta^{\V}_{\phi} \left(\,^{\Gamma}\aut^{\V, (1)}_{(\lambda, \Psi)}\big(e_0\big)\right) = \eta^{\V}_{\phi} \circ \,^{\Gamma}\aut^{\V, (1)}_{(\lambda, \Psi)} \circ (\eta^{\V}_{\phi})^{-1}(e_0),
    \end{align*}
    where the second equality comes from the identity $\Gamma_{\eta_{\phi}(\Psi)}(-e_1) = \Gamma_{\Psi}(-e_1)$ and from the commutativity of Diagram \eqref{Diag_etaV_eta} and the third one from the fact that $\eta^{\V}_{\phi}$ is an algebra morphism and from the equality $\eta^{\V}_{\phi}(\Gamma_{\Psi}(-e_1)) = \Gamma_{\Psi}(-e_1)$. Next,
    \begin{align*}
        & ^{\Gamma}\aut^{\V, (1)}_{(\lambda, \eta_{\phi}(\Psi))}(e_1) = \Ad_{\Gamma_{\eta_{\phi}(\Psi)}^{-1}(-e_1)}(\lambda e_1) = \Ad_{\Gamma_{\Psi}^{-1}(-e_1)}(\lambda e_1) \\
        & = \eta^{\V}_{\phi}\left(\Ad_{\Gamma_{\Psi}^{-1}(-e_1)}\big(\lambda (\eta^{\V}_{\phi})^{-1}(e_1)\big)\right) = \eta^{\V}_{\phi}\left(\Ad_{\Gamma_{\Psi}^{-1}(-e_1)}\big(\lambda e_1\big)\right) \\
        & = \eta^{\V}_{\phi}\left(\,^{\Gamma}\aut^{\V, (1)}_{(\lambda, \Psi)}(e_1)\right) = \eta^{\V}_{\phi} \circ \,^{\Gamma}\aut^{\V, (1)}_{(\lambda, \Psi)} \circ (\eta^{\V}_{\phi})^{-1}(e_1),
    \end{align*}
    where the second equality comes from the identity $\Gamma_{\eta_{\phi}(\Psi)}(-e_1) = \Gamma_{\Psi}(-e_1)$ and the third one from the fact that $\eta^{\V}_{\phi}$ is an algebra morphism and from the equality $\eta^{\V}_{\phi}(\Gamma_{\Psi}(-e_1)) = \Gamma_{\Psi}(-e_1)$. Finally, for $g \in G$,
    \begin{align*}
        & ^{\Gamma}\aut^{\V, (1)}_{(\lambda, \eta_{\phi}(\Psi))}(g) = \Ad_{\Gamma_{\eta_{\phi}(\Psi)}^{-1}(-e_1) \beta(\eta_{\phi}(\Psi) \otimes 1)}(g) = \Ad_{\Gamma_{\Psi}^{-1}(-e_1) \eta^{\V}_{\phi}(\beta(\Psi \otimes 1))}(g) \\
        & = \eta^{\V}_{\phi}\left(\Ad_{\Gamma_{\Psi}^{-1}(-e_1) \beta(\Psi \otimes 1)}\big( (\eta^{\V}_{\phi})^{-1}(g)\big)\right) = \eta^{\V}_{\phi}\left(\Ad_{\Gamma_{\Psi}^{-1}(-e_1) \beta(\Psi \otimes 1)}\big(\phi^{-1}(g)\big)\right) \\
        & = \eta^{\V}_{\phi} \left(\,^{\Gamma}\aut^{\V, (1)}_{(\lambda, \Psi)}\big(\phi^{-1}(g)\big)\right) = \eta^{\V}_{\phi} \circ \,^{\Gamma}\aut^{\V, (1)}_{(\lambda, \Psi)} \circ (\eta^{\V}_{\phi})^{-1}(g),
    \end{align*}
    where the second equality comes from the identity $\Gamma_{\eta_{\phi}(\Psi)}(-e_1) = \Gamma_{\Psi}(-e_1)$ and from the commutativity of Diagram \eqref{Diag_etaV_eta}, the third one from the fact that $\eta^{\V}_{\phi}$ is an algebra morphism and from the equality $\eta^{\V}_{\phi}(\Gamma_{\Psi}(-e_1)) = \Gamma_{\Psi}(-e_1)$ and the fourth and sixth ones from the fact that $(\eta^{\V}_{\phi})^{-1}(g) = \phi^{-1}(g)$.
\end{proof}

\begin{corollary}
    \label{Gammaaut_eta_etaWM_Gammaaut}
    For $(\phi, \lambda, \Psi) \in \Aut(G) \times \K^{\times} \times \G(\KX)$ we have
    \begin{enumerate}[label=(\roman*), leftmargin=*]
        \item \label{Gammaaut_eta_etaW_Gammaaut} $^{\Gamma}\aut^{\W, (1)}_{(\lambda, \eta_{\phi}(\Psi))} = \eta^{\W}_{\phi} \circ \,^{\Gamma}\aut^{\W, (1)}_{(\lambda, \Psi)} \circ (\eta^{\W}_{\phi})^{-1}$.
        \item \label{Gammaaut_eta_etaM_Gammaaut} $^{\Gamma}\aut^{\M, (10)}_{(\lambda, \eta_{\phi}(\Psi))} = \eta^{\M}_{\phi} \circ \,^{\Gamma}\aut^{\M, (10)}_{(\lambda, \Psi)} \circ (\eta^{\M}_{\phi})^{-1}$.
    \end{enumerate}
\end{corollary}
\begin{proof}
    This follows from Lemma \ref{Gammaaut_eta_etaV_Gammaaut} thanks to Proposition-Definition \ref{Gamma_aut_alg_mod_2} and Lemma \ref{from_etaV}.
\end{proof}

\begin{corollary}
    \label{Gammaaut_act2}
    The group $(\Aut(G) \times \K^{\times}) \ltimes \G(\KX)$ acts on $(\widehat{\W}^{\DR}_G, \widehat{\M}^{\DR}_G)$ by
    \begin{equation*}
        (\Aut(G) \times \K^{\times}) \ltimes \G(\KX) \to \Aut_{\K{\text -}\alg{\text -}\Mod}(\widehat{\W}^{\DR}_G, \widehat{\M}^{\DR}_G); (\phi, \lambda, \Psi) \mapsto \left(\,^{\Gamma}\aut^{\W, (1)}_{(\phi, \lambda, \Psi)}, \,^{\Gamma}\aut^{\M, (10)}_{(\phi, \lambda, \Psi)}\right)
    \end{equation*}
\end{corollary}
\begin{proof}
    Let $(\phi, \lambda, \Psi) , (\phi^{\prime}, \nu, \Phi) \in (\Aut(G) \times \K^{\times}) \rtimes \G(\KX)$. We have
    \begin{align*}
        ^{\Gamma}\aut^{\W, (1)}_{(\phi, \lambda, \Psi) \circledast (\phi^{\prime}, \nu, \Phi)} = & \,^{\Gamma}\aut^{\W, (1)}_{(\phi \circ \phi^{\prime}, \lambda \nu, \Psi \circledast (\eta_{\phi}(\lambda \bullet \Phi)))} = \,^{\Gamma}\aut^{\W, (1)}_{(\lambda \nu, \Psi \circledast (\lambda \bullet \eta_{\phi}(\Phi)))} \circ \eta^{\W}_{\phi \circ \phi^{\prime}} \\
        = & \,^{\Gamma}\aut^{\W, (1)}_{(\lambda, \Psi) \circledast (\nu, \eta_{\phi}(\Phi))} \circ \eta^{\W}_{\phi} \circ \eta^{\W}_{\phi^{\prime}} \\
        = & \,^{\Gamma}\aut^{\W, (1)}_{(\lambda, \Psi)} \circ \,^{\Gamma}\aut^{\W, (1)}_{(\nu, (\eta_{\phi}(\Phi)))} \circ \eta^{\W}_{\phi} \circ \eta^{\W}_{\phi^{\prime}} \\
        = & \,^{\Gamma}\aut^{\W, (1)}_{(\lambda, \Psi)} \circ \eta^{\W}_{\phi} \circ \,^{\Gamma}\aut^{\W, (1)}_{(\nu,\Phi)} \circ (\eta^{\W}_{\phi})^{-1} \circ \eta^{\W}_{\phi} \circ \eta^{\W}_{\phi^{\prime}} \\
        = & \,^{\Gamma}\aut^{\W, (1)}_{(\lambda, \Psi)} \circ \eta^{\W}_{\phi} \circ \,^{\Gamma}\aut^{\W, (1)}_{(\nu,\Phi)} \circ \eta^{\W}_{\phi^{\prime}} \\
        = & \,^{\Gamma}\aut^{\W, (1)}_{(\phi, \lambda, \Psi)} \circ \,^{\Gamma}\aut^{\W, (1)}_{(\phi^{\prime}, \nu, \Phi)},
    \end{align*}
    where the second equality comes from Lemma \ref{eta_bullet}, the third one from the fact that $\eta^{\W} : \Aut(G) \to \Aut_{\K{\text -}\alg}^{\mathrm{top}}(\widehat{\W}^{\DR}_G)$ is a group morphism, the fourth one from Corollary \ref{act_gamma_aut_WM} and the fifth one from Corollary \ref{Gammaaut_eta_etaWM_Gammaaut}.\ref{Gammaaut_eta_etaW_Gammaaut}. Next, we have
    \begin{align*}
        ^{\Gamma}\aut^{\M, (10)}_{(\phi, \lambda, \Psi) \circledast (\phi^{\prime}, \nu, \Phi)} = & \,^{\Gamma}\aut^{\M, (10)}_{(\phi \circ \phi^{\prime}, \lambda \nu, \Psi \circledast (\eta_{\phi}(\lambda \bullet \Phi)))} = \,^{\Gamma}\aut^{\M, (10)}_{(\lambda \nu, \Psi \circledast (\lambda \bullet \eta_{\phi}(\Phi)))} \circ \eta^{\M}_{\phi \circ \phi^{\prime}} \\
        = & \,^{\Gamma}\aut^{\M, (10)}_{(\lambda, \Psi) \circledast (\nu, \eta_{\phi}(\Phi))} \circ \eta^{\M}_{\phi} \circ \eta^{\M}_{\phi^{\prime}} \\
        = & \,^{\Gamma}\aut^{\M, (10)}_{(\lambda, \Psi)} \circ \,^{\Gamma}\aut^{\M, (10)}_{(\nu, (\eta_{\phi}(\Phi)))} \circ \eta^{\M}_{\phi} \circ \eta^{\M}_{\phi^{\prime}} \\
        = & \,^{\Gamma}\aut^{\M, (10)}_{(\lambda, \Psi)} \circ \eta^{\M}_{\phi} \circ \,^{\Gamma}\aut^{\M, (10)}_{(\nu,\Phi)} \circ (\eta^{\M}_{\phi})^{-1} \circ \eta^{\M}_{\phi} \circ \eta^{\M}_{\phi^{\prime}} \\
        = & \,^{\Gamma}\aut^{\M, (10)}_{(\lambda, \Psi)} \circ \eta^{\M}_{\phi} \circ \,^{\Gamma}\aut^{\M, (10)}_{(\nu,\Phi)} \circ \eta^{\M}_{\phi^{\prime}} \\
        = & \,^{\Gamma}\aut^{\M, (10)}_{(\phi, \lambda, \Psi)} \circ \,^{\Gamma}\aut^{\M, (10)}_{(\phi^{\prime}, \nu, \Phi)},
    \end{align*}
    where the second equality comes from Lemma \ref{eta_bullet}, the third one from the fact that $\eta^{\M} : \Aut(G) \to \Aut_{\K{\text -}\Mod}^{\mathrm{top}}(\widehat{\M}^{\DR}_G)$ is a group morphism, the fourth one from Corollary \ref{act_gamma_aut_WM} and the fifth one from Corollary \ref{Gammaaut_eta_etaWM_Gammaaut}.\ref{Gammaaut_eta_etaM_Gammaaut}.
\end{proof}

\subsection{The double shuffle group as a stabilizer of a ``de Rham'' coproduct} \label{Stabs}
\begin{proposition} \ 
    \begin{enumerate}[label=(\roman*), leftmargin=*]
        \item \label{act_on_DeltaW} The group $(\Aut(G) \times \K^{\times}) \ltimes \G(\KX)$ acts on $\mathrm{Cop}_{\K{\text -}\alg_{\mathrm{top}}}(\widehat{\W}^{\DR}_G)$ by
        \begin{equation*}
            (\phi, \lambda, \Psi) \cdot D^{\W} := \left(\,^{\Gamma}\aut^{\W, (1)}_{(\phi, \lambda, \Psi)}\right)^{\otimes 2} \circ D^{\W} \circ \left(\,^{\Gamma}\aut^{\W, (1)}_{(\phi, \lambda, \Psi)}\right)^{-1}.
        \end{equation*}
        \item \label{act_on_DeltaM} The group $(\Aut(G) \times \K^{\times}) \ltimes \G(\KX)$ acts on $\mathrm{Cop}_{\K{\text -}\Mod_{\mathrm{top}}}(\widehat{\M}^{\DR}_G)$ by
        \begin{equation*}
            (\phi, \lambda, \Psi) \cdot D^{\M} := \left(\,^{\Gamma}\aut^{\M, (10)}_{(\phi, \lambda, \Psi)}\right)^{\otimes 2} \circ D^{\M} \circ \left(\,^{\Gamma}\aut^{\M, (10)}_{(\phi, \lambda, \Psi)}\right)^{-1}.
        \end{equation*}
    \end{enumerate}
    \label{act_on_Delta}
\end{proposition}
\begin{proof} \ 
    \begin{enumerate}[label=(\roman*), leftmargin=*]
        \item This is the formula for the pull-back of the action \eqref{from_Aut_to_Cop} with $\mathcal{C} = \K{\text -}\alg_{\mathrm{top}}$ and $O = \widehat{\W}^{\DR}_G$ by the group morphism $(\phi, \lambda, \Psi) \mapsto \,^{\Gamma}\aut^{\W, (1)}_{(\phi, \lambda, \Psi)}$ of Corollary \ref{Gammaaut_act2}.
        \item This is the formula for the pull-back of the action \eqref{from_Aut_to_Cop} with $\mathcal{C} = \K{\text -}\Mod_{\mathrm{top}}$ and $O = \widehat{\M}^{\DR}_G$ by the group morphism $(\phi, \lambda, \Psi) \mapsto \,^{\Gamma}\aut^{\M, (10)}_{(\phi, \lambda, \Psi)}$ of Corollary \ref{Gammaaut_act2}.
    \end{enumerate}    
\end{proof}

\begin{definition} \ 
    \begin{enumerate}[label=(\roman*), leftmargin=*]
        \item \label{Stab_DeltaW} We denote $\Stab_{(\Aut(G) \times \K^{\times}) \ltimes \G(\KX)}(\widehat{\Delta}^{\W, \DR}_{G})(\K)$ the stabilizer subgroup of the coproduct $\widehat{\Delta}^{\W, \DR}_{G} \in \mathrm{Cop}_{\K{\text -}\alg_{\mathrm{top}}}(\widehat{\W}^{\DR}_G)$ for the action of Proposition \ref{act_on_Delta}.\ref{act_on_DeltaW}. Namely,
        \begin{align*}
            & \Stab_{(\Aut(G) \times \K^{\times}) \ltimes \G(\KX)}(\widehat{\Delta}^{\W, \DR}_{G})(\K) := \\
            & \mbox{\footnotesize$\left\{ (\phi, \lambda, \Psi) \in (\Aut(G) \times \K^{\times}) \ltimes \G(\KX) \, | \, \left( {^{\Gamma}\aut^{\W, (1)}_{(\phi, \lambda, \Psi)}} \right)^{\otimes 2} \circ \widehat{\Delta}^{\W, \DR}_{G} = \widehat{\Delta}^{\W, \DR}_{G} \circ \,^{\Gamma}\aut^{\W, (1)}_{(\phi, \lambda, \Psi)} \right\}$}. \notag
        \end{align*}
        \item \label{Stab_DeltaM} We denote $\Stab_{(\Aut(G) \times \K^{\times}) \ltimes \G(\KX)}(\widehat{\Delta}^{\M, \DR}_{G})(\K)$ the stabilizer subgroup of the coproduct $\widehat{\Delta}^{\M, \DR}_{G} \in \mathrm{Cop}_{\K{\text -}\Mod_{\mathrm{top}}}(\widehat{\M}^{\DR}_G)$ for the action of Proposition \ref{act_on_Delta}.\ref{act_on_DeltaM}. Namely,
        \begin{align*}
            & \Stab_{(\Aut(G) \times \K^{\times}) \ltimes \G(\KX)}(\widehat{\Delta}^{\M, \DR}_{G})(\K) := \\
            & \mbox{\footnotesize$\left\{ (\phi, \lambda, \Psi) \in (\Aut(G) \times \K^{\times}) \ltimes \G(\KX) \, | \, \left( {^{\Gamma}\aut^{\M, (10)}_{(\phi, \lambda, \Psi)}} \right)^{\otimes 2} \circ \widehat{\Delta}^{\M, \DR}_{G} = \widehat{\Delta}^{\M, \DR}_{G} \circ \,^{\Gamma}\aut^{\M, (10)}_{(\phi, \lambda, \Psi)} \right\}$}. \notag
        \end{align*}
    \end{enumerate}
    \label{Stab_DeltaWM}
\end{definition}

Since $(\G(\KX), \circledast)$ is a subgroup of $(\Aut(G) \times \K^{\times}) \ltimes \G(\KX)$, the actions of Proposition \ref{act_on_Delta} induce actions of $(\G(\KX), \circledast)$ on the spaces $\mathrm{Cop}_{\K{\text -}\alg_{\mathrm{top}}}(\widehat{\W}^{\DR}_G)$ and $\mathrm{Cop}_{\K{\text -}\Mod_{\mathrm{top}}}(\widehat{\M}^{\DR}_G)$. This enables us to define the stabilizer subgroups (see \cite[(2.29) and (2.31)]{Yad})
\begin{equation*}
    \mbox{\footnotesize$\Stab_{\G(\KX)}(\widehat{\Delta}^{\W, \DR}_{G})(\K) := \left\{ \Psi \in \G(\KX) \, | \, \left( {^{\Gamma}\aut^{\W, (1)}_{\Psi}} \right)^{\otimes 2} \circ \widehat{\Delta}^{\W, \DR}_{G} = \widehat{\Delta}^{\W, \DR}_{G} \circ \,^{\Gamma}\aut^{\W, (1)}_{\Psi} \right\}$}.
\end{equation*}
and
\begin{equation*}
    \mbox{\footnotesize$\Stab_{\G(\KX)}(\widehat{\Delta}^{\M, \DR}_{G})(\K) := \left\{ \Psi \in \G(\KX) \, | \, \left( {^{\Gamma}\aut^{\M, (10)}_{\Psi}} \right)^{\otimes 2} \circ \widehat{\Delta}^{\M, \DR}_{G} = \widehat{\Delta}^{\M, \DR}_{G} \circ \,^{\Gamma}\aut^{\M, (10)}_{\Psi} \right\}$}.
\end{equation*}
Moreover, we have (\cite[Theorem 2.4.1]{Yad})
\begin{equation}
    \label{Stab_inclusion}
    \Stab_{\G(\KX)}(\widehat{\Delta}^{\M, \DR}_{G})(\K) \subset \Stab_{\G(\KX)}(\widehat{\Delta}^{\W, \DR}_{G})(\K).
\end{equation}

\begin{proposition}
    \label{semidirect_Stabs}
    We have
    \begin{enumerate}[label=(\roman*), leftmargin=*]
        \item \label{semidirect_StabW} $\Stab_{(\Aut(G) \times \K^{\times}) \ltimes \G(\KX)}(\widehat{\Delta}^{\W, \DR}_{G})(\K) = (\Aut(G) \times \K^{\times}) \ltimes \Stab_{\G(\KX)}(\widehat{\Delta}^{\W, \DR}_{G})(\K)$.
        \item \label{semidirect_StabM} $\Stab_{(\Aut(G) \times \K^{\times}) \ltimes \G(\KX)}(\widehat{\Delta}^{\M, \DR}_{G})(\K) = (\Aut(G) \times \K^{\times}) \ltimes \Stab_{\G(\KX)}(\widehat{\Delta}^{\M, \DR}_{G})(\K)$.
    \end{enumerate}
\end{proposition}
\noindent It is a consequence of the following general lemma
\begin{lemma}
    \label{semi_direct_universal_property}
    Let us consider the semidirect product group $H \ltimes R$. If $K$ is a subgroup of $H \ltimes R$ containing $H$, then
    \[
        K = H \ltimes (K \cap R).
    \]
\end{lemma}
\begin{proof}[Proof of Proposition \ref{semidirect_Stabs}]
    Set $\mathcal{X} = \W$ or $\M$. We use Lemma \ref{semi_direct_universal_property} where $H = \Aut(G) \times \K^{\times}$, $R = \G(\KX)$ and $K = \Stab_{(\Aut(G) \times \K^{\times}) \ltimes \G(\KX)}(\widehat{\Delta}^{\mathcal{X}, \DR}_{G})(\K)$. We have that
    \[
        K \cap R = \Stab_{(\Aut(G) \times \K^{\times}) \ltimes \G(\KX)}(\widehat{\Delta}^{\mathcal{X}, \DR}_{G})(\K) \cap \G(\KX) = \Stab_{\G(\KX)}(\widehat{\Delta}^{\mathcal{X}, \DR}_{G})(\K).
    \]
    Additionally, $\Stab_{(\Aut(G) \times \K^{\times}) \ltimes \G(\KX)}(\widehat{\Delta}^{\mathcal{X}, \DR}_{G})(\K)$ contains $\Aut(G) \times \K^{\times}$.
    Therefore, the condition of Lemma \ref{semi_direct_universal_property} is met and the result then follows.
\end{proof}

\noindent Finally, one has from \cite[Theorem 1.2]{EF0} that
\begin{equation}
    \label{DMRsubsetStab}
    \DMR_0^G(\K) = \{ \Psi \in \Stab_{\G(\KX)}(\widehat{\Delta}^{\M, \DR}_{G})(\K) \, | \, (\Psi | x_0) = (\Psi | x_1) = 0 \}.
\end{equation}
This establishes an inclusion $\DMR_0^G(\K) \subset \Stab_{\G(\KX)}(\widehat{\Delta}^{\M, \DR}_{G})(\K)$ of subgroups of $(\G(\KX), \circledast)$. We then have the following result:

\begin{corollary}
    \label{DMR_sub_StabM_sub_StabW}
    We have
    \begin{eqnarray*}
        (\Aut(G) \times \K^{\times}) \ltimes \DMR_0^G(\K) \subset & \Stab_{(\Aut(G) \times \K^{\times}) \ltimes \G(\KX)}(\widehat{\Delta}^{\M, \DR}_{G})(\K) \\
        & \cap \\
        & \Stab_{(\Aut(G) \times \K^{\times}) \ltimes \G(\KX)}(\widehat{\Delta}^{\W, \DR}_{G})(\K)
    \end{eqnarray*}
\end{corollary}
\begin{proof}
    Thanks to Proposition \ref{semidirect_Stabs}.\ref{semidirect_StabM}, we have
    \begin{equation}
        \label{semidirect_StabM_equality}
        \mbox{\small$\Stab_{(\Aut(G) \times \K^{\times}) \ltimes \G(\KX)}(\widehat{\Delta}^{\M, \DR}_{G})(\K) = (\Aut(G) \times \K^{\times}) \ltimes \Stab_{\G(\KX)}(\widehat{\Delta}^{\M, \DR}_{G})(\K)$}.
    \end{equation}
    On the other hand, using equality \eqref{DMRsubsetStab}, we obtain
    \begin{equation}
        \label{semidirect_DMR_subset_StabM}
        (\Aut(G) \times \K^{\times}) \ltimes \DMR_0^G(\K) \subset (\Aut(G) \times \K^{\times}) \ltimes \Stab_{\G(\KX)}(\widehat{\Delta}^{\M, \DR}_{G})(\K).
    \end{equation}
    From equality \eqref{semidirect_StabM_equality} and inclusion \eqref{semidirect_DMR_subset_StabM}, we obtain the inclusion
    \[
        (\Aut(G) \times \K^{\times}) \ltimes \DMR_0^G(\K) \subset \Stab_{(\Aut(G) \times \K^{\times}) \ltimes \G(\KX)}(\widehat{\Delta}^{\M, \DR}_{G})(\K),
    \]
    which is the wanted first inclusion. For the second inclusion, thanks to inclusion \eqref{Stab_inclusion}, we have that
    \[
        (\Aut(G) \times \K^{\times}) \ltimes \Stab_{\G(\KX)}(\widehat{\Delta}^{\M, \DR}_{G})(\K) \subset (\Aut(G) \times \K^{\times}) \ltimes \Stab_{\G(\KX)}(\widehat{\Delta}^{\W, \DR}_{G})(\K).
    \]
    Thanks to Proposition \ref{semidirect_Stabs}, this inclusion implies that
    \[
        \Stab_{(\Aut(G) \times \K^{\times}) \ltimes \G(\KX)}(\widehat{\Delta}^{\M, \DR}_{G})(\K) \subset \Stab_{(\Aut(G) \times \K^{\times}) \ltimes \G(\KX)}(\widehat{\Delta}^{\W, \DR}_{G})(\K).
    \]
\end{proof}
    \section{Construction of ``Betti'' coproducts} \label{Betti_side}
In this section, we construct a ``Betti'' version of the double shuffle formalism.
The relevant algebras and modules are introduced in §\ref{Betti_algmod} : (i) an algebra $\widehat{\V}^{\B}_N$ defined as the inverse limit of an algebra $\V^{\B}_N$ endowed with a suitable filtration; (ii) an algebra-module $(\widehat{\W}^{\B}_N, \widehat{\M}^{\B}_N)$ composed of a subalgebra $\widehat{\W}^{\B}_N$ of $\widehat{\V}^{\B}_N$ and a $\K$-module $\widehat{\M}^{\B}_N$ which has a $\widehat{\V}^{\B}_N$-module structure inducing a free rank one $\widehat{\W}^{\B}_N$-module structure on it. In proposition \ref{compat_iso_MWV}, we construct algebra-module isomorphisms $(\iso^{\W, \iota}, \iso^{\M, \iota})$ from $(\widehat{\W}^{\B}_N, \widehat{\M}^{\B}_N)$ to $(\widehat{\W}^{\DR}_G, \widehat{\M}^{\DR}_G)$ indexed by $\iota \in \Emb(G)$. This gives rise to a family of algebra-module isomorphisms $\left(\,^{\Gamma}\comp^{\W, (1)}_{(\iota, \lambda, \Psi)}, \,^{\Gamma}\comp^{\M, (10)}_{(\iota, \lambda, \Psi)}\right)$ indexed by elements $(\iota, \lambda, \Psi) \in \Emb(G) \times \K^{\times} \times \G(\KX)$.
In §\ref{Betti_coprod}, we show that the transport by this isomorphism of the ``de Rham'' pair of coproducts $(\widehat{\Delta}^{\W, \DR}_G, \widehat{\Delta}^{\M, \DR}_G)$ is independent of the element $(\iota, \lambda, \Psi) \in \DMR_{\times}(\K)$ (see Theorem \ref{Delta_B_N}). This is derived from the chain of inclusions of Corollary \ref{DMR_sub_StabM_sub_StabW} and from the torsor structure of $\DMR_{\times}(\K)$ over $(\Aut(G) \times \K^{\times}) \ltimes \DMR_0^G(\K)$ (see Proposition \ref{DMR_x_torsor}). The resulting pair of coproducts is denoted $(\widehat{\Delta}^{\W, \B}_N, \widehat{\Delta}^{\M, \B}_N)$ and equips $\widehat{\W}^{\B}_N$ and $\widehat{\M}^{\B}_N$ with Hopf algebra and coalgebra structures respectively (see Corollary \ref{WM_categories}).

\subsection{The topological algebra-module \texorpdfstring{$(\widehat{\W}^{\B}_N, \widehat{\M}^{\B}_N)$}{(WBN, MBN)}} \label{Betti_algmod}

\subsubsection{The filtered algebra \texorpdfstring{$\V^{\B}_N$}{VBN}}
Let \index{$F_2$}$F_2$ be the free group generated by two elements denoted \index{$X_0$}$X_0$ and \index{$X_1$}$X_1$. We consider the group morphism $F_2 \to \mu_N$ given by $X_0 \mapsto \zeta_N$ and $X_1 \mapsto 1$; where \index{$\zeta_N$}$\zeta_N := e^{\frac{i 2\pi}{N}}$.
\begin{lemma}
    The group $\ker(F_2 \to \mu_N)$ is isomorphic to the free group of rank $N+1$ denoted \index{$F_{N+1}$}$F_{N+1}$.
    \label{FN1}
\end{lemma}
\noindent In order to prove this, we use the following result:
\begin{proposition}[Nielsen-Schreier Theorem, see {\cite[Theorem 3]{Ste}}]
    Let $F$ be a free group on a non-empty set $X$ and let $H$ be a subgroup of $F$. Let $\sigma : H \backslash F \to F$ be a section of the canonical projection $F \to H \backslash F$ such that $T := \sigma(H \backslash F)$ is stable under left prefixation. Then $H$ is freely generated by
    \[
        \left\{ tx (\overline{tx})^{-1} \, | \, (t, x) \in T \times X \text{ and } tx (\overline{tx})^{-1} \neq 1 \right\},
    \]
    where for $g \in F$, $\bar{g}$ the image of $g$ under the composition $F \to H \backslash F \overset{\sigma}{\to} F$. 
\end{proposition}

\begin{proof}[Proof of Lemma \ref{FN1}]
    We apply the Nielsen-Schreier Theorem for $X=\{X_0, X_1\}$, $F=F_2$, $H=\ker(F_2 \to \mu_N)$ and $\sigma : \ker(F_2 \to \mu_N) \backslash F_2 \simeq \mu_N \to F_2$ where the first map is the isomorphism induced by the surjective morphism $F_2 \to \mu_N$ and the second map given by $e^{i\frac{2n\pi}{N}} \mapsto X_0^n$ for $n \in \llbracket 0, N-1 \rrbracket$. Therefore, we have $T = \left\{X_0^n, n \in \llbracket 0, N-1 \rrbracket \right\}$. The theorem then states that $\ker(F_2 \to \mu_N)$ is freely generated by:
    \begin{itemize}[leftmargin=*]
        \item \(X_0^n X_0 (\overline{X_0^n X_0})^{-1} = X_0^{n+1} (\overline{X_0^{n+1}})^{-1}=
        \begin{cases}
            X_0^{n+1} (X_0^{n+1})^{-1} = 1 & \text{ if } n \in \llbracket 0, N-2 \rrbracket \\
            X_0^{N} 1^{-1} = X_0^N & \text{ if } n = N-1 
        \end{cases}\)
        \item \(X_0^n X_1 (\overline{X_0^n X_1})^{-1} = X_0^n X_1 (X_0^n)^{-1} = X_0^n X_1 X_0^{-n} \) 
    \end{itemize}
    Finally, $\ker(F_2 \to \mu_N)$ is freely generated by the $N+1$ elements
    \[
        \left\{ X_0^N, (X_0^n X_1 X_0^{-n})_{n \in \llbracket 0, N-1 \rrbracket} \right\}.
    \]
    Moreover, if we denote $\left(\widetilde{X}_0, \left(\widetilde{X}_{\zeta_N^{n}}\right)_{n \in \llbracket 0, N-1\rrbracket}\right)$ the generators of the free group $F_{N+1}$ of rank $N+1$, one checks that correspondence
    \[
        \widetilde{X}_0 \mapsto X_0^{N}, \, \widetilde{X}_{\zeta_N^{n}} \mapsto X_0^n X_1 X_0^{-n} \text{ for } n \in \llbracket 0, N-1 \rrbracket  
    \]
    defines a free group isomorphism from $F_{N+1}$ to $\ker(F_2 \to \mu_N)$.
\end{proof}
\noindent We then obtain the following short exact sequence
\begin{equation}
    \label{exact_seq}
    \{1\} \to F_{N+1} \to F_2 \to \mu_N \to \{1\} 
\end{equation}

Next, let \index{$\sigma$}$\sigma : \mu_N \to F_2$ be the set-theoretic section of $F_2 \to \mu_N$ given by $e^{\frac{i2n\pi}{N}} \mapsto X_0^n$ for $n \in \llbracket 0, N-1 \rrbracket$. Thanks to the exact sequence (\ref{exact_seq}) we obtain a bijection \index{$\Sigma$}
\begin{equation}
    \label{bijection_Sigma}
    \Sigma : \mu_N \times F_{N+1} \to F_2, \quad (\zeta, x) \mapsto \sigma(\zeta) x;
\end{equation}
where $F_{N+1}$ is seen as $\ker(F_2 \to \mu_N) \subset F_2$ thanks to Lemma \ref{FN1}. \newline
The set $\mu_N \times F_{N+1}$ is equipped with a right $F_{N+1}$-set structure by
\[
    (\zeta, x) * y := (\zeta, xy), \text{ for } (\zeta, x) \in \mu_N \times F_{N+1} \text{ and } y \in F_{N+1}.
\]
The group $F_2$ is also equipped with a right $F_{N+1}$-set structure given by
\[
    x*y:= xy, \text{ for } x \in F_2 \text{ and } y \in F_{N+1};
\]
where $F_{N+1}$ is seen as $\ker(F_2 \to \mu_N) \subset F_2$ thanks to Lemma \ref{FN1}. One checks that (\ref{bijection_Sigma}) upgrades to a right $F_{N+1}$-set isomorphism. \newline
Let us consider the tensor functor \index{$\K(-)$}
\[
    \K(-) : \{\text{right } F_{N+1}\text{-sets}\} \longrightarrow \{\text{right } \K F_{N+1}\text{-modules}\}
\]
taking $X$ to $\K X$, the set of finitely supported maps $X \to \K$. Applying this functor to the isomorphism of right $F_{N+1}$-sets (\ref{bijection_Sigma}), one obtains the right $\K F_{N+1}$-module isomorphism \index{$\K\Sigma$}
\begin{equation}
    \label{Def_kSigma}
    \K\Sigma : \K\mu_N \otimes \K F_{N+1} \to \K F_2, 
\end{equation}
where both the source and the target are equipped with the right $\K F_{N+1}$-module structure given by the right $F_{N+1}$-set structure on $\mu_N \times F_{N+1}$ and $F_2$ respectively.

Let us denote \index{$\mathcal{I}$}$\mathcal{I} := \ker(\K F_2 \to \K\mu_N)$ where $\K F_2 \to \K\mu_N$ is the $\K$-algebra morphism induced from the group morphism $F_2 \to \mu_N$. Then $\mathcal{I}$ is a two-sided ideal of $\K F_2$. In particular, $\mathcal{I}$ is a right $\K F_{N+1}$-module. \newline
Let \index{$\varepsilon$}$\varepsilon : \K F_{N+1} \to \K$ be the augmentation morphism of the group algebra $\K F_{N+1}$. It is equipped with a right regular $\K F_{N+1}$-module structure.

\begin{lemma} \ 
    \begin{enumerate}[label=(\roman*), leftmargin=*]
        \item \label{ker_id_eps} The $\K$-module isomorphism $\K\Sigma : \K\mu_N \otimes \K F_{N+1} \to \K F_2$ sets up a right $\K F_{N+1}$-module isomorphism of $\mathcal{I}$ with $\K\mu_N \otimes \ker(\varepsilon)$.
        \item \label{Igenerators} The ideal $\mathcal{I}$ is linearly generated by $\sigma(\zeta) (x-1)$ where $\zeta \in \mu_N$ and $x \in F_{N+1}$. 
    \end{enumerate}
    \label{IdealI}
\end{lemma}
\begin{proof} \ 
    \begin{enumerate}[label=(\roman*), leftmargin=*]
        \item The following commutative diagram of $F_{N+1}$-set morphisms
        \[\begin{tikzcd}
            \mu_N \times F_{N+1} \ar["\Sigma"]{rr} \ar["p_1"']{rd} & & F_2 \ar{ld} \\
            & \mu_N 
        \end{tikzcd}\]
        induces a commutative diagram of $\K F_{N+1}$-module morphisms
        \[\begin{tikzcd}
            \K\mu_N \otimes \K F_{N+1} \ar["\K\Sigma"]{rr} \ar["\mathrm{id} \otimes \varepsilon"']{rd} & & \K F_2 \ar{ld} \\
            & \K\mu_N
        \end{tikzcd}\]
        One checks that the associated group algebra morphism of the first projection $p_1 : \mu_N \times F_{N+1} \to \mu_N$ is identified with $\mathrm{id} \otimes \varepsilon : \K\mu_N \otimes \K F_{N+1} \to \K\mu_N$ thanks to the identification $\K\mu_N \otimes \K F_{N+1} \simeq \K(\mu_N \times F_{N+1})$. Therefore, the ideal $\mathcal{I}$ is mapped by the isomorphism $\K\Sigma$ to the ideal $\ker(\mathrm{id} \otimes \varepsilon) = \K\mu_N \otimes \ker(\varepsilon)$.
        \item Since $\varepsilon : \K F_{N+1} \to \K$ is the augmentation morphism, its kernel is generated by elements $x-1$ with $x \in F_{N+1}$. Therefore, taking the image of the generators by $\K\Sigma$, we obtain generators of the ideal $\mathcal{I}$ as announced.
    \end{enumerate}
\end{proof}

\begin{propdef}
    \label{VBfiltration}
    Let $\V^{\B}_N$ be the group algebra of $F_2$ over $\K$ endowed with the filtration \index{$\mathcal{F}^m \V^{\B}_N$}
    \begin{equation*}
        \mathcal{F}^m \V^{\B}_N = \mathcal{I}^m,
    \end{equation*}
    for $m \in \N$, where $\mathcal{I}^m$ is the $m^{\text{th}}$-power of the ideal $\mathcal{I}$ with the convention that $\mathcal{I}^0=\V^{\B}_N$. The filtration $(\mathcal{F}^m \V^{\B}_N)_{m \in \N}$ is an algebra filtration.
\end{propdef}
\begin{proof}
    Immediate.
\end{proof}

\begin{lemma}
    \label{filtr}
    Let $m \in \N$. The $\K$-module isomorphism $\K\Sigma : \K\mu_N \otimes \K F_{N+1} \to \K F_2$ sets up a right $\K F_{N+1}$-module isomorphism of $\mathcal{F}^m \V^{\B}_N$ with $\K\mu_N \otimes (\K F_{N+1})_0^m$,
    where \index{$(\K F_{N+1})_0$}$(\K F_{N+1})_0$ is the augmentation ideal of the group algebra $\K F_{N+1}$.
\end{lemma}
\begin{proof}
    If $m=0$, we have $\K\mu_N \otimes \K F_{N+1} \simeq \K(\mu_N \times F_{N+1}) \xrightarrow{\K\Sigma} \V^{\B}_N = \mathcal{F}^0 \V^{\B}_N$. \newline
    Next, if $m=1$, we have
    \[
        \mathcal{F}^1 \V^{\B}_N = \mathcal{I} \simeq \K\mu_N \otimes \ker(\varepsilon) = \K\mu_N \otimes (\K F_{N+1})_0,
    \]
    where the identification is given by Lemma \ref{IdealI} \ref{ker_id_eps}. \newline
    Now, let $m \geq 2$. Since $\K\mu_N \otimes \K F_{N+1}$ is a right $\K F_{N+1}$-module, we have that
    \begin{equation}
        \label{right_decomp}
        \K\mu_N \otimes (\K F_{N+1})_0^m = \left(\K\mu_N \otimes (\K F_{N+1})_0\right) \cdot (\K F_{N+1})_0^{m-1}.
    \end{equation}
    The composition $\K\mu_N \otimes \K F_{N+1} \simeq \K( \mu_N \times F_{N+1}) \xrightarrow{\K\Sigma} \V^{\B}_N$ is a right $\K F_{N+1}$-module isomorphism which, combined with the identification $\mathcal{I} \simeq \K\mu_N \otimes (\K F_{N+1})_0$ and equality (\ref{right_decomp}), gives us
    \[
        \K\mu_N \otimes (\K F_{N+1})_0^m \simeq \mathcal{I} \cdot (\K F_{N+1})_0^{m-1},
    \]
    where $(\K F_{N+1})_0^{m-1}$ is seen as a subset of $\K F_{N+1} = \K\ker(F_2 \to \mu_N) \subset \K F_2$. \newline
    It remains to show that $\mathcal{I} \cdot (\K F_{N+1})_0^{m-1} = \mathcal{I}^m$. First, since $(\K F_{N+1})_0 \subset \mathcal{I}$, we have $\mathcal{I} \cdot (\K F_{N+1})_0^{m-1} \subset \mathcal{I}^m$.
    Conversely, thanks to Lemma \ref{IdealI} \ref{Igenerators}, $\mathcal{I}^m$ is linearly generated by elements \index{$\Pi$}
    \[
        \Pi((\zeta_1, x_1), \dots, (\zeta_m, x_m)) := \sigma(\zeta_1) (x_1 - 1) \cdots \sigma(\zeta_m) (x_m - 1)
    \]
    with $(\zeta_1, x_1), \dots, (\zeta_m, x_m) \in \mu_N \times F_{N+1}$. Moreover, we have that
    \begin{align*}
        & \Pi((\zeta_1, x_1), \dots, (\zeta_m, x_m)) = \sigma(\zeta_1) \cdots \sigma(\zeta_m) \left(\Ad_{\sigma(\zeta_m)^{-1} \cdots \sigma(\zeta_{2})^{-1}}(x_1) - 1\right) & \\
        & \left(\Ad_{\sigma(\zeta_m)^{-1} \cdots \sigma(\zeta_{3})^{-1}}(x_2) - 1\right) \cdots \left(\Ad_{\sigma(\zeta_m)^{-1}}(x_{m-1}) - 1\right) \, (x_m - 1). &
    \end{align*}
    Next, since $F_{N+1}$ is a normal subgroup of $F_2$, we have that
    \[
        \left(\Ad_{\sigma(\zeta_m)^{-1} \cdots \sigma(\zeta_{3})^{-1}}(x_2) - 1\right) \cdots \left(\Ad_{\sigma(\zeta_m)^{-1}}(x_{m-1}) - 1\right) \, (x_m - 1) \in (\K F_{N+1})_0^{m-1}.
    \]
    In addition, thanks to Lemma \ref{IdealI} \ref{Igenerators}, we have
    \[
        \sigma(\zeta_1) \cdots \sigma(\zeta_m) \left(\Ad_{\sigma(\zeta_m)^{-1} \cdots \sigma(\zeta_{2})^{-1}}(x_1) - 1\right) \in \K F_2 \cdot (\K F_{N+1})_0. 
    \]
    Since $(\K F_{N+1})_0 \subset \mathcal{I}$, it follows that $\K F_2 \cdot (\K F_{N+1})_0 \subset \K F_2 \cdot \mathcal{I}$ and since $\mathcal{I}$ is a two-sided ideal of $\K F_2$, we have $\K F_2 \cdot \mathcal{I} = \mathcal{I}$. Therefore,
    \[
        \sigma(\zeta_1) \cdots \sigma(\zeta_m) \left(\Ad_{\sigma(\zeta_m)^{-1} \cdots \sigma(\zeta_{2})^{-1}}(x_1) - 1\right) \in \mathcal{I},
    \]
    and then $\Pi((\zeta_1, x_1), \dots, (\zeta_m, x_m)) \in \mathcal{I} \cdot (\K F_{N+1})_0^{m-1}$, thus proving the wanted inclusion.
\end{proof}

\subsubsection{The topological algebra \texorpdfstring{$\widehat{\V}^{\B}_N$}{hatVBN}}
The decreasing filtration $(\mathcal{F}^m \V^{\B}_N)_{m \in \N}$ given in Proposition-Definition \ref{VBfiltration} induces an algebra morphism $\V^{\B}_N / \mathcal{F}^{m+1} \V^{\B}_N \to \V^{\B}_N / \mathcal{F}^m \V^{\B}_N$. One defines
\begin{definition}
    We denote \index{$\widehat{\V}^{\B}_N$}
    \[
        \widehat{\V}^{\B}_N := \lim_{\longleftarrow} \V^{\B}_N/\mathcal{F}^m \V^{\B}_N
    \]
    the inverse limit of the system $\left(\V^{\B}_N/\mathcal{F}^m \V^{\B}_N, \V^{\B}_N / \mathcal{F}^{m+1} \V^{\B}_N \to \V^{\B}_N / \mathcal{F}^m \V^{\B}_N\right)$.
\end{definition}
\noindent The algebra $\widehat{\V}^{\B}_N$ is equipped with the filtration \index{$\mathcal{F}^m \widehat{\V}^{\B}_N$} \(\displaystyle \mathcal{F}^m \widehat{\V}^{\B}_N := \lim_{\longleftarrow} \mathcal{F}^m \V^{\B}_N / \mathcal{F}^{\max(m,l)} \V^{\B}_N\). When equipped with the topology defined by this filtration, $\widehat{\V}^{\B}_N$ is a complete separated topological algebra. \newline
Recall that $\K F_{N+1}$ is a group algebra equipped with a filtration given by the powers of its augmentation ideal. Let us denote \index{$\widehat{\K F_{N+1}}$}$\widehat{\K F_{N+1}}$ the completion of this group algebra with respect to this filtration.  
\begin{lemma} \ 
    \begin{enumerate}[label=(\roman*), leftmargin=*]
        \item \label{hatKSigma_circ_1} The $\K$-algebra morphism $\K\Sigma \circ (1 \otimes -) : \K F_{N+1} \to \V^{\B}_N$ gives rise to a topological $\K$-algebra morphism $\widehat{\K F_{N+1}} \to \widehat{\V}^{\B}_N$.
        \item \label{hatKSigma} The $\K$-module morphism $\K\Sigma : \K\mu_N \otimes \K F_{N+1} \to \V^{\B}_N$ gives rise to an isomorphism of topological right $\widehat{\K F_{N+1}}$-module $\widehat{\K\Sigma} : \K\mu_N \otimes \widehat{\K F_{N+1}} \to \widehat{\V}^{\B}_N$.
        \item The $\K$-algebra morphism $\widehat{\K F_{N+1}} \to \widehat{\V}^{\B}_N$ is injective. 
    \end{enumerate}
\end{lemma}
\begin{proof} \ 
    \begin{enumerate}[label=(\roman*), leftmargin=*]
        \item This follows from the fact that the $\K$-algebra morphism $\K\Sigma \circ (1 \otimes -) : \K F_{N+1} \to \V^{\B}_N$ is compatible with filtrations, which follows from Lemma \ref{filtr}.
        \item This follows from the fact that $\K\Sigma : \K\mu_N \otimes \K F_{N+1} \to \V^{\B}_N$ is an isomorphism of filtered right module over $ \K F_{N+1}$ (see \eqref{Def_kSigma}).
        \item By \ref{hatKSigma_circ_1}, the topological $\K$-algebra morphism $\widehat{\K F_{N+1}} \to \widehat{\V}^{\B}_N$ is equal to the composition $\widehat{\K\Sigma} \circ (1 \otimes -) : \widehat{\K F_{N+1}} \to \widehat{\V}^{\B}_N$. The map $1 \otimes - : \widehat{\K F_{N+1}} \to \K\mu_N \otimes \widehat{\K F_{N+1}}$ is trivially injective and $\widehat{\K\Sigma} : \K\mu_N \otimes \widehat{\K F_{N+1}} \to \widehat{\V}^{\B}_N$ is injective by \ref{hatKSigma}. This implies that their composition is injective, implying the claim.  
    \end{enumerate}
\end{proof}

\begin{propdef}
    \label{isoViota}
    Let $\iota \in \Emb(G)$. There is a unique topological algebra isomorphism \index{$\iso^{\V, \iota}$}$\iso^{\V, \iota} : \widehat{\V}^{\B}_N \to \widehat{\V}^{\DR}_G$ given by
    \[
        X_0 \mapsto \exp\left(\frac{1}{N} e_0\right) g_{\iota}; \quad \text{and} \quad X_1 \mapsto \exp(e_1),
    \]
    where $g_{\iota} = \iota^{-1}(e^{\frac{i2\pi}{N}})$.
\end{propdef}
\begin{proof}
    Recall that the set $\mathrm{Mor}_{\K{\text -}\alg}(\K F_2, \widehat{\V}^{\DR}_G)$ is identified with $\mathrm{Mor}_{\mathrm{grp}}\Big(F_2, \big(\widehat{\V}^{\DR}_G\big)^{\times}\Big)$. As a consequence, there is an algebra morphism $\V^{\B}_N \to \widehat{\V}^{\DR}_G$ given by
    \[
        X_0 \mapsto \exp\left(\frac{1}{N} e_0\right) g_{\iota} \quad \text{and} \quad X_1 \mapsto \exp(e_1)
    \]
    since the images of $X_0$ and $X_1$ are invertible.
    Composing the $\K$-algebra morphism $\V^{\B}_N \to \widehat{\V}^{\DR}_G$ with the $\K$-module isomorphism $\K\Sigma : \K\mu_N \otimes \K F_{N+1} \to \V^{\B}_N$ and the inverse of the $\K$-algebra isomorphism $\beta : \KX \rtimes G \to \widehat{\V}^{\DR}_G$ respectively from the left and from the right, we obtain a $\K$-module morphism
    \begin{equation}
        \label{tensor_morphism}
        \K\mu_N \otimes \K F_{N+1} \to \KX \rtimes G.
    \end{equation}
    One checks that morphism \eqref{tensor_morphism} is a right module morphism over the $\K$-algebra morphism $\K F_{N+1} \to \KX$ given by
    \[
        \widetilde{X}_0 \mapsto \exp(x_0) \text{ and } \widetilde{X}_{\zeta_N^n} \mapsto \exp\left(\frac{n}{N} x_0\right) \exp(-x_{g_{\iota}^n}) \exp\left(-\frac{n}{N} x_0\right), \text{ for } n \in \llbracket 0, N-1 \rrbracket.
    \]
    In addition, $(\zeta_N^l \otimes 1)_{l \in \llbracket 0, N-1 \rrbracket}$ and $\left(\exp\left(\frac{l}{N} x_0\right) \otimes g_{\iota}^l\right)_{l \in \llbracket 0, N-1 \rrbracket}$ are bases of $\K\mu_N \otimes \K F_{N+1}$ and $\KX \rtimes G$ respectively and the morphism \eqref{tensor_morphism} induces the following bijection between the bases
    \begin{equation}
        \label{zeta_to_expg}
        \zeta_N^l \otimes 1 \mapsto \exp\left(\frac{l}{N} x_0\right) \otimes g_{\iota}^l, \text{ for } l \in \llbracket 0, N-1 \rrbracket.
    \end{equation}
    Furthermore, there is a topological $\K$-algebra isomorphism $\widehat{\K F_{N+1}} \to \KX$ such that the following diagram
    \[\begin{tikzcd}
        \K F_{N+1} \ar{rr} \ar[hook]{rd} & & \KX \\
        & \widehat{\K F_{N+1}} \ar[]{ru} &
    \end{tikzcd}\]
    commutes, where $\K F_{N+1} \hookrightarrow \widehat{\K F_{N+1}}$ is the canonical $\K$-algebra morphism. \newline
    Indeed, such an isomorphism is obtained by composing the topological $\K$-algebra isomorphism $\widehat{\K F_{N+1}} \to \KX$ obtained from \cite[Example A2.12]{Qui} and the topological $\K$-algebra automorphism of $\widehat{\K F_{N+1}}$ given by
    \[
        \widetilde{X}_0 \mapsto \widetilde{X}_0 \text{ and } \widetilde{X}_{\zeta_{N}^n} \mapsto \Ad_{\exp\left(\frac{n}{N}\log(\widetilde{X}_0)\right)}(\widetilde{X}_{\zeta_{N}^n}^{-1}) \text{ for } n \in \llbracket 0, N-1 \rrbracket. 
    \]
    On the other hand, one checks that $\K\mu_N \otimes \widehat{\K F_{N+1}}$ is a free right $\widehat{\K F_{N+1}}$-module with basis $(\zeta_N^l \otimes 1)_{l \in \llbracket 0, N-1 \rrbracket}$ and recall that $\KX \rtimes G$ is a free right $\KX$-module with basis $\left(\exp\left(\frac{l}{N} x_0\right) \otimes g_{\iota}^l\right)_{l \in \llbracket 0, N-1 \rrbracket}$. Therefore, there is a unique module isomorphism $\K\mu_N \otimes \widehat{\K F_{N+1}} \to \KX \rtimes G$ over the $\K$-algebra isomorphism $\widehat{\K F_{N+1}} \to \KX$ which extends bijection \eqref{zeta_to_expg} between bases.
    Therefore, the restriction to the bases of the following diagram
    \begin{equation}
        \label{Diag_hatB_hatDR_tensor}
        \begin{tikzcd}
            \K\mu_N \otimes \K F_{N+1} \ar{rr} \ar[hook]{rd} & & \KX \rtimes G \\
            & \K\mu_N \otimes \widehat{\K F_{N+1}} \ar[]{ru} &
        \end{tikzcd}
    \end{equation}
    commutes, where $\K\mu_N \otimes \K F_{N+1} \to \K\mu_N \otimes \widehat{\K F_{N+1}}$ is the tensor product of the identity of $\K\mu_N$ with $\K F_{N+1} \hookrightarrow \widehat{\K F_{N+1}}$. This implies that the diagram commutes. \newline
    Next, by composing the $\K$-module isomorphism $\K\mu_N \otimes \widehat{\K F_{N+1}} \to \KX \rtimes G$ from the left and from the right with the isomorphisms $\K\Sigma^{-1} : \widehat{\V}^{\B}_N \to \K\mu_N \otimes \widehat{\K F_{N+1}}$ and $\beta : \KX \rtimes G \to \widehat{\V}^{\B}_N$ respectively, we obtain a $\K$-module isomorphism $\widehat{\V}^{\B}_N \to \widehat{\V}^{\DR}_G$. \newline
    Let us prove that this $\K$-module isomorphism is a $\K$-algebra isomorphism. It is, therefore, enough to show that it is a $\K$-algebra morphism. Let us consider the following prism
    \[\begin{tikzcd}
        \mathbf{k}\mu_N \otimes \mathbf{k}F_{N+1} \arrow[dd, "\mathbf{k}\Sigma"'] \arrow[rr] \arrow[rd, hook] & & \mathbf{k}\langle\langle X \rangle\rangle \rtimes G \arrow[dd, "\beta"] \\
        & \mathbf{k}\mu_N \otimes \widehat{\mathbf{k}F_{N+1}} \arrow[ru] \arrow[dd, "\widehat{\mathbf{k}\Sigma}"', near start] & \\
        \mathcal{V}^{\mathrm{B}}_N \arrow[rd, hook] \arrow[rr] & & \widehat{\mathcal{V}}^{\mathrm{DR}}_G \\
        & \widehat{\mathcal{V}}^{\mathrm{B}}_N \arrow[ru] &
    \end{tikzcd}\]
    The left, right and middle squares commute by definition of $\widehat{\K\Sigma}$, $\widehat{\V}^{\B}_N \to \widehat{\V}^{\DR}_G$ and $\K\mu_N \otimes \K F_{N+1} \to \KX \rtimes G$ respectively and the upper triangle is Diagram \eqref{Diag_hatB_hatDR_tensor}, so is commutative. Additionally, the arrows going from the upper triangle to the lower triangle are isomorphisms. Therefore, the lower triangle is commutative.
    The restriction of the topological $\K$-module isomorphism $\widehat{\V}^{\B}_N \to \widehat{\V}^{\DR}_G$ to $\V^{\B}_N$ is an algebra morphism, which by the density of $\V^{\B}_N$ in $\widehat{\V}^{\B}_N$ implies that $\widehat{\V}^{\B}_N \to \widehat{\V}^{\DR}_G$ is a topological $\K$-algebra morphism and therefore a topological $\K$-algebra isomorphism.
    Finally, the commutativity of the triangle also implies that the $\K$-algebra isomorphism $\widehat{\V}^{\B}_N \to \widehat{\V}^{\DR}_G$ is as announced.
\end{proof}

\begin{proposition}
    \label{eta_phi_iso}
    Let $\iota \in \Emb(G)$ and $\phi \in \Aut(G)$. We have
    \[
        \iso^{\V, \iota \circ \phi^{-1}} = \eta_{\phi}^{\V} \circ \iso^{\V, \iota},
    \]
    with $\eta^{\V}_{\phi} \in \Aut_{\K{\text -}\alg_{\mathrm{top}}}(\widehat{\V}^{\DR}_G)$ given in \eqref{etaV_def}.
\end{proposition}
\begin{proof}
    Since both sides are given as a composition of topological $\K$-algebra morphisms, let us the equality by checking on the family of topological generators:
    \[
        \mathrm{iso}^{\V, \iota \circ \phi^{-1}}(X_1) = \exp(e_1) = \eta_{\phi}^{\V} \circ \mathrm{iso}^{\V, \iota}(X_1)
    \]
    and
    \begin{align*}
        \mathrm{iso}^{\V, \iota \circ \phi^{-1}}(X_0) & = \exp\left(\frac{1}{N} e_0\right) g_{\iota \circ \phi^{-1}}
        = \exp\left(\frac{1}{N} e_0\right) \phi(g_{\iota}) = \eta_{\phi}^{\V}\left( \exp\left(\frac{1}{N} e_0\right) g_{\iota} \right) \\
        & = \eta_{\phi}^{\V} \circ \mathrm{iso}^{\V, \iota}(X_0)
    \end{align*}
\end{proof}

\subsubsection{The filtered algebra \texorpdfstring{$\W^{\B}_N$}{WBN}}
\begin{propdef}
    Let us denote \index{$\W^{\B}_N$}
    \begin{equation}
        \W^{\B}_N := \K \oplus \V^{\B}_N (X_1 - 1).
    \end{equation}
    It is a subalgebra of $\V^{\B}_N$ endowed with the filtration \index{$\mathcal{F}^m \W^{\B}_N$}
    \begin{equation}
        \label{WBFiltration}
        \mathcal{F}^m \W^{\B}_N := \W^{\B}_N \cap \mathcal{F}^m \V^{\B}_N 
    \end{equation}
    for $m \in \N$. The filtration $(\mathcal{F}^m \W^{\B}_N)_{m \in \N}$ is an algebra filtration.
\end{propdef}
\begin{proof}
    Immediate.
\end{proof}

\begin{lemma}
    \label{FmWB}
    For $m \in \N^{\ast}$, we have
    \begin{multicols}{2}
    \begin{enumerate}[label=(\roman*), leftmargin=*]
        \item \label{FmWB_equality} $\mathcal{F}^m \W^{\B}_N = \mathcal{F}^m \V^{\B}_N \cap \V^{\B}_N (X_1 - 1)$.
        \item \label{FmWB_to_FmVB}$\mathcal{F}^m \W^{\B}_N = \mathcal{F}^{m-1} \V^{\B}_N (X_1 - 1)$.
    \end{enumerate}
    \end{multicols}
\end{lemma}
\begin{proof} \ 
    \begin{enumerate}[label=(\roman*), leftmargin=*]
        \item Let $m \in \N^{\ast}$. We have
        \begin{align*}
            \mathcal{F}^m \W^{\B}_N = & \mathcal{F}^m \V^{\B}_N \cap \left(\K \oplus \V^{\B}_N (X_1 - 1)\right) \\
            = & \mathcal{F}^m \V^{\B}_N \cap \left(\ker(\V^{\B}_N \to \K) \cap (\K \oplus \V^{\B}_N(X_1 - 1))\right) \\
            = & \mathcal{F}^m \V^{\B}_N \cap \V^{\B}_N (X_1 - 1),
        \end{align*}
        where the second equality follows from the inclusion $\mathcal{F}^m \V^{\B}_N \subset \ker(\V^{\B}_N \to \K)$ since
        \begin{equation}
            \label{kerVBkmuk}
            \mathcal{F}^m \V^{\B}_N = \ker(\V^{\B}_N \to \K\mu_N)^m \subset \ker(\V^{\B}_N \to \K\mu_N) \subset \ker(\V^{\B}_N \to \K),
        \end{equation}
        where the last inclusion of \eqref{kerVBkmuk} is a consequence of the fact that $\V^{\B}_N \to \K$ is the composition $\V^{\B}_N \to \K\mu_N \to \K$ (the maps with target $\K$ being the augmentation morphisms).
        The third equality follows from
        \[
            \ker(\V^{\B}_N \to \K) \cap (\K \oplus \V^{\B}_N(X_1 - 1)) = \V^{\B}_N(X_1 - 1)
        \]
        which, in turn, follows from the fact that $\ker(\V^{\B}_N \to \K) \cap (\K \oplus \V^{\B}_N(X_1 - 1))$ is the kernel of the composed map $\K \oplus \V^{\B}_N(X_1 - 1) \subset \V^{\B}_N \to \K$ which is the identity on $\K$ and takes $\V^{\B}_N(X_1 - 1)$ to $0$. Its kernel is therefore $\V^{\B}_N(X_1 - 1)$.
        \item Recall from Lemma \ref{filtr} that, for $m \in \N^{\ast}$, the $\K$-module isomorphism $\K\Sigma : \K\mu_N \otimes \K F_{N+1} \to \V^{\B}_N$ induces an isomorphism
        \begin{equation}
            \label{FmVB}
            \mathcal{F}^m \V^{\B}_N \simeq \K\mu_N \otimes (\K F_{N+1})_0^m,
        \end{equation}
        where $(\K F_{N+1})_0$ is the augmentation ideal of the group algebra $\K F_{N+1}$. The isomorphism $\K\Sigma$ also induces an isomorphism
        \[
            \V^{\B}_N (X_1 -1) \simeq \K\mu_N \otimes \K F_{N+1} (\widetilde{X}_{\zeta_N^0} - 1).   
        \]
        Thanks to Lemma \ref{FmVB} \ref{FmWB_equality}, this induces the isomorphism
        \[
            \mathcal{F}^m \W^{\B}_N \simeq \K\mu_N \otimes \left((\K F_{N+1})_0^m \cap \K F_{N+1}(\widetilde{X}_{\zeta_N^0} - 1)\right).
        \]
        Next, thanks to \cite[Proposition 6.2.6]{Wei}, we have a $\K F_{N+1}$-module isomorphism $(\K F_{N+1})^{\oplus (N+1)} \to (\K F_{N+1})_0$. This isomorphism induces the following isomorphisms
        \[
            \K F_{N+1} \oplus \{0\}^{\oplus N} \simeq \K F_{N+1} (X_{\zeta_N^0}-1) \text{ and } (\K F_{N+1})_0^{m-1})^{\oplus (N+1)} \simeq (\K F_{N+1})_0^m,
        \]
        where for the latter one we use the fact that $(\K F_{N+1})_0^m = (\K F_{N+1})_0^{m-1} (\K F_{N+1})_0$ and the fact that $(\K F_{N+1})_0^{m-1}$ is an ideal of $\K F_{N+1}$. \newline
        On the other hand, using the inclusion $(\K F_{N+1})_0^{m-1} \subset \K F_{N+1}$ and the isomorphism $\K F_{N+1} \oplus \{0\}^{\oplus N} \simeq \K F_{N+1} (X_{\zeta_N^0}-1)$, one obtains
        \[
            (\K F_{N+1})_0^{m-1})^{\oplus (N+1)} \cap \left(\K F_{N+1} \oplus \{0\}^{\oplus N}\right) = (\K F_{N+1})_0^{m-1} \oplus \{0\}^{\oplus N}.
        \]
        Finally, one checks that the isomorphism $(\K F_{N+1})^{\oplus (N+1)} \to (\K F_{N+1})_0$ induces an isomorphism
        \[
            (\K F_{N+1})_0^{m-1} \oplus \{0\}^{\oplus N} \simeq (\K F_{N+1})_0^{m-1} (\widetilde{X}_{\zeta_N^0} - 1)
        \]
        and using \eqref{FmVB} for $m$ replaced by $m-1$, together with the fact that $\K\Sigma$ intertwines right multiplication by $X_1 - 1$ on $\V^{\B}_N$ with the tensor product of the identity on $\K\mu_N$ with right multiplication by $\widetilde{X}_{\zeta_N^0} - 1$ on $\K F_{N+1}$ implies 
        \[
            \K\mu_N \otimes (\K F_{N+1})_0^{m-1} (\widetilde{X}_{\zeta_N^0} - 1) \simeq \mathcal{F}^{m-1} \V^{\B}_N (X_1 - 1),
        \]
        thus proving the wanted result. 
    \end{enumerate}
\end{proof}

\subsubsection{The topological algebra \texorpdfstring{$\widehat{\W}^{\B}_N$}{hatWBN}}
The decreasing filtration $(\mathcal{F}^m \W^{\B}_N)_{m \in \N}$ given in (\ref{WBFiltration}) induces an algebra morphism $\W^{\B}_N / \mathcal{F}^{m+1} \W^{\B}_N \to \W^{\B}_N / \mathcal{F}^m \W^{\B}_N$.
\begin{definition}
    We denote \index{$\widehat{\W}^{\B}_N$}
    \[
        \widehat{\W}^{\B}_N := \lim_{\longleftarrow} \W^{\B}_N/\mathcal{F}^m \W^{\B}_N
    \]
    the inverse limit of the projective system $(\W^{\B}_N/\mathcal{F}^m \W^{\B}_N, \W^{\B}_N / \mathcal{F}^{m+1} \W^{\B}_N \to \W^{\B}_N / \mathcal{F}^m \W^{\B}_N)$.
\end{definition}
The algebra $\widehat{\W}^{\B}_N$ equipped with filtration \index{$\mathcal{F}^m \widehat{\W}^{\B}_N$}
\[
    \mathcal{F}^m \widehat{\W}^{\B}_N := \lim_{\longleftarrow} \mathcal{F}^m \W^{\B}_N / \mathcal{F}^{\max(m,l)} \W^{\B}_N
\]
and endowed with the topology defined by this filtration is a complete separated topological algebra.

\begin{lemma}
    \label{hatW_subset_hatV}
    The $\K$-algebra inclusion $\W^{\B}_N \subset \V^{\B}_N$ gives rise to an injective morphism of topological $\K$-algebras $\widehat{\W}^{\B}_N \to \widehat{\V}^{\B}_N$.
\end{lemma}
\begin{proof}
    This follows from the compatibility of the inclusion $\W^{\B}_N \subset \V^{\B}_N$ with filtrations and the fact that the filtration on $\W^{\B}_N$ is induced by that of $\V^{\B}_N$. 
\end{proof}

\begin{proposition}
    \label{WB}
    The topological algebra $\widehat{\W}^{\B}_N$ is isomorphic to the topological subalgebra $\K \oplus \widehat{\V}^{\B}_N (X_1 - 1)$ of $\widehat{\V}^{\B}_N$. 
\end{proposition}
\begin{proof}
    This will be done following this program:
    \begin{description}[leftmargin=1pt]
        \item[Step 1] Construction of the topological $\K$-module $\widehat{\W}^{\B}_{N, +}$. \newline
        Let us define a $\K$-submodule \index{$\W^{\B}_{N, +}$}$\W^{\B}_{N, +} := \V^{\B}_N (X_1 - 1) \subset \W^{\B}_N$. It is equipped with the filtration \index{$\mathcal{F}^m \W^{\B}_{N, +}$}
        \[
            \mathcal{F}^m \W^{\B}_{N, +} := \W^{\B}_{N, +} \cap \mathcal{F}^{m} \W^{\B}_N \text{ , for } m \in \N
        \]
        induced by the inclusion $\W^{\B}_{N, +} \subset \W^{\B}_N$. Denote as follows the associated inverse limit \index{$\widehat{\W}^{\B}_{N, +}$}
        \[
            \widehat{\W}^{\B}_{N, +} := \lim_{\longleftarrow} \W^{\B}_{N, +}/\mathcal{F}^m \W^{\B}_{N, +}.
        \]
        One checks that the $\K$-module inclusion $\W^{\B}_{N, +} \subset \W^{\B}_N$ is compatible with the filtrations, which induces a morphism of topological $\K$-modules $\widehat{\W}^{\B}_{N, +} \to \widehat{\W}^{\B}_N$. As the filtration of $\W^{\B}_{N, +}$ is induced by that of $\W^{\B}_{N}$, this morphism is injective. Thanks to Lemma \ref{hatW_subset_hatV}, we then have a chain of injections
        \begin{equation}
            \label{hat_inclusion}
            \widehat{\W}^{\B}_{N, +} \hookrightarrow \widehat{\W}^{\B}_N \hookrightarrow \widehat{\V}^{\B}_N.
        \end{equation}
        On the other hand, for any $m \in \N^{\ast}$, we have
        \[
            \mathcal{F}^m \W^{\B}_{N, +} = \W^{\B}_{N, +} \cap \mathcal{F}^{m} \W^{\B}_N = \V^{\B}_N (X_1 - 1) \cap \mathcal{F}^{m-1} \V^{\B}_N (X_1 - 1) = \mathcal{F}^{m-1} \V^{\B}_N (X_1 - 1),
        \]
        where the second equality comes from Lemma \ref{FmWB} \ref{FmWB_to_FmVB}. Therefore, for any $m \in \N$,
        \begin{equation}
            \label{FmWB_plus}
            \mathcal{F}^m \W^{\B}_{N, +} =
            \begin{cases}
                \V^{\B}_N (X_1 - 1) & \text{ if } m = 0 \\
                \mathcal{F}^{m-1} \V^{\B}_N (X_1 - 1) & \text{ otherwise}
            \end{cases}
        \end{equation}
        Moreover, let us notice that $\W^{\B}_N = \K \, \oplus \, \W^{\B}_{N, +}$. Using \eqref{FmWB_plus} we obtain
        \begin{align*}
            \mathcal{F}^0 \W^{\B}_N & = \K \, \oplus \, \mathcal{F}^0 \W^{\B}_{N, +}; \\
            \mathcal{F}^m \W^{\B}_N & = \mathcal{F}^m \W^{\B}_{N, +}, \text{ for } m \in \N^{\ast}. 
        \end{align*}
        These equalities induce the following topological $\K$-algebra isomorphism
        \begin{equation}
            \label{hatWB_hatWB+}
            \widehat{\W}^{\B}_N = \lim_{\longleftarrow} \W^{\B}_N / \mathcal{F}^m \W^{\B}_N \simeq \K \, \oplus \, \lim_{\longleftarrow} \W^{\B}_{N, +} / \mathcal{F}^m \W^{\B}_{N, +} = \K \, \oplus \, \widehat{\W}^{\B}_{N, +}.
        \end{equation}
        where, on the right, the algebra structure is defined by the conditions that $1 \in \K$ is a unit and that the inclusion $\widehat{\W}^{\B}_{N, +} \subset \K \oplus \widehat{\W}^{\B}_{N, +}$ is a non-unital algebra morphism.
        \item[Step 2] The existence of a topological $\K$-module morphism $\widehat{\varphi} : \widehat{\V}^{\B}_N \to \widehat{\W}^{\B}_{N, +}$ such that the triangle
        \begin{equation}
            \label{triangle_hat_varphi}
            \begin{tikzcd}
                \widehat{\W}^{\B}_{N, +} \ar[hook]{rrrr} & & & & \widehat{\V}^{\B}_N \\
                & & \widehat{\V}^{\B}_N \ar["\widehat{\varphi}"]{llu} \ar["- \cdot (X_1 - 1)"']{rru} & 
            \end{tikzcd}
        \end{equation}
        commutes.
        First, let us consider the $\K$-module morphism $\varphi : \V^{\B}_N \to \W^{\B}_{N, +}$ given by $v \mapsto v (X_1 - 1)$. For any $m \in \N^{\ast}$, one has
        \[
            \varphi(\mathcal{F}^m \V^{\B}_N) = \mathcal{F}^m \V^{\B}_N (X_1-1) \subset \mathcal{F}^{m-1} \V^{\B}_N (X_1-1) = \mathcal{F}^m \W^{\B}_{N, +},
        \]
        where the first equality follows from the definition of $\varphi$, the inclusion follows from decreasing character of $(\mathcal{F}^m \V^{\B}_N)_{m \in \N}$ and the last equality follows from \eqref{FmWB_plus}. One also has
        \[
            \varphi(\mathcal{F}^0 \V^{\B}_N) = \V^{\B}_N (X_1-1) = \mathcal{F}^0 \W^{\B}_{N, +}.
        \]
        This implies that the morphism $\varphi : \V^{\B}_N \to \W^{\B}_{N, +}$ is compatible with filtrations. This induces a $\K$-module morphism $\varphi_m : \V^{\B}_N / \mathcal{F}^m \V^{\B}_N \to \W^{\B}_{N, +} / \mathcal{F}^m \W^{\B}_{N, +}$. \newline
        In the following prism
        \[\begin{tikzcd}
            \W^{\B}_{N, +} \ar[hook]{rr} \ar[two heads]{dd} & & \V^{\B}_N \ar[two heads]{dd} \\
            & \V^{\B}_N \ar[two heads]{dd} \ar["\varphi"]{lu} \ar["- \cdot (X_1 - 1)"']{ru} & \\
            \W^{\B}_{N, +} / \mathcal{F}^m \W^{\B}_{N, +} \ar[hook]{rr} & & \V^{\B}_N / \mathcal{F}^m \V^{\B}_N \\
            & \V^{\B}_N / \mathcal{F}^m \V^{\B}_N \ar["\varphi_m"]{lu} \ar["- \cdot (X_1 - 1)"']{ru} &
        \end{tikzcd}\]
        the upper triangle commutes by definition of $\varphi : \V^{\B}_N \to \W^{\B}_{N, +}$ and all the squares commute thanks to the compatibility of the maps $\varphi : \V^{\B}_N \to \W^{\B}_{N, +}$, $- \cdot (X_1 - 1) : \V^{\B}_N \to \V^{\B}_N$ and $\W^{\B}_{N, +} \subset \V^{\B}_{N}$ with filtrations. Therefore, thanks to the surjectivity of the projection $\V^{\B}_N \to \V^{\B}_N / \mathcal{F}^m \V^{\B}_N$, the lower triangle commutes.
        As a consequence, the morphism $\varphi : \V^{\B}_N \to \W^{\B}_{N, +}$ induces a morphism of topological $\K$-modules $\widehat{\varphi} : \widehat{\V}^{\B}_N \to \widehat{\W}^{\B}_{N, +}$ such that Diagram \eqref{triangle_hat_varphi} commutes. Finally, the commutativity of the latter diagram implies
        \begin{align}
            \widehat{\V}^{\B}_N (X_1 - 1) & = \mathrm{Im}\big(- \cdot (X_1 - 1)\big) \notag \\
            \label{mult_subset_W+} & = \mathrm{Im}\left(\widehat{\V}^{\B}_N \xrightarrow{\widehat{\varphi}} \widehat{\W}^{\B}_{N, +} \hookrightarrow \widehat{\V}^{\B}_N\right) \subset \mathrm{Im}\left(\widehat{\W}^{\B}_{N, +} \hookrightarrow \widehat{\V}^{\B}_N\right).
        \end{align}
        \item[Step 3] The existence of a topological $\K$-module morphism $\tilde{\phi} : \widehat{\W}^{\B}_{N, +} \to \widehat{\V}^{\B}_N$ such that the triangle
        \begin{equation}
            \label{triangle_tilde_phi}
            \begin{tikzcd}
                \widehat{\W}^{\B}_{N, +} \ar[hook]{rrrr} \ar["\tilde{\phi}"']{rrd} & & & & \widehat{\V}^{\B}_N \\
                & & \widehat{\V}^{\B}_N \ar["- \cdot (X_1 - 1)"']{rru} & & 
            \end{tikzcd}
        \end{equation}
        commutes. First, one notices that $\varphi : \V^{\B}_N \to \W^{\B}_{N, +}$ is a surjective $\K$-module morphism. It is injective thanks to the integral domain status of the algebra $\V^{\B}_N$. Therefore, the map $\varphi : \V^{\B}_N \to \W^{\B}_{N, +}$ is a $\K$-module isomorphism whose inverse will be denoted $\phi : \W^{\B}_{N, +} \to \V^{\B}_N$. Thanks to \eqref{FmWB_plus}, the $\K$-module isomorphism $\phi : \W^{\B}_{N, +} \to \V^{\B}_N$ restricts to an isomorphism $\mathcal{F}^m \W^{\B}_{N, +} \to \mathcal{F}^{m-1} \V^{\B}_N$, for any $m \in \N^{\ast}$. This induces a $\K$-module isomorphism $\phi_m : \W^{\B}_{N, +} / \mathcal{F}^m \W^{\B}_{N, +} \to \V^{\B}_N / \mathcal{F}^{m-1} \V^{\B}_N$, for any $m \in \N^{\ast}$ and, via a prism similar to the one of Step 2, one checks that the following triangle
        \[\begin{tikzcd}
            \W^{\B}_{N, +} / \mathcal{F}^m \W^{\B}_{N, +} \ar[hook]{rr} \ar["\phi_m"']{rd} & & \V^{\B}_N / \mathcal{F}^m \V^{\B}_N \\
            & \V^{\B}_N / \mathcal{F}^{m-1} \V^{\B}_N  \ar["- \cdot (X_1 - 1)"']{ru} &
        \end{tikzcd}\]
        commutes where the morphism $- \cdot (X_1 - 1) : \V^{\B}_N / \mathcal{F}^{m-1} \V^{\B}_N \to \V^{\B}_N / \mathcal{F}^m \V^{\B}_N$ is well-defined thanks to the inclusion $\mathcal{F}^{m-1} \V^{\B}_N (X_1 - 1) \subset \mathcal{F}^m \V^{\B}_N$ being a consequence of \eqref{FmWB_plus}. On the other hand, we have, for any $m \in \N^{\ast}$, the following triangle
        \[\begin{tikzcd}
            & \V^{\B}_N / \mathcal{F}^m \V^{\B}_N & \\
            \V^{\B}_N / \mathcal{F}^{m-1} \V^{\B}_N \ar["- \cdot (X_1 - 1)"]{ru} & & \V^{\B}_N / \mathcal{F}^m \V^{\B}_N \ar[two heads, "- \cdot (X_1 - 1)"']{lu} \ar[two heads, "\pi_m"]{ll}
        \end{tikzcd}\]
        where $\pi_m : \V^{\B}_N / \mathcal{F}^m \V^{\B}_N \twoheadrightarrow \V^{\B}_N / \mathcal{F}^{m-1} \V^{\B}_N$ is the morphism which associates to the class of an element modulo $\mathcal{F}^m \V^{\B}_N$, its class modulo $\mathcal{F}^{m-1} \V^{\B}_N$; this is well-defined and surjective thanks to the inclusion $\mathcal{F}^m \V^{\B}_N \subset \mathcal{F}^{m-1} \V^{\B}_N$. One then checks that this triangle commutes. By linking the two triangles and doing the inverse limit we obtain the following diagram
        \[\begin{tikzcd}
            \widehat{\W}^{\B}_{N, +} \ar[hook]{rr} \ar["\widehat{\phi}"']{rd} & & \widehat{\V}^{\B}_N & \\
            & \displaystyle \lim_{\longleftarrow} \V^{\B}_N / \mathcal{F}^{m-1} \V^{\B}_N \ar[]{ru} & & \widehat{\V}^{\B}_N \ar[two heads, "- \cdot (X_1 - 1)"']{lu} \ar["\widehat{\pi}"]{ll}
        \end{tikzcd}\]
        where $\displaystyle \widehat{\pi} := \lim_{\longleftarrow} \pi_m : \widehat{\V}^{\B}_N \to \displaystyle \lim_{\longleftarrow} \V^{\B}_N / \mathcal{F}^{m-1} \V^{\B}_N$ is obtained by degree shifting and is therefore a topological $\K$-module isomorphism. Let us set $\widetilde{\phi} := \widehat{\pi}^{-1} \circ \widehat{\phi} : \widehat{\W}^{\B}_{N, +} \to \widehat{\V}^{\B}_N$. It is a topological $\K$-module morphism such that Diagram \eqref{triangle_tilde_phi} commutes. Finally, the commutativity of the latter diagram implies
        \begin{align}
           \mathrm{Im}\left(\widehat{\W}^{\B}_{N, +} \hookrightarrow \widehat{\V}^{\B}_N\right)
           & = \mathrm{Im}\big(- \cdot (X_1 - 1) \circ \tilde{\phi} \big) \notag \\
           \label{W+_subset_mult} & \subset \mathrm{Im}\big(- \cdot (X_1 - 1)\big) = \widehat{\V}^{\B}_N (X_1 - 1).
        \end{align}
    \end{description}
    Finally, combining inclusions \eqref{mult_subset_W+} and \eqref{W+_subset_mult} we obtain
    \[
        \widehat{\W}^{\B}_{N, +} \simeq \mathrm{Im}\left(\widehat{\W}^{\B}_{N, +} \hookrightarrow \widehat{\V}^{\B}_N\right) = \widehat{\V}^{\B}_N (X_1 - 1).
    \]
    In addition, thanks to \ref{hatWB_hatWB+}, the topological $\K$-algebras $\K \oplus \widehat{\W}^{\B}_{N, +}$ and $\widehat{\W}^{\B}_N$ are isomorphic. One then obtains the isomorphism of topological $\K$-algebras
    \[
        \widehat{\W}^{\B}_N \simeq \K \oplus \widehat{\V}^{\B}_N (X_1 - 1).  
    \]
\end{proof}

\begin{propdef}
    \label{isoW}
    Let $\iota \in \Emb(G)$. There exists a topological algebra isomorphism \index{$\mathrm{iso}^{\W, \iota}$}$\mathrm{iso}^{\W, \iota} : \widehat{\W}^{\B}_N \to \widehat{\W}^{\DR}_G$ such that the following diagram
    \begin{equation}
        \label{diag_isoW}
        \begin{tikzcd}
            \widehat{\W}^{\B}_N \ar["\iso^{\W, \iota}"]{rrr} \ar[hook]{d} &&& \widehat{\W}^{\DR}_G \ar[hook]{d} \\
            \widehat{\V}^{\B}_N \ar["\iso^{\V, \iota}"']{rrr} &&& \widehat{\V}^{\DR}_G
        \end{tikzcd}
    \end{equation}
    commutes.
\end{propdef}
\begin{proof}
    We have
    \[
        \iso^{\V, \iota}(X_1 - 1) = \exp(e_1) - 1 = u e_1,
    \]
    where $u= f(e_1)$ with $f(x)$ being the invertible formal series $\frac{\exp(x)-1}{x}$.
    Moreover, since $\mathrm{iso}^{\V, \iota} : \widehat{\V}^{\B}_N \to \widehat{\V}^{\DR}_G$ is a $\K$-algebra isomorphism, we obtain 
    \[
        \iso^{\V, \iota}\left(\widehat{\V}^{\B}_N(X_1 - 1)\right) = \iso^{\V, \iota}(\widehat{\V}^{\B}_N) \,\, \iso^{\V, \iota}(X_1 - 1) = \widehat{\V}^{\DR}_G u e_1 = \widehat{\V}^{\DR}_G e_1.
    \]
    This implies that $\mathrm{iso}^{\V, \iota}_{\big| \widehat{\V}^{\B}_N (X_1 - 1)} : \widehat{\V}^{\B}_N(X_1 - 1) \to \widehat{\V}^{\DR}_G e_1$ is a surjective $\K$-module morphism which is trivially injective, therefore, is a $\K$-module isomorphism. Taking the direct sum with $\K$, we obtain a $\K$-module isomorphism
    \[
        \K \oplus \widehat{\V}^{\B}_N(X_1 - 1) \to \K \oplus \widehat{\V}^{\DR}_G e_1,
    \]
    which is a $\K$-algebra isomorphism. Finally, thanks to Lemma \ref{WB}, this isomorphism is the wanted $\K$-algebra isomorphism $\mathrm{iso}^{\W, \iota} : \widehat{\W}^{\B}_N \to \widehat{\W}^{\DR}_G$.
\end{proof}

\begin{corollary}
    \label{eta_W_phi_iso}
    Let $\iota \in \mathrm{Emb}(G)$ and $\phi \in \Aut(G)$. We have
    \[
        \iso^{\W, \iota \circ \phi^{-1}} = \eta_{\phi}^{\W} \circ \iso^{\W, \iota},
    \]
    with $\eta^{\W}_{\phi} \in \Aut_{\K{\text -}\alg_{\mathrm{top}}}(\widehat{\W}^{\DR}_G)$ given in Lemma \ref{from_etaV}.\ref{etaW}.
\end{corollary}
\begin{proof}
    The statement follows from Proposition \ref{eta_phi_iso} thanks to the commutativity of diagrams \eqref{diag_isoW} and \eqref{diag_etaW_etaV}.
\end{proof}

\subsubsection{The filtered module \texorpdfstring{$\M^{\B}_N$}{MBN}}
\begin{propdef}
    \label{prop_def_MB}
    The quotient $\K$-module \index{$\M^{\B}_N$}
    \begin{equation}
        \M^{\B}_N := \V^{\B}_N \Big/ \V^{\B}_N (X_0 - 1)
    \end{equation}
    is a $\V^{\B}_N$-module. Moreover, if we denote \index{$1_{\B}$}$1_{\B}$ the class of $1 \in \V^{\B}_N$ in $\M^{\B}_N$, then the canonical projection \index{$- \cdot 1_{\B}$}
    \[
        - \cdot 1_{\B} : \V^{\B}_N \to \M^{\B}_N
    \]
    is a surjective $\V^{\B}_N$-module morphism and its restriction to $\W^{\B}_N$ is a $\W^{\B}_N$-module isomorphism.
\end{propdef}
\begin{proof}
    This follows from the direct sum decomposition
    \[
        \V^{\B}_N = \K \oplus \V^{\B}_N (X_1 - 1) \oplus \V^{\B}_N (X_0 - 1) = \W^{\B}_N \oplus \V^{\B}_N (X_0 - 1)
    \]
    given by \cite[Proposition 6.2.6]{Wei}.
\end{proof}

\noindent\textbf{Remark}. The statement implies that $(- \cdot 1_{\B})_{|\W^{\B}_N} : \W^{\B}_N \to \M^{\B}_N$ is a $\W^{\B}_N$-module isomorphism, therefore $\M^{\B}_N$ is a free $\W^{\B}_N$-module of rank $1$.  
\begin{propdef}
    The $\K$-module $\M^{\B}_N$ is endowed with the decreasing $\K$-module filtration given by \index{$\mathcal{F}^m \M^{\B}_N$}
    \begin{equation}
        \label{MBFiltration}
        \mathcal{F}^m \M^{\B}_N := \mathcal{F}^m \W^{\B}_N \cdot 1_{\B} \text{ \, for } m \in \N.
    \end{equation}
    Moreover, the pair $\left(\M^{\B}_N, \left(\mathcal{F}^m \M^{\B}_N\right)_{m \in \N}\right)$ is a filtered module over the filtered algebra $\left(\W^{\B}_N, \left(\mathcal{F}^m \W^{\B}_N\right)_{m \in \N}\right)$.
\end{propdef}
\begin{proof}
    Immediate.
\end{proof}

\begin{lemma} \ 
    \begin{enumerate}[label=(\roman*), leftmargin=*]
        \item \label{FmWFmM} For any $m \in \N$, the $\K$-module isomorphism $- \cdot 1_{\B} : \W^{\B}_N \to \M^{\B}_N$ induces a $\K$-modules isomorphism $\mathcal{F}^m\W^{\B}_N \to \mathcal{F}^m\M^{\B}_N$.
        \item \label{FmVFmM} For any $m \in \N$, we have $\mathcal{F}^m\M^{\B}_N = \mathcal{F}^m\V^{\B}_N \cdot 1_{\B}$.
    \end{enumerate}
    \label{1B_filtrations}
\end{lemma}
\begin{proof} \ 
    \begin{enumerate}[label=(\roman*), leftmargin=*]
        \item By definition of $\mathcal{F}^m\M^{\B}_N$, the isomorphism $- \cdot 1_{\B} : \W^{\B}_N \to \M^{\B}_N$ restricts to a surjective $\K$-module morphism $\mathcal{F}^m\W^{\B}_N \to \mathcal{F}^m\M^{\B}_N$. In addition, since $- \cdot 1_{\B} : \W^{\B}_N \to \M^{\B}_N$ is injective, so is the restriction $\mathcal{F}^m\W^{\B}_N \to \mathcal{F}^m\M^{\B}_N$.
        \item First, if $m = 0$, the equality follows from Proposition-Definition \ref{prop_def_MB}. \newline
        From now on, let $m \in \N^{\ast}$. Since $\mathcal{F}^m\W^{\B}_N \subset \mathcal{F}^m\V^{\B}_N$, we have that
        \[
            \mathcal{F}^m\M^{\B}_N \subset \mathcal{F}^m\V^{\B}_N \cdot 1_{\B}.
        \]
        Conversely, let us prove the inclusion $\mathcal{F}^m\V^{\B}_N \cdot 1_{\B} \subset \mathcal{F}^m\M^{\B}_N$.
        This inclusion is equivalent to
        \[
            \mathcal{F}^m\V^{\B}_N \subset \mathcal{F}^m\W^{\B}_N + \V^{\B}_N (X_0 - 1).
        \]
        Since $\mathcal{F}^m\V^{\B}_N = \mathcal{I}^m = \ker(\V^{\B}_N \to \K\mu_N)^m$ and by Lemma \ref{FmWB} \ref{FmWB_equality}, this inclusion is also equivalent to
        \begin{equation}
            \mathcal{I}^m \subset \left(\mathcal{I}^m \cap \V^{\B}_N (X_1 - 1)\right) + \V^{\B}_N (X_0 - 1).
        \end{equation}
        We have
        \[
            \mathcal{I} = \ker(\V^{\B}_N \to \K\mu_N) \subset \ker(\V^{\B}_N \to \K) = \V^{\B}_N (X_0 - 1) + \V^{\B}_N (X_1 - 1),  
        \]
        with $\ker(\V^{\B}_N \to \K)$ being the augmentation ideal of the group algebra $\V^{\B}_N = \K F_2$ and the last equality being a consequence of \cite[Proposition 6.2.6]{Wei}. \newline
        This implies
        \begin{align}
            \mathcal{I}^m = \mathcal{I}^{m-1} \mathcal{I} & \subset \mathcal{I}^{m-1} \left(\V^{\B}_N (X_1 - 1) + \V^{\B}_N (X_0 - 1)\right) \notag \\
            \label{subset_Im_direct_sum} & \subset \mathcal{I}^{m-1} \V^{\B}_N (X_1 - 1) + \V^{\B}_N (X_0 - 1).
        \end{align}
        Moreover, $\V^{\B}_N (X_1 - 1) \subset \ker(\V^{\B}_N \to \K\mu_N)$ since $X_1 - 1 \mapsto 0$ through the map $\V^{\B}_N \to \K\mu_N$. This implies
        \begin{equation}
            \label{subset_Im}
            \mathcal{I}^{m-1} \V^{\B}_N (X_1 - 1) \subset \mathcal{I}^{m-1} \mathcal{I} = \mathcal{I}^{m}.
        \end{equation}
        On the other hand, we have
        \begin{equation}
            \label{subset_VBX1}
            \mathcal{I}^{m-1} \V^{\B}_N (X_1 - 1) \subset \V^{\B}_N (X_1 - 1).
        \end{equation}
        From \eqref{subset_Im} and \eqref{subset_VBX1} we obtain
        \begin{equation}
            \label{subset_and_cap}
            \mathcal{I}^{m-1} \V^{\B}_N (X_1 - 1) \subset \left(\mathcal{I}^{m} \cap \V^{\B}_N (X_1 - 1)\right). 
        \end{equation}
        Finally, from \eqref{subset_Im_direct_sum} and \eqref{subset_and_cap}, we obtain
        \[
            \mathcal{I}^m \subset \left(\mathcal{I}^{m} \cap \V^{\B}_N (X_1 - 1)\right) + \V^{\B}_N (X_0 - 1),
        \]
        which is the wanted inclusion.
    \end{enumerate}
\end{proof}

\subsubsection{The topological module \texorpdfstring{$\widehat{\M}^{\B}_N$}{hatMBN}}
The decreasing filtration $(\mathcal{F}^m \M^{\B}_N)_{m \in \N}$ given in (\ref{MBFiltration}) induces a $\K$-module morphism $\M^{\B}_N / \mathcal{F}^{m+1} \M^{\B}_N \to \M^{\B}_N / \mathcal{F}^m \M^{\B}_N$.
\begin{definition}
    We denote \index{$\widehat{\M}^{\B}_N$}
    \[
        \widehat{\M}^{\B}_N := \lim_{\longleftarrow} \M^{\B}_N/\mathcal{F}^m \M^{\B}_N. 
    \]
    the limit of the projective system $(\M^{\B}_N/\mathcal{F}^m \M^{\B}_N, \M^{\B}_N / \mathcal{F}^{m+1} \M^{\B}_N \to \M^{\B}_N / \mathcal{F}^m \M^{\B}_N)$.
\end{definition}
\noindent The $\K$-module $\widehat{\M}^{\B}_N$ is a $\widehat{\V}^{\B}_N$-module equipped with the filtration \index{$\mathcal{F}^m \widehat{\M}^{\B}_N$}
\[
    \mathcal{F}^m \widehat{\M}^{\B}_N := \lim_{\longleftarrow} \mathcal{F}^m \M^{\B}_N / \mathcal{F}^{\max(m,l)} \M^{\B}_N, \text{ for } m \in \N.
\]
When equipped with the topology defined by this filtration, $\widehat{\M}^{\B}_N$ is a complete separated topological $\K$-module.

\begin{lemma} \ 
    \begin{enumerate}[label=(\roman*), leftmargin=*]
        \item \label{hatVhatM} The surjective $\K$-module morphism $- \cdot 1_{\B} : \V^{\B}_N \to \M^{\B}_N$ induces a topological surjective $\K$-module morphism \index{$\widehat{- \cdot 1_{\B}}$}$\widehat{- \cdot 1_{\B}} : \widehat{\V}^{\B}_N \to \widehat{\M}^{\B}_N$.
        \item \label{hatWhatM} The $\K$-module isomorphism $- \cdot 1_{\B} : \W^{\B}_N \to \M^{\B}_N$ induces a topological $\K$-module isomorphism $\widehat{- \cdot 1_{\B}} : \widehat{\W}^{\B}_N \to \widehat{\M}^{\B}_N$. 
    \end{enumerate}
    \label{hat1B}
\end{lemma}
\begin{proof} \ 
    \begin{enumerate}[label=(\roman*), leftmargin=*]
        \item By definition of $\widehat{\V}^{\B}_N$ and $\widehat{\M}^{\B}_N$, this follows from Lemma \ref{1B_filtrations} \ref{FmVFmM}.
        \item By definition of $\widehat{\M}^{\B}_N$, this follows from Lemma \ref{1B_filtrations} \ref{FmWFmM}.  
    \end{enumerate}
\end{proof}

\begin{corollary} \ 
    \begin{enumerate}[label=(\roman*), leftmargin=*]
        \item \label{hatVhatM_algmod} The pair $(\widehat{\V}^{\B}_N, \widehat{\M}^{\B}_N)$ is an object in the category $\K{\text -}\alg{\text -}\Mod_{\mathrm{top}}$.
        \item \label{hatWhatM_algmod} The pair $(\widehat{\W}^{\B}_N, \widehat{\M}^{\B}_N)$ is an object in the category $\K{\text -}\alg{\text -}\Mod_{\mathrm{top}}$. Moreover, $\widehat{\M}^{\B}_N$ is a free $\widehat{\W}^{\B}_N$-module of rank $1$.
    \end{enumerate}
    \label{algmod_B}
\end{corollary}
\begin{proof}
    It immediately follows from Lemma \ref{hat1B}.
\end{proof}

\begin{proposition}
    \label{hatM_is_quotient}
    The topological $\K$-module morphism $\widehat{- \cdot 1_{\B}} : \widehat{\V}^{\B}_N \to \widehat{\M}^{\B}_N$ induces an isomorphism $\widehat{\V}^{\B}_N \big/ \widehat{\V}^{\B}_N (X_0 - 1) \to \widehat{\M}^{\B}_N$ of topological $\K$-modules.
\end{proposition}
\noindent In order to prove this, we will need the following Lemma:
\begin{lemma}
    \label{linear_lemma}
    Let $V$ be a $\K$-module and $u$ be an endomorphism of $V$. Let $f : V^N \to V^N$ be the endomorphism given by
    \[
        (v_0, \dots, v_{N-1}) \mapsto (u(v_{N-1}) - v_0, v_0 - v_1, v_1 - v_2, \dots, v_{N-2} - v_{N-1}).  
    \]
    Then we have an isomorphism
    \[
        \mathrm{coker}(f) \simeq \mathrm{coker}(u-\mathrm{id}).
    \]
\end{lemma}
\begin{proof}[Proof of Lemma \ref{linear_lemma}]
    Let us consider the $\K$-module morphism $\mathrm{sum} : V^N \to V$ given by $(v_0, \dots, v_{N-1}) \mapsto v_0 + \cdots + v_{N-1}$. This morphism sends $\mathrm{Im}(f)$ to $\mathrm{Im}(u - \mathrm{id})$. Therefore, there is a unique $\K$-module morphism $V^N/\mathrm{Im}(f) \to V/\mathrm{Im}(u - \mathrm{id})$ such that the diagram
    \[
        \begin{tikzcd}
            V^N \ar["\mathrm{sum}"]{r} \ar[two heads]{d} & V \ar[two heads]{d} \\
            V^N/\mathrm{Im}(f) \ar[]{r} & V/\mathrm{Im}(u - \mathrm{id})
        \end{tikzcd}
    \]
    commutes. Let us show that the morphism $V^N/\mathrm{Im}(f) \to V/\mathrm{Im}(u - \mathrm{id})$ is an isomorphism. First, the surjectivity of the morphism $\mathrm{sum} : V^N \to V$ implies that the morphism $V^N/\mathrm{Im}(f) \to V/\mathrm{Im}(u - \mathrm{id})$ is surjective as well. Second, let $(w_0, \dots, w_{N-1}) \in V^N$ such that there exists an element $v \in V$ such that $w_0 + \cdots w_{N-1} = u(v) - v$. The element $(v_0, \dots, v_{N-1}) \in V^N$ given by
    \[
        v_{N-1} = v, v_{N-2} = w_{N-1} + v, v_{N-3} = w_{N-2} + w_{N-1} + v, \dots, v_0 = w_1 + \cdots w_{N-1} + v
    \]
    is such that
    \[
        (w_0, \dots, w_{N-1}) = (u(v_{N-1}) - v_0, v_0 - v_1, \dots, v_{N-2} - v_{N-1}) \in \mathrm{Im}(f).
    \]
    Thus proving the injectivity of $V^N/\mathrm{Im}(f) \to V/\mathrm{Im}(u - \mathrm{id})$. 
\end{proof}
\begin{proof}[Proof of Proposition \ref{hatM_is_quotient}]
    The proof consists of the following steps:
    \begin{description}[leftmargin=1pt]
        \item[Step 1] Construction of the $\K$-module isomorphism $P_0 : (\K F_{N+1})^N \to \V^{\B}_N$. \newline
        As in \eqref{Def_kSigma}, one defines the $\K$-module isomorphism $\K F_{N+1} \otimes \K\mu_N \to \V^{\B}_N$ such that for $(x, \zeta) \in F_{N+1} \times \mu_N$
        \begin{equation}
            \label{tildeSigma}
            x \otimes \zeta \mapsto x \, \sigma(\zeta),
        \end{equation}
        where $F_{N+1}$ is seen as $\ker(F_2 \to \mu_N) \subset F_2$ thanks to Lemma \ref{FN1}.
        Moreover, one checks there is a $\K$-module isomorphism $(\K F_{N+1})^N \to \K F_{N+1} \otimes \K\mu_N$ given by
        \begin{equation}
            \label{kFNN}
            (v_0, \dots, v_{N-1}) \mapsto \sum_{i=0}^{N-1} v_i \otimes \zeta_N^i.
        \end{equation}
        Therefore, the composition $P_0 : (\K F_{N+1})^N \to \K F_{N+1} \otimes \K\mu_N \to \V^{\B}_N$ is a $\K$-module isomorphism and is given by
        \[
            (v_0, \dots, v_{N-1}) \mapsto v_0 + v_1 X_0 + \cdots + v_{N-1} X_0^{N-1}.
        \]
        \item[Step 2] Identification of $\M^{\B}_N$. \newline
        One checks that the endomorphism $f : (\K F_{N+1})^N \to (\K F_{N+1})^N$ given by
        \begin{equation}
            \label{endo_f}
            (v_0, \dots, v_{N-1}) \mapsto (v_{N-1} \widetilde{X}_0 - v_0, v_0 - v_1, v_1 - v_2, \dots, v_{N-2} - v_{N-1}),
        \end{equation}
        is such that the following diagram
        \begin{equation}
            \label{Diag_f}
            \begin{tikzcd}
                (\K F_{N+1})^N \ar["f"]{rr} \ar["P_0"']{d} && (\K F_{N+1})^N \ar["P_0"]{d} \\
                \V^{\B}_N \ar["- \cdot (X_0 - 1)"']{rr} && \V^{\B}_N 
            \end{tikzcd}
        \end{equation}
        commutes. This induces a $\K$-module isomorphism $\mathrm{coker}(f) \simeq \mathrm{coker}(-\cdot(X_0-1))$. \newline
        On the other hand, by applying Lemma \ref{linear_lemma} with $V = \K F_{N+1}$ and $u = - \cdot \widetilde{X}_0$, we obtain an isomomorphism $\mathrm{coker}(f) \simeq \mathrm{coker}(u - \mathrm{id})$.
        It then follows that
        \begin{align*}
            \M^{\B}_N = \V^{\B}_N / \V^{\B}_N (X_0 - 1) & = \mathrm{coker}(-\cdot(X_0-1)) \\ & \simeq \mathrm{coker}(f) \simeq \mathrm{coker}(u - \mathrm{id}) = \K F_{N+1} / \K F_{N+1} (\widetilde{X}_0 - 1). 
        \end{align*}
        \item[Step 3] Compatibility of the isomorphism $\K F_{N+1} / \K F_{N+1} (\widetilde{X}_0 - 1) \to \M^{\B}_N$ with filtrations. Let us show that for any $m \in \N$, we have
        \[
            (\K F_{N+1})_0^m \Big/ \left((\K F_{N+1})_0^m \cap \K F_{N+1} (\widetilde{X}_0 - 1)\right) \simeq \mathcal{F}^m \M^{\B}_N.
        \]
        If $m=0$, this has been proved in Step 2. From now on, let us assume that $m \in \N^{\ast}$.
        The isomorphism $\K F_{N+1} / \K F_{N+1} (\widetilde{X}_0 - 1) \to \M^{\B}_N$ fits in the following commutative diagram
        \begin{equation}
            \label{Diag_kF_MB}
            \begin{tikzcd}
                \K F_{N+1} \ar[hook]{r} \ar[two heads]{d} & \V^{\B}_N \ar[two heads]{d} \\
                \K F_{N+1} / \K F_{N+1} (\widetilde{X}_0 - 1) \ar["\simeq"]{r} & \M^{\B}_N
            \end{tikzcd}
        \end{equation}
        where $\K F_{N+1} \to \V^{\B}_N$ is the group algebra morphism induced by the group morphism $F_{N+1} \simeq \ker(F_2 \to \mu_N) \subset F_2$ obtained in Lemma \ref{FN1}.
        This group algebra morphism induces the injection
        \[
            (\K F_{N+1})_0^m \hookrightarrow \mathcal{F}^m \V^{\B}_N.
        \]
        Then, thanks to the commutativity of Diagram \eqref{Diag_kF_MB}, the $\K$-module isomorphism $\K F_{N+1} / \K F_{N+1} (\widetilde{X}_0 - 1) \to \M^{\B}_N$ induces an injection
        \[
            (\K F_{N+1})_0^m \Big/ \left((\K F_{N+1})_0^m \cap \K F_{N+1} (\widetilde{X}_0 - 1)\right) \hookrightarrow \mathcal{F}^m \V^{\B}_N \cdot 1_{\B} = \mathcal{F}^m \M^{\B}_N,  
        \]
        where the equality comes from Lemma \ref{1B_filtrations}.\ref{FmVFmM}. This implies that
        \[
            \mathrm{Im}\left((\K F_{N+1})_0^m \Big/ \left((\K F_{N+1})_0^m \cap \K F_{N+1} (\widetilde{X}_0 - 1)\right) \to \M^{\B}_N \right) \subset \mathcal{F}^m \M^{\B}_N.
        \]
        Conversely, let us show the opposite inclusion. Thanks to Lemma \ref{FmWB}\ref{FmWB_to_FmVB}, we have  
        \[
            \mathcal{F}^m \M^{\B}_N = \mathcal{F}^m \W^{\B}_N \cdot 1_{\B} = \mathcal{F}^{m-1} \V^{\B}_N (X_1 - 1) \cdot 1_{\B}.
        \]
        Moreover, we have by definition that $\mathcal{F}^{m-1} \V^{\B}_N = (\mathcal{F}^1 \V^{\B}_N)^{m-1}$. This implies, thanks to Lemma \ref{IdealI} \ref{Igenerators} that $\mathcal{F}^{m-1} \V^{\B}_N (X_1 - 1) \cdot 1_{\B}$ is linearly generated by elements
        \[
            \sigma(\zeta_1) (x_1 - 1) \cdots \sigma(\zeta_{m-1}) (x_{m-1} - 1) (X_1 - 1) \cdot 1_{\B}
        \]
        with $(\zeta_1, x_1), \dots, (\zeta_{m-1}, x_{m-1}) \in \mu_N \times F_{N+1}$. Additionally, we have that
        \begin{align*}
            & \sigma(\zeta_1) (x_1 - 1) \cdots \sigma(\zeta_{m-1}) (x_{m-1} - 1) (X_1 - 1) \cdot 1_{\B} = \\
            & \left(\Ad_{\sigma(\zeta_1)}(x_1) - 1\right) \cdots \left(\Ad_{\sigma(\zeta_1) \cdots \sigma(\zeta_{m-1})}(x_{m-1}) - 1\right) \left(\Ad_{\sigma(\zeta_1) \cdots \sigma(\zeta_{m-1})}(X_1) - 1\right) \\
            & \sigma(\zeta_1) \cdots \sigma(\zeta_{m-1}) \cdot 1_{\B} = \\
            & \left(\Ad_{\sigma(\zeta_1)}(x_1) - 1\right) \cdots \left(\Ad_{\sigma(\zeta_1) \cdots \sigma(\zeta_{m-1})}(x_{m-1}) - 1\right) \left(\Ad_{\sigma(\zeta_1) \cdots \sigma(\zeta_{m-1})}(X_1) - 1\right) \cdot 1_{\B} 
        \end{align*}
        which belongs to the image of $(\K F_{N+1})_0^m \Big/ \left((\K F_{N+1})_0^m \cap \K F_{N+1} (\widetilde{X}_0 - 1)\right)$ by the isomorphism $\K F_{N+1} / \K F_{N+1} (\widetilde{X}_0 - 1) \to \M^{\B}_N$ as
        \[
            \left(\Ad_{\sigma(\zeta_1)}(x_1) - 1\right) \cdots \left(\Ad_{\sigma(\zeta_1) \cdots \sigma(\zeta_{m-1})}(x_{m-1}) - 1\right) \left(\Ad_{\sigma(\zeta_1) \cdots \sigma(\zeta_{m-1})}(X_1) - 1\right),
        \]
        seen as an element of $\K F_{N+1}$, belongs to $(\K F_{N+1})_0^m$. This implies that
        \[
            \mathcal{F}^m \M^{\B}_N \subset \mathrm{Im}\left((\K F_{N+1})_0^m \Big/ \left((\K F_{N+1})_0^m \cap \K F_{N+1} (\widetilde{X}_0 - 1)\right) \to \M^{\B}_N \right).
        \]
        Therefore, one has equality
        \[
            \mathrm{Im}\left((\K F_{N+1})_0^m \Big/ \left((\K F_{N+1})_0^m \cap \K F_{N+1} (\widetilde{X}_0 - 1)\right) \to \M^{\B}_N \right) = \mathcal{F}^m \M^{\B}_N,
        \]
        which establishes the wanted isomorphism.  
        \item[Step 4] Identification of $\widehat{\M}^{\B}_N$. \newline
        Thanks to Step 3, one has for any $m \in \N$
        \[
            \M^{\B}_N / \mathcal{F}^m \M^{\B}_N \simeq \K F_{N+1} \Big/ \left(\K F_{N+1} (\widetilde{X}_0 - 1) + (\K F_{N+1})_0^m \right)
        \]
        and, on the other hand, for any $m \in \N^{\ast}$,
        \[
            \mbox{\small$\K F_{N+1} \Big/ \left(\K F_{N+1} (\widetilde{X}_0 - 1) + (\K F_{N+1})_0^m \right) \simeq \mathrm{coker}\left( \K F_{N+1} / (\K F_{N+1})_0^{m-1} \to  \K F_{N+1} / (\K F_{N+1})_0^m \right)$},
        \]
        where the morphism
        \begin{equation}
            \label{m_minus_1_to_m}
            \K F_{N+1} / (\K F_{N+1})_0^{m-1} \to  \K F_{N+1} / (\K F_{N+1})_0^m
        \end{equation}
        is induced by the endomorphism $- \cdot (\widetilde{X}_0 - 1)$ of $\K F_{N+1}$.
        Therefore,
        \[
            \widehat{\M}^{\B}_N \simeq \lim_{\longleftarrow} \mathrm{coker}\left( \K F_{N+1} / (\K F_{N+1})_0^{m-1} \to  \K F_{N+1} / (\K F_{N+1})_0^m \right). 
        \]
        \item[Step 5] Identification of $\widehat{\V}^{\B}_N / \widehat{\V}^{\B}_N (X_0 - 1)$. \newline
        As in Lemma \ref{filtr}, one proves that the $\K$-module isomorphism $\K F_{N+1} \otimes \K\mu_N \to \V^{\B}_N$ given in \eqref{tildeSigma} allows us to identify $(\K F_{N+1})_0^m \otimes \K\mu_N$ with $\mathcal{F}^m \V^{\B}_N$ for any $m \in \N$.
        Recall the isomorphism $(\K F_{N+1})^N \to \K F_{N+1} \otimes \K\mu_N$ given in \eqref{kFNN}. One checks that it is compatible with the filtration of $(\K F_{N+1})^N$ given for any $m \in \N$ by $\prod_{i=1}^N (\K F_{N+1})_0^m$.
        Therefore the isomorphism $P_0 : (\K F_{N+1})^N \to \K F_{N+1} \otimes \K\mu_N \to \V^{\B}_N$ of Step 1 is compatible with filtrations. Therefore, it extends to a topological $\K$-module isomorphism
        \[
            \widehat{P}_0 : \widehat{(\K F_{N+1})^N} \to \widehat{\V}^{\B}_N.
        \]
        On the other hand, the endomorphism $f : (\K F_{N+1})^N \to (\K F_{N+1})^N$ given in \eqref{endo_f} is compatible with filtrations and then extends to a topological endomorphism $\hat{f} : \widehat{(\K F_{N+1})^N} \to \widehat{(\K F_{N+1})^N}$ and, thanks to Diagram \eqref{Diag_f}, it is such that the following diagram
        \[
            \begin{tikzcd}
                \widehat{(\K F_{N+1})^N} \ar["\hat{f}"]{rr} \ar["\simeq"']{d} && \widehat{(\K F_{N+1})^N} \ar["\simeq"]{d} \\
                \widehat{\V}^{\B}_N \ar["- \cdot (X_0 - 1)"']{rr} && \widehat{\V}^{\B}_N 
            \end{tikzcd}
        \]
        commutes. This induces a $\K$-module isomorphism $\mathrm{coker}(\hat{f}) \simeq \mathrm{coker}(- \cdot (X_0-1))$. \newline
        Similarly to Step 1, by applying Lemma \ref{linear_lemma} with $V = \widehat{\K F_{N+1}}$ and $u = - \cdot \widetilde{X}_0$, we obtain an isomomorphism $\mathrm{coker}(\hat{f}) \simeq \mathrm{coker}(u - \mathrm{id})$.
        It then follows that
        \begin{align*}
            \widehat{\V}^{\B}_N / \widehat{\V}^{\B}_N (X_0 - 1) & = \mathrm{coker}(-\cdot(X_0-1)) \\ & \simeq \mathrm{coker}(\hat{f}) \simeq \mathrm{coker}(u - \mathrm{id}) = \widehat{\K F_{N+1}} / \widehat{\K F_{N+1}} (\widetilde{X}_0 - 1). 
        \end{align*}
        On the other hand, we have
        \[
            \mbox{\small$\displaystyle \widehat{\K F_{N+1}} / \widehat{\K F_{N+1}} (\widetilde{X}_0 - 1) \simeq \mathrm{coker}\left( \lim_{\longleftarrow} \K F_{N+1} / (\K F_{N+1})_0^{m-1} \to  \lim_{\longleftarrow} \K F_{N+1} / (\K F_{N+1})_0^m \right)$},
        \]
        where the morphism $\displaystyle\lim_{\longleftarrow} \K F_{N+1} / (\K F_{N+1})_0^{m-1} \to  \lim_{\longleftarrow} \K F_{N+1} / (\K F_{N+1})_0^m$ is induced by the morphism $\K F_{N+1} / (\K F_{N+1})_0^{m-1} \to  \K F_{N+1} / (\K F_{N+1})_0^m$ given in \eqref{m_minus_1_to_m}.
        \item[Step 6] Cokernel of limits and limit of cokernels coincide. \newline
        For any $m \in \N^{\ast}$, the morphism
        \begin{equation}
            \label{kFN_degree_morphism}
            \K F_{N+1} / (\K F_{N+1})_0^{m-1} \to  \K F_{N+1} / (\K F_{N+1})_0^m
        \end{equation}
        induced by the endomorphism $- \cdot (\widetilde{X}_0 - 1)$ of $\K F_{N+1}$ is injective. Indeed, let $x \in \K F_{N+1}$ such that $x(\widetilde{X}_0 - 1) \in (\K F_{N+1})_0^m$. Let $l$ to be the smallest integer such that $x \in (\K F_{N+1})_0^l$. Let us show that $l \geq m-1$. Otherwise, since $x \in (\K F_{N+1})_0^l$, we have that $[x] \in \mathrm{gr}_l(\K F_{N+1})$. Moreover, we have that $[\widetilde{X}_0 - 1] \in \mathrm{gr}_1(\K F_{N+1})$. Therefore, $[x(\widetilde{X}_0 - 1)] \in \mathrm{gr}_{l+1}(\K F_{N+1})$. Since, by assumption, $l+1 \leq m-1$, the condition $x(\widetilde{X}_0 - 1) \in (\K F_{N+1})_0^m$ implies that $[x(\widetilde{X}_0 - 1)] = 0$. Finally, since $\mathrm{gr}_{l+1}(\K F_{N+1})$ is an integral domain we would obtain that $[x]=0$, contradicting the minimality of $l$. Therefore, $l \geq m-1$ and the morphism \eqref{kFN_degree_morphism} is injective. \newline
        In addition, the image of morphism \eqref{kFN_degree_morphism} is the same as the image of the morphism $\K F_{N+1} / (\K F_{N+1})_0^m \to  \K F_{N+1} / (\K F_{N+1})_0^m$ induced by the endomorphism $- \cdot (\widetilde{X}_0 - 1)$ of $\K F_{N+1}$. We then have the short exact sequence
        \[
            \{0\} \to \K F_{N+1} / (\K F_{N+1})_0^{m-1} \to \K F_{N+1} / (\K F_{N+1})_0^m
        \]
        \[
            \to \mathrm{coker}\big((\K F_{N+1} / (\K F_{N+1})_0^m \to \K F_{N+1} / (\K F_{N+1})_0^m)\big) \to \{0\}
        \]
        which, by applying the inverse limit functor, gives us
        \begin{align*}
            &\{0\} \to \lim_{\longleftarrow} \K F_{N+1} / (\K F_{N+1})_0^{m-1} \to \lim_{\longleftarrow} \K F_{N+1} / (\K F_{N+1})_0^m \to \\
            &\lim_{\longleftarrow} \mathrm{coker}\big((\K F_{N+1} / (\K F_{N+1})_0^m \to \K F_{N+1} / (\K F_{N+1})_0^m)\big) \to \lim_{\longleftarrow}{}^1 \, \K F_{N+1} / (\K F_{N+1})_0^{m-1},
        \end{align*}
        where $\lim_{\longleftarrow}{}^1$ is the functor given in \cite[§IX.2.1]{BK72}. \newline
        Since the transition maps of the inverse system $(\K F_{N+1} / (\K F_{N+1})_0^{m-1})_{m \in \N^{\ast}}$ are surjective, this implies that $\displaystyle \lim_{\longleftarrow}{}^1 \, \K F_{N+1} / (\K F_{N+1})_0^{m-1} = 0$ (see, for example, \cite[Propostion IX.2.4]{BK72}). As a consequence,
        \begin{align*}
            &\lim_{\longleftarrow} \mathrm{coker}\big((\K F_{N+1} / (\K F_{N+1})_0^m \to \K F_{N+1} / (\K F_{N+1})_0^m)\big) \simeq \\ &\mathrm{coker}\big(\lim_{\longleftarrow}(\K F_{N+1} / (\K F_{N+1})_0^m \to \lim_{\longleftarrow} \K F_{N+1} / (\K F_{N+1})_0^m)\big).
        \end{align*}
        Thanks to Step 4 and Step 5, this proves the wanted result.
    \end{description}
\end{proof}

\begin{propdef}
    \label{isoMiota}
    Let $\iota \in \Emb(G)$. There exists an unique topological $\K$-module isomorphism \index{$\mathrm{iso}^{\M, \iota}$} $\mathrm{iso}^{\M, \iota} : \widehat{\M}^{\B}_N \to \widehat{\M}^{\DR}_G$ such that the following diagram
    \begin{equation}
        \label{diag_isoM}
        \begin{tikzcd}
            \widehat{\V}^{\B}_N \ar["\mathrm{iso}^{\V, \iota}"]{rrr} \ar["\widehat{- \cdot 1_{\B}}"']{d} &&& \widehat{\V}^{\DR}_G \ar["\widehat{- \cdot 1_{\DR}}"]{d} \\
            \widehat{\M}^{\B}_N \ar["\mathrm{iso}^{\M, \iota}"']{rrr} &&& \widehat{\M}^{\DR}_G \\
        \end{tikzcd}
    \end{equation}
    commutes.
\end{propdef}
\begin{proof}
    Let us construct a topological module morphism $\iso^{\M, \iota} : \widehat{\M}^{\B}_N \to \widehat{\M}^{\DR}_G$ over the topological algebra morphism $\iso^{\V, \iota} : \widehat{\V}^{\B}_N \to \widehat{\V}^{\DR}_G$. We consider the composition
    \begin{equation}
        \label{iso_circ_hat1DR}
        \widehat{\V}^{\B}_N \xrightarrow{\iso^{\V, \iota}} \widehat{\V}^{\DR}_G \xrightarrow{\widehat{-\cdot 1_{\DR}}} \widehat{\M}^{\DR}_G.
    \end{equation}
    This composition sends the $\K$-submodule $\widehat{\V}^{\B}_N (X_0 - 1)$ to $0$. Indeed, this comes from the fact that \eqref{iso_circ_hat1DR} is a module morphism over the algebra morphism $\iso^{\V, \iota}$ and the following computation
    \[
        \mbox{\small $\iso^{\V, \iota}(X_0 - 1) = g_{\iota} \exp\left( \frac{1}{N} e_0 \right) - 1 = g_{\iota} \left(\exp\left( \frac{1}{N} e_0 \right) - 1\right) + (g_{\iota} - 1) \in \widehat{\V}^{\DR} e_0 + \widehat{\V}^{\DR} (g_{\iota} - 1)$}
    \]
    Therefore, thanks to Proposition \ref{hatM_is_quotient}, the composition \eqref{iso_circ_hat1DR} factorises into a $\K$-module morphism $\iso^{\M, \iota} : \widehat{\M}^{\B}_N \to \widehat{\M}^{\DR}_G$ which is a module morphism over the algebra morphism $\iso^{\V, \iota} : \widehat{\V}^{\B}_N \to \widehat{\V}^{\DR}_G$. \newline
    Next, let us show that $\iso^{\M, \iota} : \widehat{\M}^{\B}_N \to \widehat{\M}^{\DR}_G$ is an isomorphism. Recall from Proposition \ref{isoW} that $\iso^{\W, \iota} : \widehat{\W}^{\B}_N \to \widehat{\W}^{\DR}_G$ is an algebra submorphism of $\iso^{\V, \iota} : \widehat{\V}^{\B}_N \to \widehat{\V}^{\DR}_G$. As a result, $\iso^{\M, \iota} : \widehat{\M}^{\B}_N \to \widehat{\M}^{\DR}_G$ is a module morphism over the algebra isomorphism $\iso^{\W, \iota} : \widehat{\W}^{\B}_N \to \widehat{\W}^{\DR}_G$. In addition, $\widehat{\M}^{\B}_N$ and $\widehat{\M}^{\DR}_G$ are both free rank $1$ modules over $\widehat{\W}^{\B}_N$ and $\widehat{\W}^{\DR}_G$ respectively and $\iso^{\M, \iota}$ sends $1_{\B}$ to $1_{\DR}$ and therefore a basis of the source to a basis of the target. Thus $\iso^{\M, \iota} : \widehat{\M}^{\B}_N \to \widehat{\M}^{\DR}_G$ is a module isomorphism over $\iso^{\W, \iota}$ and then a $\K$-module isomorphism.
\end{proof}

\begin{remark}
    Let us notice that we have the following equality of $\K$-submodules of $\widehat{\V}^{\DR}_G$:
    \[
        \widehat{\V}^{\DR}_G (g_{\iota}\exp(e_0) - 1) = \widehat{\V}^{\DR}_G e_0 + \widehat{\V}^{\DR}_G (g_{\iota} - 1).
    \]
    Indeed, since we have that
    \[
        g_{\iota} \exp( e_0 ) - 1 = g_{\iota} (\exp(e_0) - 1) + (g_{\iota} - 1),
    \]
    this gives us the inclusion $\widehat{\V}^{\DR}_G (g_{\iota}\exp(e_0) - 1) \subset \widehat{\V}^{\DR}_G e_0 + \widehat{\V}^{\DR}_G (g_{\iota} - 1)$. Conversely, the inclusion $\widehat{\V}^{\DR}_G e_0 + \widehat{\V}^{\DR}_G (g_{\iota} - 1) \subset \widehat{\V}^{\DR}_G (g_{\iota}\exp(e_0) - 1)$ follows from
    \[
        e_0 = \frac{e_0}{\exp(Ne_0) - 1} \big( 1 + g_{\iota} \exp(e_0) + \cdots + g_{\iota}^{N-1} \exp((N-1)e_0) \big) (g_{\iota} \exp(e_0) - 1)
    \]
    and from
    \begin{align*}
        g_{\iota} - 1 = & \exp(-e_0) (g_{\iota} \exp(e_0) - 1) + (\exp(-e_0) - 1) \\
        = & \exp(-e_0) (g_{\iota} \exp(e_0) - 1) + \\
        & \frac{\exp(-e_0) - 1}{\exp(Ne_0) - 1} \big( 1 + g_{\iota} \exp(e_0) + \cdots + g_{\iota}^{N-1} \exp((N-1)e_0) \big) (g_{\iota} \exp(e_0) - 1) \\
        = & \mbox{\small$\left(\exp(-e_0) + \frac{\exp(-e_0) - 1}{\exp(Ne_0) - 1} \big( 1 + g_{\iota} \exp(e_0) + \cdots + g_{\iota}^{N-1} \exp((N-1)e_0) \big)\right)(g_{\iota} \exp(e_0) - 1)$}.  
    \end{align*}
\end{remark}

\begin{proposition}
    \label{compat_iso_MWV}
    For any $(\lambda, \Psi) \in \K^{\times} \times \G(\KX)$, the following pairs are isomorphisms in the category $\K\text{-}\alg\text{-}\Mod_{\mathrm{top}}$: 
    \begin{enumerate}[label=(\roman*), leftmargin=*]
        \item \label{rel_iso_MV} $\left(\iso^{\V, \iota}, \iso^{\M, \iota}\right) : (\widehat{\V}^{\B}_N, \widehat{\M}^{\B}_N) \to (\widehat{\V}^{\DR}_G, \widehat{\M}^{\DR}_G)$.
        \item \label{rel_iso_MW} $\left(\iso^{\W, \iota}, \iso^{\M, \iota}\right) : (\widehat{\W}^{\B}_N, \widehat{\M}^{\B}_N) \to (\widehat{\W}^{\DR}_G, \widehat{\M}^{\DR}_G)$.
    \end{enumerate}
\end{proposition}

\begin{proof} \ 
    \begin{enumerate}[label=(\roman*), leftmargin=*]
        \item The fact that $\iso^{\V, \iota}$ (resp. $\iso^{\M, \iota}$) is a $\K$-algebra (resp $\K$-module) isomorphism follows from Proposition-Definition \ref{isoViota} (resp. Proposition-Definition \ref{isoMiota}). Let $(a, m) \in \widehat{\V}^{\B}_N \times \widehat{\M}^{\B}_N$. There exists $v \in \widehat{\V}^{\B}_N$ such that $m = v \cdot 1_{\B}$. We have
        \begin{align*}
            \iso^{\M, \iota}(am) = & \iso^{\M, \iota}(av \cdot 1_{\B}) = \iso^{\V, \iota}(av) \cdot 1_{\DR} = \iso^{\V, \iota}(a) \, \iso^{\V, \iota}(v) \cdot 1_{\DR} \\
            = & \iso^{\V, \iota}(a) \, \iso^{\M, \iota}(v \cdot 1_{\B})
            = \iso^{\V, \iota}(a) \, \iso^{\M, \iota}(m),
        \end{align*}
        where the second and fourth equalities come from Proposition-Definition \ref{isoMiota}.
        \item The fact that $\iso^{\W, \iota}$ (resp. $\iso^{\M, \iota}$) is a $\K$-algebra (resp $\K$-module) isomorphism follows from Proposition-Definition \ref{isoW} (resp. Proposition-Definition \ref{isoMiota}). One proves, for any $(w, m) \in \widehat{\W}^{\B}_N \times \widehat{\M}^{\B}_N$, that 
        \[
             \iso^{\M, \iota}(wm) =  \iso^{\W, \iota}(w) \iso^{\M, \iota}(m)
        \]
        using the argument of \ref{rel_iso_MV} and Proposition-Definition \ref{isoW}.
    \end{enumerate}
\end{proof}

\begin{corollary}
    \label{eta_M_phi_iso}
    Let $\iota \in \mathrm{Emb}(G)$ and $\phi \in \Aut(G)$. We have
    \[
        \iso^{\M, \iota \circ \phi^{-1}} = \eta_{\phi}^{\M} \circ \iso^{\M, \iota},
    \]
    with $\eta^{\M}_{\phi} \in \Aut_{\K{\text -}\alg_{\mathrm{top}}}(\widehat{\M}^{\DR}_G)$ given in Lemma \ref{from_etaV}.\ref{etaM}.
\end{corollary}
\begin{proof}
    It follows from Proposition \ref{eta_phi_iso} thanks to the commutativity of diagrams \eqref{diag_isoM} and \eqref{diag_etaV_etaM}.
\end{proof}

\subsection{The coproducts \texorpdfstring{$\widehat{\Delta}^{\W, \B}_N$}{DWBN} and \texorpdfstring{$\widehat{\Delta}^{\M, \B}_N$}{DMBN}} \label{Betti_coprod}

\subsubsection{Comparison isomorphisms}

\begin{definition}
    For $(\iota, \lambda, \Psi) \in \Emb(G) \times \K^{\times} \times \G(\KX)$, we define the topological $\K$-algebra-module isomorphism
    \begin{equation}
        \left(\,^{\Gamma}\comp^{\V, (1)}_{(\iota, \lambda, \Psi)}, \,^{\Gamma}\comp^{\V, (10)}_{(\iota, \lambda, \Psi)}\right) : (\widehat{\V}^{\B}_N, \widehat{\V}^{\B}_N) \to (\widehat{\V}^{\DR}_G, \widehat{\V}^{\DR}_G)
    \end{equation}
    given by
    \[
        \left(\,^{\Gamma}\comp^{\V, (1)}_{(\iota, \lambda, \Psi)}, \,^{\Gamma}\comp^{\V, (10)}_{(\iota, \lambda, \Psi)}\right) := \left(\,^{\Gamma}\aut^{\V, (1)}_{(\lambda, \Psi)}, \,^{\Gamma}\aut^{\V, (10)}_{(\lambda, \Psi)}\right) \circ \left(\mathrm{iso}^{\V, \iota}, \mathrm{iso}^{\V, \iota} \right)
    \]
\end{definition}

\begin{propdef}
    \label{comp_W_Gamma_comp_W}
    For $(\iota, \lambda, \Psi) \in \Emb(G) \times \K^{\times} \times \G(\KX)$, we define the topological $\K$-algebra-module isomorphism
    \begin{equation}
        \left(\,^{\Gamma}\comp^{\W, (1)}_{(\iota, \lambda, \Psi)}, \,^{\Gamma}\comp^{\M, (10)}_{(\iota, \lambda, \Psi)}\right) :  (\widehat{\W}^{\B}_N, \widehat{\W}^{\B}_N) \to (\widehat{\W}^{\DR}_G, \widehat{\W}^{\DR}_G)
    \end{equation}
    given by
    \[
        \left(\,^{\Gamma}\comp^{\W, (1)}_{(\iota, \lambda, \Psi)}, \,^{\Gamma}\comp^{\M, (10)}_{(\iota, \lambda, \Psi)}\right) := \left(\,^{\Gamma}\aut^{\W, (1)}_{(\lambda, \Psi)}, \,^{\Gamma}\aut^{\M, (10)}_{(\lambda, \Psi)}\right) \circ \left(\mathrm{iso}^{\W, \iota}, \mathrm{iso}^{\M, \iota}\right).
    \]
    It is such that the following diagrams
    \begin{equation}
        \label{diag_GammacompW_GammacompV}
        \begin{tikzcd}
            \widehat{\W}^{\B}_N \ar["^{\Gamma}\comp^{\W, (1)}_{(\iota, \lambda, \Psi)}"]{rrr} \ar[hook]{d} &&& \widehat{\W}^{\DR}_N \ar[hook]{d} \\
            \widehat{\V}^{\B}_N \ar["^{\Gamma}\comp^{\V, (1)}_{(\iota, \lambda, \Psi)}"']{rrr} &&& \widehat{\V}^{\DR}_N
        \end{tikzcd}
    \end{equation}
    and
    \begin{equation}
        \label{Diag_Gamma_comp_M}
        \begin{tikzcd}
            \widehat{\V}^{\B}_N \ar["^{\Gamma}\comp^{\V, (10), \iota}_{(\lambda, \Psi)}"]{rrr} \ar["\widehat{- \cdot 1_{\B}}"']{d} &&& \widehat{\V}^{\DR}_G \ar["\widehat{- \cdot 1_{\DR}}"]{d} \\
            \widehat{\M}^{\B}_N \ar["^{\Gamma}\comp^{\M, (10), \iota}_{(\lambda, \Psi)}"']{rrr} &&& \widehat{\M}^{\DR}_G
        \end{tikzcd}
    \end{equation}
    commute.
\end{propdef}
\begin{proof}
    From Proposition-Definition \ref{Gamma_aut_alg_mod_2} and Proposition \ref{compat_iso_MWV}.\ref{rel_iso_MW}, we have that the pairs $\left(\,^{\Gamma}\aut^{\W, (1)}_{(\lambda, \Psi)}, \,^{\Gamma}\aut^{\M, (10)}_{(\lambda, \Psi)}\right)$ and $\left(\mathrm{iso}^{\W, \iota}, \mathrm{iso}^{\W, \iota}\right)$ are isomorphisms in $\K{\text -}\alg{\text -}\Mod_{\mathrm{top}}$; the composition is then an isomorphism in $\K{\text -}\alg{\text -}\Mod_{\mathrm{top}}$. Next, the commutativity of the diagrams follows from the commutativity of Diagrams \eqref{diag_GammaautW2_GammaautV2} and \eqref{diag_isoW} and Diagrams \eqref{diag_GammaautV2_GammaautM2} and \eqref{diag_isoM}.
\end{proof}

\noindent Recall the action of the group $(\Aut(G) \times \K^{\times}) \ltimes \G(\KX)$ on $\Emb(G) \times \K^{\times} \times \G(\KX)$ given in Corollary \ref{EmbGxkxxGkX_torsor}. One has the following result:
\begin{proposition}
    \label{Gamma_comp_aut1}
    For $(\phi, \lambda, \Psi) \in \Aut(G) \times \K^{\times} \times \G(\KX)$ and $(\iota, \nu, \Phi) \in \Emb(G) \times \K^{\times} \times \G(\KX)$, we have
    \begin{enumerate}[label=(\roman*), leftmargin=*]
        \item \label{compWxautW} $^{\Gamma}\comp^{\W, (1)}_{(\phi, \lambda, \Psi) \cdot (\iota, \nu, \Phi)} = \, ^{\Gamma}\aut^{\W, (1)}_{(\phi, \lambda, \Psi)} \circ \,^{\Gamma}\comp^{\W, (1), \iota}_{(\iota, \nu, \Phi)}$.
        \item \label{compMxautM} $^{\Gamma}\comp^{\M, (10)}_{(\phi, \lambda, \Psi) \cdot (\iota, \nu, \Phi)} = \, ^{\Gamma}\aut^{\M, (10)}_{(\phi, \lambda, \Psi)} \circ \,^{\Gamma}\comp^{\M, (10), \iota}_{(\iota, \nu, \Phi)}$.
    \end{enumerate}
\end{proposition}
\begin{proof} \ 
    \begin{enumerate}[label=(\roman*), leftmargin=*]
        \item We have
        \begin{align*}
            ^{\Gamma}\comp^{\W, (1)}_{(\phi, \lambda, \Psi) \cdot (\iota, \nu, \Phi)} & = \,^{\Gamma}\comp^{\W, (1)}_{(\iota \circ \phi^{-1}, \lambda \nu, \Psi \circledast \eta_{\phi}(\lambda \bullet \Phi))} = \, ^{\Gamma}\aut^{\W, (1)}_{(\lambda, \Psi) \circledast (\nu, \eta_{\phi}(\Phi))} \circ \iso^{\W, \iota \circ \phi^{-1}} \\
            & = \,^{\Gamma}\aut^{\W, (1)}_{(\lambda, \Psi)} \circ \, ^{\Gamma}\aut^{\W, (1)}_{(\nu, \eta_{\phi}(\Phi))} \circ \eta_{\phi}^{\W} \circ \iso^{\W, \iota} \\
            & = \,^{\Gamma}\aut^{\W, (1)}_{(\lambda, \Psi)} \circ \eta_{\phi}^{\W} \circ \,^{\Gamma}\aut^{\W, (1)}_{(\nu, \Phi)} \circ \iso^{\W, \iota} \\
            & = \,^{\Gamma}\aut^{\W, (1)}_{(\phi, \lambda, \Psi)} \circ \, ^{\Gamma}\comp^{\W, (1)}_{(\iota, \nu, \Phi)},
        \end{align*}
        where the second equality follows from Lemma \ref{eta_bullet}, the third equality from Corollary \ref{act_gamma_aut_WM} and Corollary \ref{eta_W_phi_iso} and the fourth one from Corollary \ref{Gammaaut_eta_etaWM_Gammaaut}.\ref{Gammaaut_eta_etaW_Gammaaut}.
        \item We have
        \begin{align*}
            ^{\Gamma}\comp^{\M, (10)}_{(\phi, \lambda, \Psi) \cdot (\iota, \nu, \Phi)} & = \,^{\Gamma}\comp^{\M, (10)}_{(\iota \circ \phi^{-1}, \lambda \nu, \Psi \circledast \eta_{\phi}(\lambda \bullet \Phi))} = \, ^{\Gamma}\aut^{\M, (10)}_{(\lambda, \Psi) \circledast (\nu, \eta_{\phi}(\Phi))} \circ \iso^{\M, \iota \circ \phi^{-1}} \\
            & = \,^{\Gamma}\aut^{\M, (10)}_{(\lambda, \Psi)} \circ \, ^{\Gamma}\aut^{\M, (10)}_{(\nu, \eta_{\phi}(\Phi))} \circ \eta_{\phi}^{\M} \circ \iso^{\M, \iota} \\
            & = \,^{\Gamma}\aut^{\M, (10)}_{(\lambda, \Psi)} \circ \eta_{\phi}^{\M} \circ \,^{\Gamma}\aut^{\M, (10)}_{(\nu, \Phi)} \circ \iso^{\M, \iota} \\
            & = \,^{\Gamma}\aut^{\M, (10)}_{(\phi, \lambda, \Psi)} \circ \, ^{\Gamma}\comp^{\M, (10)}_{(\iota, \nu, \Phi)},
        \end{align*}
        where the second equality follows from Lemma \ref{eta_bullet}, the third equality from Corollary \ref{act_gamma_aut_WM} and Corollary \ref{eta_M_phi_iso} and the fourth one from Corollary \ref{Gammaaut_eta_etaWM_Gammaaut}.\ref{Gammaaut_eta_etaM_Gammaaut}.
    \end{enumerate}
\end{proof}

\subsubsection{The coproducts \texorpdfstring{$\widehat{\Delta}^{\W, \B}_{N}$}{DeltaWBN} and \texorpdfstring{$\widehat{\Delta}^{\M, \B}_{N}$}{DeltaMBN}}

\begin{theorem}
    \label{Delta_B_N}
    The composition
    \begin{align*}
        &\mbox{\footnotesize$\left(\left(\left(^{\Gamma}\comp^{\W, (1)}_{(\iota, \nu, \Phi)}\right)^{\otimes 2}\right)^{-1}, \left(\left(^{\Gamma}\comp^{\M, (10)}_{(\iota, \nu, \Phi)}\right)^{\otimes 2}\right)^{-1}\right) \circ \left(\widehat{\Delta}^{\W, \DR}_G, \widehat{\Delta}^{\M, \DR}_G\right) \circ \left(\,^{\Gamma}\comp^{\W, (1)}_{(\iota, \nu, \Phi)}, \,^{\Gamma}\comp^{\M, (10)}_{(\iota, \nu, \Phi)}\right)$} : \\
        &\left(\widehat{\W}^{\B}_N, \widehat{\M}^{\B}_N\right) \to \left((\widehat{\W}^{\B}_N)^{\otimes 2}, (\widehat{\M}^{\B}_N)^{\otimes 2}\right)
    \end{align*}
    is independent of the choice of $(\iota, \nu, \Phi) \in \DMR_{\times}(\K)$. We denote it $(\widehat{\Delta}^{\W, \B}_N, \widehat{\Delta}^{\M, \B}_N)$. Moreover, the pair $(\widehat{\Delta}^{\W, \B}_N, \widehat{\Delta}^{\M, \B}_N)$ is an element of $\mathrm{Cop}_{\K{\text -}\alg{\text -}\Mod_{\mathrm{top}}}\left(\widehat{\W}^{\B}_N, \widehat{\M}^{\B}_N\right)$.
\end{theorem}
\begin{proof}
    Let $(\iota, \nu, \Phi)$ and $(\iota^{\prime}, \nu^{\prime}, \Phi^{\prime}) \in \DMR_{\times}(\K)$. Thanks to Proposition \ref{DMR_x_torsor}, there exists a unique $(\phi, \lambda, \Psi) \in (\Aut(G) \times \K^{\times}) \ltimes \DMR_0^G(\K)$ such that $(\iota^{\prime}, \nu^{\prime}, \Phi^{\prime}) = (\phi, \lambda, \Psi) \cdot (\iota, \nu, \Phi)$. We have
    \begin{align*}
        & \left(\left(^{\Gamma}\comp^{\M, (10)}_{(\iota^{\prime}, \nu^{\prime}, \Phi^{\prime})}\right)^{\otimes 2}\right)^{-1} \circ \widehat{\Delta}^{\M, \DR}_G \circ \,^{\Gamma}\comp^{\M, (10)}_{(\iota^{\prime}, \nu^{\prime}, \Phi^{\prime})} \\
        & = \left(\left(^{\Gamma}\comp^{\M, (10)}_{(\phi, \lambda, \Psi) \cdot (\iota, \nu, \Phi)}\right)^{\otimes 2}\right)^{-1} \circ \widehat{\Delta}^{\M, \DR}_G \circ \,^{\Gamma}\comp^{\M, (10)}_{(\phi, \lambda, \Psi) \cdot (\iota, \nu, \Phi)} \\
        & = \mbox{\small$\left(\left(^{\Gamma}\comp^{\M, (10)}_{(\iota, \nu, \Phi)}\right)^{\otimes 2}\right)^{-1} \circ \left(\left(^{\Gamma}\aut^{\M, (10)}_{(\phi, \lambda, \Psi)}\right)^{\otimes 2}\right)^{-1} \circ \widehat{\Delta}^{\M, \DR}_G \circ \, ^{\Gamma}\aut^{\M, (10)}_{(\phi, \lambda, \Psi)} \circ \,^{\Gamma}\comp^{\M, (10)}_{(\iota, \nu, \Phi)}$} \\
        & = \left(\left(^{\Gamma}\comp^{\M, (10)}_{(\iota, \nu, \Phi)}\right)^{\otimes 2}\right)^{-1} \circ \widehat{\Delta}^{\M, \DR}_G \circ \,^{\Gamma}\comp^{\M, (10)}_{(\iota, \nu, \Phi)},
    \end{align*}
    where the second equality comes from Proposition \ref{Gamma_comp_aut1}.\ref{compMxautM} and the last equality from the inclusion $(\Aut(G) \times \K^{\times}) \ltimes \DMR_0^G(\K) \subset \Stab_{(\Aut(G) \times \K^{\times}) \ltimes \G(\KX)}(\widehat{\Delta}^{\M, \DR}_{G})(\K)$ of Corollary \ref{DMR_sub_StabM_sub_StabW}.
    Similary, we prove that
    \begin{equation*}
        \mbox{\scriptsize$\left(\left(^{\Gamma}\comp^{\W, (1)}_{(\iota^{\prime}, \nu^{\prime}, \Phi^{\prime})}\right)^{\otimes 2}\right)^{-1} \circ \widehat{\Delta}^{\W, \DR}_G \circ \,^{\Gamma}\comp^{\W, (1)}_{(\iota^{\prime}, \nu^{\prime}, \Phi^{\prime})} = \left(\left(^{\Gamma}\comp^{\W, (1)}_{(\iota, \nu, \Phi)}\right)^{\otimes 2}\right)^{-1} \circ \widehat{\Delta}^{\W, \DR}_G \circ \,^{\Gamma}\comp^{\W, (1)}_{(\iota, \nu, \Phi)}$},
    \end{equation*}
    by replacing $\M, (10)$ (resp. $\M, \DR)$ by $\W, (1)$ (resp. $\W, \DR$) in the exponents and the use of Proposition \ref{Gamma_comp_aut1}.\ref{compMxautM} by that of Proposition \ref{Gamma_comp_aut1}.\ref{compWxautW}; and using the the inclusion $(\Aut(G) \times \K^{\times}) \ltimes \DMR_0^G(\K) \subset \Stab_{(\Aut(G) \times \K^{\times}) \ltimes \G(\KX)}(\widehat{\Delta}^{\W, \DR}_{G})(\K)$ of Corollary \ref{DMR_sub_StabM_sub_StabW}. \newline
    Finally, $\left(\widehat{\Delta}^{\W, \B}_N, \widehat{\Delta}^{\M, \B}_N\right)$ is an element of $\mathrm{Cop}_{\K{\text -}\alg{\text -}\Mod_{\mathrm{top}}}\left(\widehat{\W}^{\B}_N, \widehat{\M}^{\B}_N\right)$ since the pair $\left(\widehat{\Delta}^{\W, \DR}_G, \widehat{\Delta}^{\M, \DR}_G\right)$ is an element of $\mathrm{Cop}_{\K{\text -}\alg{\text -}\Mod_{\mathrm{top}}}\left(\widehat{\W}^{\DR}_G, \widehat{\M}^{\DR}_G\right)$ thanks to Lemma \ref{DeltaW_DelaM_DR_Cop} and the pair $\left(\,^{\Gamma}\comp^{\W, (1)}_{(\iota, \lambda, \Psi)}, \,^{\Gamma}\comp^{\M, (10)}_{(\iota, \lambda, \Psi)}\right)$ is a $\K$-algebra-module isomorphism thanks to Proposition-Definition \ref{comp_W_Gamma_comp_W}. 
\end{proof}

\begin{corollary}
    We have $\widehat{\Delta}^{\M, \B}_N(1_{\B}) = 1_{\B}^{\otimes 2}$.
\end{corollary}
\begin{proof}
    From Theorem \ref{Delta_B_N}, let us compute $\widehat{\Delta}^{\M, \B}_N(1_{\B})$ by considering an element $(\iota, \lambda, \Psi) \in \DMR_{\times}(\K)$. First, we have
    \begin{equation}
        \label{comp_1B}
        ^{\Gamma}\comp^{\M, (10)}_{(\iota, \lambda, \Psi)}(1_{\B}) = \,^{\Gamma}\comp^{\V, (10)}_{(\iota, \lambda, \Psi)}(1) \cdot 1_{\DR} = \Gamma_{\Psi}^{-1}(-e_1) \beta(\Psi \otimes 1) \cdot 1_{\DR} = \Psi^{\star}. 
    \end{equation}
    Therefore,
    \begin{align*}
        \widehat{\Delta}^{\M, \B}_N(1_{\B}) & = \left(\left(^{\Gamma}\comp^{\M, (10)}_{(\iota, \lambda, \Psi)}\right)^{\otimes 2}\right)^{-1} \circ \widehat{\Delta}^{\M, \DR}_G \circ \, ^{\Gamma}\comp^{\M, (10)}_{(\iota, \lambda, \Psi)}(1_{\B}) \\
        & = \left(\left(^{\Gamma}\comp^{\M, (10)}_{(\iota, \lambda, \Psi)}\right)^{\otimes 2}\right)^{-1} \circ \widehat{\Delta}^{\M, \DR}_G \left(\Psi^{\star}\right) \\
        & = \left(\left(^{\Gamma}\comp^{\M, (10)}_{(\iota, \lambda, \Psi)}\right)^{\otimes 2}\right)^{-1} (\Psi^{\star} \otimes \Psi^{\star}) = 1_{\B}^{\otimes 2},
    \end{align*}
    where the first and last equalities come from \eqref{comp_1B} and the second one from the fact that $\Psi \in \DMR_{\lambda}^{\iota}(\K)$.
\end{proof}

\begin{corollary} \ 
    \begin{enumerate}[label=(\roman*), leftmargin=*]
        \item \label{betti_Hopf} The pair $(\widehat{\W}^{\B}_N, \widehat{\Delta}^{\W, \B}_N)$ is an object in the category $\K{\text -}\mathrm{Hopf}_{\mathrm{top}}$.
        \item \label{betti_coalg} The pair $(\widehat{\M}^{\B}_N, \widehat{\Delta}^{\M, \B}_N)$ is an object in the category $\K{\text -}\mathrm{coalg}_{\mathrm{top}}$.
        \item The pair $\left((\widehat{\W}^{\B}_N, \widehat{\Delta}^{\W, \B}_N), (\widehat{\M}^{\B}_N, \widehat{\Delta}^{\M, \B}_N)\right)$ is an object in the category $\K{\text -}\mathrm{HAMC}_{\mathrm{top}}$.
    \end{enumerate}
    \label{WM_categories}
\end{corollary}
\begin{proof} \ 
    \begin{enumerate}[label=(\roman*), leftmargin=*]
        \item From Theorem \ref{Delta_B_N}, it follows that $\widehat{\Delta}^{\W, \B}_N$ is an algebra morphism. In addition, one checks that the coassociativity of $\widehat{\Delta}^{\W, \B}_N$ follows from the coassociativity of $\widehat{\Delta}^{\W, \DR}_G$.
        \item From Theorem \ref{Delta_B_N}, it follows that $\widehat{\Delta}^{\M, \B}_N$ is a $\K$-module morphism. In addition, one checks that the coassociativity of $\widehat{\Delta}^{\M, \B}_N$ follows from the coassociativity of $\widehat{\Delta}^{\M, \DR}_G$.
        \item It follows from \ref{betti_Hopf} and \ref{betti_coalg} and the fact that the pair $(\widehat{\Delta}^{\W, \B}_N, \widehat{\Delta}^{\M, \B}_N)$ is an element of $\mathrm{Cop}_{\K{\text -}\alg{\text -}\Mod_{\mathrm{top}}}\left(\widehat{\W}^{\B}_N, \widehat{\M}^{\B}_N\right)$.
    \end{enumerate}
\end{proof}
    \section{Expression of the torsor \texorpdfstring{$\DMR_{\times}(\K)$}{DMRx} in terms of the Betti and de Rham coproducts} \label{torsorBDR}
In this section, we show that $\DMR_{\times}(\K)$ is a subtorsor of a stabilizer torsor of the pair of coproducts $\left(\widehat{\Delta}^{\M, \B}_{N}, \widehat{\Delta}^{\M, \DR}_{G}\right)$. In §\ref{subtorsor_stabs}, we define the setwise stabilizers $\mbox{\small$\Stab_{\Emb(G) \times \K^{\times} \times \G(\KX)}\left(\widehat{\Delta}^{\M, \B}_{N}, \widehat{\Delta}^{\M, \DR}_{G}\right)(\K)$}$ and $\mbox{\small$\Stab_{\Emb(G) \times \K^{\times} \times \G(\KX)}\left(\widehat{\Delta}^{\W, \B}_{N}, \widehat{\Delta}^{\W, \DR}_{G}\right)(\K)$}$ and show that they are equipped with a torsor structure for the actions of the stabilizer groups $\mbox{\small$\Stab_{(\Aut(G) \times \K^{\times}) \ltimes \G(\KX)}\left(\widehat{\Delta}^{\M, \DR}_{G}\right)(\K)$}$ and $\mbox{\normalsize$\Stab_{(\Aut(G) \times \K^{\times}) \ltimes \G(\KX)}\left(\widehat{\Delta}^{\W, \DR}_{G}\right)(\K)$}$ respectively. In §\ref{inclusion_torsor_stabs}, we obtain a chain of inclusions of torsors involving these stabilizers and $\DMR_{\times}(\K)$.

\subsection{The stabilizer subtorsors} \label{subtorsor_stabs}
\begin{definition} \ 
    \begin{enumerate}[label=(\roman*), leftmargin=*]
        \item We denote $\Stab_{\Emb(G) \times \K^{\times} \times \G(\KX)}\left(\widehat{\Delta}^{\W, \B}_{N}, \widehat{\Delta}^{\W, \DR}_{G}\right)(\K)$ the setwise stabilizer of the pair of coproducts $\left(\widehat{\Delta}^{\W, \B}_{N}, \widehat{\Delta}^{\W, \DR}_{G}\right)  \in \mathrm{Cop}_{\K{\text -}\alg_{\mathrm{top}}}(\widehat{\W}^{\B}_N) \times \mathrm{Cop}_{\K{\text -}\alg_{\mathrm{top}}}(\widehat{\W}^{\DR}_G)$ given by
        \begin{align*}
            & \Stab_{\Emb(G) \times \K^{\times} \times \G(\KX)}\left(\widehat{\Delta}^{\W, \B}_{N}, \widehat{\Delta}^{\W, \DR}_{G}\right)(\K) := \\
            & \mbox{\footnotesize$\left\{ (\iota, \nu, \Phi) \in \Emb(G) \times \K^{\times} \times \G(\KX) \, | \, \left( {^{\Gamma}\comp^{\W, (1)}_{(\iota, \nu, \Phi)}} \right)^{\otimes 2} \circ \widehat{\Delta}^{\W, \B}_{N} = \widehat{\Delta}^{\W, \DR}_{G} \circ \,^{\Gamma}\comp^{\W, (1)}_{(\iota, \nu, \Phi)} \right\}$}.
        \end{align*}
        \item We denote $\Stab_{\Emb(G) \times \K^{\times} \times \G(\KX)}\left(\widehat{\Delta}^{\M, \B}_{N}, \widehat{\Delta}^{\M, \DR}_{G}\right)(\K)$ the setwise stabilizer of the pair of coproducts $\left(\widehat{\Delta}^{\M, \B}_{N}, \widehat{\Delta}^{\M, \DR}_{G}\right)  \in \mathrm{Cop}_{\K{\text -}\Mod_{\mathrm{top}}}(\widehat{\M}^{\B}_N) \times \mathrm{Cop}_{\K{\text -}\Mod_{\mathrm{top}}}(\widehat{\W}^{\DR}_G)$ given by
        \begin{align*}
            & \Stab_{\Emb(G) \times \K^{\times} \times \G(\KX)}\left(\widehat{\Delta}^{\M, \B}_{N}, \widehat{\Delta}^{\M, \DR}_{G}\right)(\K) := \\
            & \mbox{\footnotesize$\left\{ (\iota, \nu, \Phi) \in \Emb(G) \times \K^{\times} \times \G(\KX) \, | \, \left( {^{\Gamma}\comp^{\M, (10)}_{(\iota, \nu, \Phi)}} \right)^{\otimes 2} \circ \widehat{\Delta}^{\M, \B}_{N} = \widehat{\Delta}^{\M, \DR}_{G} \circ \,^{\Gamma}\comp^{\M, (10)}_{(\iota, \nu, \Phi)} \right\}$}.
        \end{align*}
    \end{enumerate}
\end{definition}

\begin{remark}
    \label{nonempty_stabs}
    Theorem \ref{Delta_B_N} implies that $\Stab_{\Emb(G) \times \K^{\times} \times \G(\KX)}\left(\widehat{\Delta}^{\W, \B}_{N}, \widehat{\Delta}^{\W, \DR}_{G}\right)(\K)$ and $\Stab_{\Emb(G) \times \K^{\times} \times \G(\KX)}\left(\widehat{\Delta}^{\M, \B}_{N}, \widehat{\Delta}^{\M, \DR}_{G}\right)(\K)$ contain $\DMR_{\times}(\K)$, which implies that these are nonempty sets.
\end{remark}

\begin{proposition} \ 
    \begin{enumerate}[label=\roman*., leftmargin=*]
        \item \label{StabW_subtor} The pair 
        \[
            \mbox{\small$\left(\Stab_{(\Aut(G) \times \K^{\times}) \ltimes \G(\KX)}\left(\widehat{\Delta}^{\W, \DR}_{G}\right)(\K), \Stab_{\Emb(G) \times \K^{\times} \times \G(\KX)}\left(\widehat{\Delta}^{\W, \B}_{N}, \widehat{\Delta}^{\W, \DR}_{G}\right)(\K)\right)$}
        \]
        is a subtorsor of $\Big((\Aut(G) \times \K^{\times}) \ltimes \G(\KX), \Emb(G) \times \K^{\times} \times \G(\KX)\Big)$.
        \item The pair 
        \[
            \mbox{\small$\left(\Stab_{(\Aut(G) \times \K^{\times}) \ltimes \G(\KX)}\left(\widehat{\Delta}^{\M, \DR}_{G}\right)(\K), \Stab_{\Emb(G) \times \K^{\times} \times \G(\KX)}\left(\widehat{\Delta}^{\M, \B}_{N}, \widehat{\Delta}^{\M, \DR}_{G}\right)(\K)\right)$}
        \]
        is a subtorsor of $\Big((\Aut(G) \times \K^{\times}) \ltimes \G(\KX), \Emb(G) \times \K^{\times} \times \G(\KX)\Big)$.
    \end{enumerate}
    \label{Stabs_subtor}
\end{proposition}
\noindent In order the prove this, we will need the following Lemma:
\begin{lemma}[{\cite[Lemma 2.6]{EF2}}]
    \label{Genera_stab_subtor}
    Let $(H, T)$ be a torsor, and let $V, V^{\prime}$ be $\K$-modules. Let $\rho : H \to \Aut_{\K{\text -}\Mod}(V)$ be a group morphism and let $\rho^{\prime} : T \to \mathrm{Iso}_{\K{\text -}\Mod}(V^{\prime}, V)$ be a map such that for any $h \in H$, $x \in T$, one has $\rho^{\prime}(h \cdot x) = \rho(h) \circ \rho^{\prime}(x)$. Let $v \in V$ and $v^{\prime} \in V^{\prime}$. Then $\mathrm{Stab}_H(v) := \{h \in H \, | \, \rho(h)(v) = v \}$ is a subgroup of $H$, and either $\mathrm{Stab}_T(v, v') := \{ x \in T \, | \, \rho^{\prime}(v^{\prime}) = v \}$ is empty, or $(\mathrm{Stab}_H(v), \mathrm{Stab}_T(v, v'))$ is a subtorsor of $(H, T)$.
\end{lemma}
\begin{proof}[Proof of Proposition \ref{Stabs_subtor}]
    It follows from Lemma \ref{Genera_stab_subtor} by setting :
    \begin{itemize}[leftmargin=*]
        \setlength\itemsep{1em}
        \item $(H, T) = \left((\Aut(G) \times \K^{\times}) \ltimes \G(\KX), \Emb(G) \times \K^{\times} \times \G(\KX) \right)$;
        \item $V = \mathrm{Cop}_{\K{\text -}\Mod}(\widehat{\W}^{\DR}_G)$ (resp. $V = \mathrm{Cop}_{\K{\text -}\Mod}(\widehat{\M}^{\DR}_G)$);
        \item $V^{\prime} = \mathrm{Cop}_{\K{\text -}\Mod}(\widehat{\W}^{\B}_N)$ (resp. $V^{\prime} = \mathrm{Cop}_{\K{\text -}\Mod}(\widehat{\M}^{\B}_N)$);
        \item $v = \widehat{\Delta}^{\W, \DR}_G$ (resp. $v = \widehat{\Delta}^{\M, \DR}_G$);
        \item $v^{\prime} = \widehat{\Delta}^{\W, \B}_N$ (resp. $v^{\prime} = \widehat{\Delta}^{\M, \B}_N$);
        \item $\rho : (\phi, \lambda, \Psi) \mapsto \left( V \ni D^{\W}_{\DR} \mapsto \left(^{\Gamma}\aut^{\W, (1)}_{(\phi, \lambda, \Psi)}\right)^{\otimes 2} \circ D^{\W}_{\DR} \circ \left(\,^{\Gamma}\aut^{\W, (1)}_{(\phi, \lambda, \Psi)}\right)^{-1} \in V \right)$ (resp. $\rho : (\phi, \lambda, \Psi) \mapsto \left( V \ni D^{\M}_{\DR} \mapsto \left(^{\Gamma}\aut^{\M, (10)}_{(\phi, \lambda, \Psi)}\right)^{\otimes 2} \circ D^{\M}_{\DR} \circ \left(\,^{\Gamma}\aut^{\M, (10)}_{(\phi, \lambda, \Psi)}\right)^{-1} \in V \right)$);
        \item $\rho^{\prime} : (\phi, \lambda, \Psi) \mapsto \left( V^{\prime} \ni D^{\W}_{\B} \mapsto \left(^{\Gamma}\comp^{\W, (1)}_{(\phi, \lambda, \Psi)}\right)^{\otimes 2} \circ D^{\W}_{\B} \circ \left(\,^{\Gamma}\comp^{\W, (1)}_{(\phi, \lambda, \Psi)}\right)^{-1} \in V \right)$ (resp. $\rho^{\prime} : (\phi, \lambda, \Psi) \mapsto \left( V^{\prime} \ni D^{\M}_{\B} \mapsto \left(^{\Gamma}\comp^{\M, (10)}_{(\phi, \lambda, \Psi)}\right)^{\otimes 2} \circ D^{\M}_{\B} \circ \left(\,^{\Gamma}\comp^{\M, (10)}_{(\phi, \lambda, \Psi)}\right)^{-1} \in V \right)$).
    \end{itemize}
    Finally, for $(\phi, \lambda, \Psi) \in H$ and $(\iota, \nu, \Phi) \in T$, the identity
    \[
        \rho((\phi, \lambda, \Psi) \cdot (\iota, \nu, \Phi)) = \rho(\phi, \lambda, \Psi) \circ \rho^{\prime}(\iota, \nu, \Phi)
    \]
    follows from Proposition \ref{Gamma_comp_aut1}.
\end{proof}

\subsection{Inclusion of stabilizer torsors}
\begin{theorem}
    \label{inclusion_torsor_stabs}
    We have the following inclusions of torsors
    \begin{eqnarray*}
        & \left((\Aut(G) \times \K^{\times}) \ltimes \DMR_0^G(\K), \DMR_{\times}(\K)\right) & \\
        & \cap & \\
        & \left(\Stab_{(\Aut(G) \times \K^{\times}) \ltimes \G(\KX)}\left(\widehat{\Delta}^{\M, \DR}_{G}\right)(\K), \Stab_{\Emb(G) \times \K^{\times} \times \G(\KX)}\left(\widehat{\Delta}^{\M, \B}_{N}, \widehat{\Delta}^{\M, \DR}_{G}\right)(\K)\right) & \\
        & \cap & \\
        & \left(\Stab_{(\Aut(G) \times \K^{\times}) \ltimes \G(\KX)}\left(\widehat{\Delta}^{\W, \DR}_{G}\right)(\K), \Stab_{\Emb(G) \times \K^{\times} \times \G(\KX)}\left(\widehat{\Delta}^{\W, \B}_{N}, \widehat{\Delta}^{\W, \DR}_{G}\right)(\K)\right) & \\
        & \cap & \\
        & \Big((\Aut(G) \times \K^{\times}) \ltimes \G(\KX), \Emb(G) \times \K^{\times} \times \G(\KX)\Big) &
    \end{eqnarray*}
\end{theorem}
\noindent In order the prove this, we will need the following Lemmas:
\begin{lemma}[{\cite[Lemma 2.3]{EF2}}]
    \label{inter_subtor}
    Let $(H, T)$ be a torsor and let $(H^{\prime}, T^{\prime})$ and $(H^{\prime\prime}, T^{\prime\prime})$ be subtorsors of $(H, T)$ such that $T^{\prime} \cap T^{\prime\prime} \neq \varnothing$. Then $(H^{\prime} \cap H^{\prime\prime}, T^{\prime} \cap T^{\prime\prime})$ is a subtorsor of both $(H^{\prime}, T^{\prime})$ and $(H^{\prime\prime}, T^{\prime\prime})$, therefore of $(H, T)$. 
\end{lemma}
\begin{lemma}[{\cite[Lemma 2.7]{EF2}}]
    \label{incl_subtor}
    Let $(H, T)$ be a torsor and let $(H_0, T_0)$ and $(H_1, T_1)$ be subtorsors of $(H, T)$ such that $T_0 \subset T_1$. Then $(H_0, T_0)$ is a subtorsor of $(H_1, T_1)$. If, moreover, $H_0 = H_1$ then the subtorsors $(H_0, T_0)$ and $(H_1, T_1)$ are equal.
\end{lemma}
\begin{proof}[Proof of Theorem \ref{inclusion_torsor_stabs}]
    The group-part inclusion is shown in Corollary \ref{DMR_sub_StabM_sub_StabW}. The first and last set-part inclusions are immediate. It remains to show that
    \[
        \mbox{\small$\Stab_{\Emb(G) \times \K^{\times} \times \G(\KX)}\left(\widehat{\Delta}^{\M, \B}_{N}, \widehat{\Delta}^{\M, \DR}_{G}\right)(\K) \subset \Stab_{\Emb(G) \times \K^{\times} \times \G(\KX)}\left(\widehat{\Delta}^{\W, \B}_{N}, \widehat{\Delta}^{\W, \DR}_{G}\right)(\K)$}.
    \]
    In Lemmas \ref{inter_subtor} and \ref{incl_subtor}, set
    \[
        (H, T) = \left((\Aut(G) \times \K^{\times}) \ltimes \G(\KX), \Emb(G) \times \K^{\times} \times \G(\KX) \right).
    \]
    First, let us apply Lemma \ref{inter_subtor} for
    \begin{itemize}[leftmargin=*]
        \setlength\itemsep{1em}
        \item $(H^{\prime}, T^{\prime}) = \mbox{\scriptsize$\left(\Stab_{(\Aut(G) \times \K^{\times}) \ltimes \G(\KX)}\left(\widehat{\Delta}^{\M, \DR}_{G}\right)(\K), \Stab_{\Emb(G) \times \K^{\times} \times \G(\KX)}\left(\widehat{\Delta}^{\M, \B}_{N}, \widehat{\Delta}^{\M, \DR}_{G}\right)(\K)\right)$}$;
        \item $(H^{\prime\prime}, T^{\prime\prime}) = \mbox{\scriptsize$\left(\Stab_{(\Aut(G) \times \K^{\times}) \ltimes \G(\KX)}\left(\widehat{\Delta}^{\W, \DR}_{G}\right)(\K), \Stab_{\Emb(G) \times \K^{\times} \times \G(\KX)}\left(\widehat{\Delta}^{\W, \B}_{N}, \widehat{\Delta}^{\W, \DR}_{G}\right)(\K)\right)$}$.
    \end{itemize}
    From Remark \ref{nonempty_stabs}, we have that $T^{\prime} \cap T^{\prime\prime} \neq \varnothing$. Therefore, $(H^{\prime} \cap H^{\prime\prime}, T^{\prime} \cap T^{\prime\prime})$ is a subtorsor of $(H^{\prime\prime}, T^{\prime\prime})$. Second, let us apply Lemma \ref{incl_subtor} for
    \begin{itemize}[leftmargin=*]
        \setlength\itemsep{0,5em}
        \item $(H_0, T_0) = (H^{\prime} \cap H^{\prime\prime}, T^{\prime} \cap T^{\prime\prime})$;
        \item $(H_1, T_1) = (H^{\prime}, T^{\prime})$.
    \end{itemize}
    We have that $T_0 = T^{\prime} \cap T^{\prime\prime} \subset T^{\prime} = T_1$. In addition,
    \[
        H_0 = H^{\prime} \cap H^{\prime\prime} = H^{\prime} = H_1,
    \]
    where the second equality follows from the stabilizer group inclusion in Corollary \ref{DMR_sub_StabM_sub_StabW}.
    Finally, it follows that $T^{\prime} \cap T^{\prime\prime} = T_0 = T_1 = T^{\prime}$. Thus $T^{\prime} \subset T^{\prime\prime}$, which is the wanted inclusion of setwise stabilizers.
\end{proof}
    \nocite{Enr20} \nocite{Fu12} \nocite{Kas} 
    \bibliographystyle{abstract}
    \bibliography{main}
\end{document}